\documentclass[12pt,reqno]{article}

\usepackage[margin=1in]{geometry}
\usepackage{amsmath,amsthm,amsfonts,amssymb}
\usepackage{mathrsfs}
\usepackage{cases}
\usepackage{url}
\usepackage{enumitem}
\usepackage{tikz}
\usepackage{microtype}

\usepackage[square,comma,sort&compress,numbers]{natbib}
\usetikzlibrary{arrows.meta,fit}
\usepackage[colorlinks=true,
linkcolor=blue,
citecolor=blue,
urlcolor=blue]{hyperref}

\newtheorem{thm}{Theorem}[section]
\newtheorem{lem}[thm]{Lemma}
\newtheorem{cor}[thm]{Corollary}
\newtheorem{prop}[thm]{Proposition}
\newtheorem{conj}[thm]{Conjecture}

\theoremstyle{definition}

\newtheorem{exm}[thm]{Example}
\newtheorem{rem}[thm]{Remark}

\numberwithin{equation}{section}

\newcommand{\floor}[1]{\left\lfloor #1 \right\rfloor}
\newcommand{\defi}[1]{\textit{#1}}

\newcommand{\cB}{\mathcal{B}}

\newcommand{\cN}{\mathcal{N}}

\DeclareMathOperator{\Ant}{Ant}

\title{Antichain polynomials of products of chains and minuscule posets}
\author{%
	Xi Chen,
	Yuxian Dong,
	Sihao Tang\thanks{Corresponding author.\\
		\textit{Email addresses:}
		\href{mailto:chenxi@dlut.edu.cn}{\nolinkurl{chenxi@dlut.edu.cn}} (X. Chen),
		\href{mailto:yuxiand@hotmail.com}{\nolinkurl{yuxiand@hotmail.com}} (Y. Dong), and
		\href{mailto:sihao_tang@hotmail.com}{\nolinkurl{sihao_tang@hotmail.com}} (S. Tang).}}
\date{\small School of Mathematical Sciences, Dalian University of Technology, Dalian 116024, PR China}
\begin{document}
\maketitle

\begin{abstract}
This paper studies the antichain polynomials of $[k]\times P$, where $P$ is a connected minuscule poset.
We give a formula for the number of antichains, counted by size, of an arbitrary poset.
Using this formula, we present necessary and sufficient conditions for the palindromicity of antichain polynomials for two infinite families of connected minuscule posets.
We show that, for every connected minuscule poset $P$,
if the antichain polynomial of $[k]\times P$ is palindromic, then it has only real and strictly negative zeros.
This result, in particular, gives an affirmative answer to Ding-Dong's conjecture about $\gamma$-positivity of the antichain polynomial of $[k]\times P$.
By constructing a bijection between antichains and labeled Dyck paths, we also give a layer-refined enumeration of the antichains in $[2]\times[m]\times[n]$.
We establish a relation between such antichains and Clar covers of the hexagonal flakes $O(2,m,n)$,
thus prove a conjectured determinantal formula for a family of Zhang-Zhang polynomials.
Real-rootedness and stability results are also obtained when the shortest chain has length at most two.
Finally, we present infinitely many connected Peck posets whose antichain polynomials are not unimodal,
disproving the log-concavity conjecture of Ding and Dong.
\\[1pt]
\end{abstract}
\tableofcontents

\section{Introduction}
% The Cartesian product of posets is ordered componentwise: $(x_1,\ldots,x_r)\le(y_1,\ldots,y_r)$ if and only if $x_i\le y_i$ for every $i\in[r]$.
Let $P$ be a finite poset.
An \defi{antichain} of $P$ is a subset whose elements are pairwise incomparable.
We write $\Ant(P)$ for the set of antichains of $P$ and $A(P)=|\Ant(P)|$.
%  for their total number.
The \defi{width} of $P$, denoted by $w(P)$, is the maximum cardinality of an antichain in $P$.
The \defi{antichain polynomial} of $P$ is defined as
\[
\mathcal N_P(x):=
\sum_{\mathcal A\in\Ant(P)}x^{|\mathcal A|}
=
\sum_{j=0}^{w(P)}a_jx^j,
\]
where $a_j$ is the number of antichains with $j$ elements.
It is clear that $\mathcal N_P(1)=A(P)$.
%  and $\deg\mathcal N_P(x)=w(P)$.

The enumeration of antichains is a classical problem in combinatorics.
Dedekind's problem asks for the number of antichains in the Boolean lattice $\cB_n$~\cite{Dedekind1897}.
Sperner's theorem~\cite{Sperner1928} shows that $w(\cB_n)=\binom{n}{\floor{n/2}}$,
hence $A(\cB_n)\ge 2^{\binom{n}{\floor{n/2}}}$ as every subfamily of a middle layer forms an antichain.
Successively sharper asymptotic estimates for $A(\mathcal B_n)$ were obtained in~\cite{Kleitman1969,KleitmanMarkowsky1975,Korshunov1981,Sapozhenko1991}; see also~\cite{Kahn2002,BaloghTreglownWagner2016,NoelScottSudakov2018} for approaches based on entropy, supersaturation, and container methods.
Very recently, Jenssen et al.~\cite{JenssenMalekshahianPark2026} studied the number of antichains of a prescribed cardinality in $\mathcal B_n$.
It is well known that $\cB_n$ can be viewed as the a product of two-element chains.
Therefore Dedekind's problem naturally lead to the enumeration of antichains in product of chains.
See~\cite{Tsai2019,PohoataZakharov2024,ParkSarantisTetali2025,FalgasRavryRatyTomon2026} for some recent work.

% In this paper, we study antichains according to their cardinalities.

Products of chains lead naturally to the theory of minuscule posets, 
which originates in the representation theory of simple Lie algebras. 
The weight poset of a minuscule representation is a finite distributive lattice and hence can be realized as the lattice of order ideals of a finite poset.
For a finite poset $P$, let $J(P)$ be the poset of its order ideals, ordered by inclusion, and let $J^r(P)$ denote the $r$-fold iteration of this construction.
Following Proctor~\cite[Section~4]{Proctor1984},
say that $P$ is \defi{minuscule} if $J(P)$ is isomorphic to the weight poset of a minuscule representation. 
We say that $P$ is \defi{connected} if its Hasse diagram is connected.
By Proctor's classification~\cite[Proposition~4.2]{Proctor1984}, the connected minuscule posets, up to isomorphism, are
\[
[m]\times[n],
\qquad
H_n=\{(i,j)\in[n]\times[n]:i\le j\},
\qquad
K_n=[n]\oplus([1]\sqcup[1])\oplus[n],
\]
together with two exceptional posets
\[
J^2([2]\times[3])
\qquad\text{and}\qquad
J^3([2]\times[3]).
\]
Here $\oplus$ and $\sqcup$ denote the ordinal sum and disjoint union, respectively.

In this paper, we study combinatorial properties of the antichain polynomials of connected minuscule posets,
including $\gamma$-positivity and real-rootedness.
% Let $f(x)=a_s x^s+\cdots+a_t x^t$ be a polynomial with nonnegative coefficients,
% where $s\le t$ and $a_s a_t\neq0$.
% Following Zeilberg \cite{Zei89}, we call $s+t$ the \defi{darga} of $f(x)$.
Let $f(x)=\sum_{j=0}^d a_jx^j$ be a polynomial of with positive coefficients.
The polynomial $f(x)$ is \defi{palindromic} if $a_j=a_{d-j}$ for every $0\le j\le d$, and is \defi{real-rooted} if all its zeros are real numbers.
Its coefficient sequence $(a_0,a_1,\ldots,a_d)$ is \defi{unimodal} if there exists an index $k$ such that $a_0\le\cdots\le a_k\ge\cdots\ge a_d$, and is \defi{log-concave} if $a_j^2\ge a_{j-1}a_{j+1}$ for all $1\le j<d$.
A palindromic polynomial $f(x)$ is called \defi{$\gamma$-positive} if it can be written as
\[
f(x)
=
\sum_{j=0}^{\lfloor d/2\rfloor}
\gamma_j x^j(1+x)^{d-2j}
\]
with $\gamma_j\ge0$ for all $j$.
Newton's inequalities shows that if a polynomial with positive coefficients is real-rooted then its coefficient sequence is log-concave, hence unimodal. 
% A nonnegative log-concave sequence with no internal zeros is unimodal; see~\cite{Brenti1989,Stanley1989}.
It is also known that if a palindromic polynomial with nonnegative coefficients is real-rooted then it is $\gamma$-positive; see~\cite{Branden2015} and~\cite[P. 82]{Petersen2015}.

Ding and Dong~\cite{DingDong2019} proposed the following conjecture.

\begin{conj}[{\cite[Conjecture B]{DingDong2019}}]\label{conj:DD-minuscule}
	Let $P$ be a connected minuscule poset and let $k$ be a positive integer.
	If $\mathcal N_{[k]\times P}(x)$ is palindromic, then it is $\gamma$-positive.
\end{conj}

% Ding and Dong verified Conjecture~\ref{conj:DD-minuscule} when $k=1$, for the family $K_n$, and for the two exceptional minuscule posets.
% We settle both families and obtain the following stronger conclusion.

We settle this conjecture and obtain the following stronger result.

\begin{thm}\label{thm:main-minuscule}
	Let $P$ be a connected minuscule poset and let $k$ be a positive integer.
	If $\mathcal N_{[k]\times P}(x)$ is palindromic, then all its zeros are real and negative.
\end{thm}

% Our palindromicity classifications use a common coefficient criterion.
% Let $L$ be the unique maximum antichain of a finite poset $P$, and set $d=|L|$. 
% For $B\in\Ant(P\setminus L)$, define $\delta_L(B)=|\Gamma_L(B)|-|B|$, where $\Gamma_L(B)$ is the set of elements of $L$ comparable with an element of $B$. 
% Let $D_1(P,L)$ denote the number of such antichains with $\delta_L(B)=1$.
% By Corollary~\ref{cor:defect-one-coefficient},
% \[
% [x^{d-1}]\mathcal N_P(x)=d+D_1(P,L).
% \]
% Since $[x]\mathcal N_P(x)=|P|$, palindromicity requires $D_1(P,L)=|P|-d$.
% This necessary condition is the main tool in our palindromicity classifications.

Ding and Dong~\cite{DingDong2019} already determined all palindromic cases for the last three families, namely $K_n$, $J^2([2]\times[3])$, and $J^3([2]\times[3])$; see also Theorem~\ref{thm:DD-remaining-palindromicity}. Therefore, to complete the palindromicity classification for $[k]\times P$, it remains only to consider the first two families: $[m]\times[n]$ and $H_n$. 

For both cases, we present necessary and sufficient conditions for the palindromicity of their antichain polynomials; 
see Theorems~\ref{thm:three-chain-palindromicity-classification} and~\ref{thm:H-palindromicity-classification}.
In particular,
\[
\mathcal N_{[k]\times[m]\times[n]}(x)
\text{ is palindromic}
\quad\Longleftrightarrow\quad
\begin{cases}
	k=1 \text{ and } n=m, \quad\text{or}\\
	k=2 \text{ and } n=m+1.
\end{cases}
\]
% where $1\le k\le m\le n$.
When $k=1$ and $m=n$, we obtain the type $B$ Narayana polynomial
\[
\mathcal N_{[1]\times[m]\times[m]}(x)
=
\sum_{j=0}^m\binom mj^2x^j,
\]
which is real-rooted by Br\"and\'en~\cite[Theorem~7.1]{Branden2006}.
% is the type $B$ Narayana polynomial~\cite{Reiner1997} and is real-rooted.
When $k=2$, Ding and Dong~\cite{DingDong2019} proposed the following conjectures.
% For $k=2$, we construct a bijection between antichains with prescribed
% layer sizes and labeled Dyck paths; see~Section~\ref{subsec:k-at-most-two}. 
% This gives a closed formula for the refined enumeration, from which we prove that $\mathcal N_{[2]\times[m]\times[n]}(x)$ has only simple negative zeros.
% An Euler-operator reciprocity then proves its palindromicity when $n=m+1$, completing the proof of Conjecture~\ref{conj:DD-three-chains}.

% We first apply this criterion to products of three chains, for which Ding and Dong~\cite{DingDong2019} proposed the following two conjectures.

\begin{conj}[{\cite[Conjectures~4.3 and~4.5]{DingDong2019}}]\label{conj:DD-three-chains}
	The following statements hold.
	\begin{enumerate}[label=\textup{(\roman*)}]
		\item
		For all positive integers $m$ and $n$, $\mathcal N_{[2]\times[m]\times[n]}(x)$ is real-rooted.
		\item
		For every positive integer $n$, $\mathcal N_{[2]\times[n]\times[n+1]}(x)$ is palindromic and real-rooted.
	\end{enumerate}
\end{conj}

Very recently, Jiang~\cite{Jiang2026} proved Conjecture~\ref{conj:DD-three-chains} and established real stability of the layer-refined antichain polynomials
by using a lattice-path reflection and Jacobi polynomials.
Ding and Ding~\cite{DingDing2026} subsequently settled Conjecture~4.2 in~\cite{DingDong2019},
which implies Conjecture~\ref{conj:DD-three-chains} (i).
We give an independent proof of Conjecture~\ref{conj:DD-three-chains} based on an explicit enumeration of antichains with prescribed layer sizes (Theorem~\ref{thm:fixed-layer-enumeration}).
Our technique is to construct a bijection between antichains and labeled Dyck paths.
This provides the basis for our proofs of real-rootedness and palindromicity.
Such an enumeration also has an application in chemical graph theory.

Based on extensive numerical evidence, He et al.~\cite{HeLangnerWitek2021} conjectured the determinantal formula \eqref{eq-ZZ} for the Zhang-Zhang polynomial of the hexagonal flake $O(2,m,n)$.
Witek et al.~\cite[Eq.~(7)]{WitekPodeszwaLangner2021} subsequently restated the formula and emphasized its conjectural status.
% Theorem~\ref{thm:fixed-layer-enumeration} gives an exact layer-refined enumeration of the antichains in $[2]\times[m]\times[n]$.
We give a combinatorial proof of this formula by indentifying the Zhang-Zhang polynomial $ZZ(O(2,m,n),x)$ with $\cN_{[2]\times[m]\times[n]}(x+1)$.
(Details will be shown in Section~\ref{subsec:ZZ-polynomials}.)

\begin{thm}
	\label{thm:HLW-determinant}
	For all positive integers $m$ and $n$,
	\begin{equation}\label{eq-ZZ}
		ZZ(O(2,m,n),x)
		=
		\det
		\begin{pmatrix}
			J_{m,n}(x) & K_{m,n}(x)\\
			K_{n,m}(x) & J_{m,n}(x)
		\end{pmatrix},
	\end{equation}
	where
	\begin{equation}
		\label{eq:HLW-entries}
		\begin{aligned}
			J_{m,n}(x)
			&=
			\sum_{r\ge0}\binom mr\binom nr(1+x)^r,\\
			K_{m,n}(x)
			&=
			\sum_{r\ge0}
			\binom{m-1}{r}\binom{n+1}{r+2}(1+x)^{r+1}.
		\end{aligned}
	\end{equation}
\end{thm}

% We next apply the defect-one method to $[k]\times H_n$. 
% A symmetric chain decomposition gives a weak shadow inequality, and an analysis of the equality cases reduces the possible additional parameters to a cubic Diophantine equation. 
% Its relevant integral solutions are ruled out using elliptic logarithms together with a reproducible computational certificate. 
% Consequently, $\mathcal N_{[k]\times H_n}(x)$ is palindromic
% if and only if $k=1$ and $n$ is odd, or
% \[
% (k,n)\in\{(3,2),(5,3),(3,4),(7,4)\}.
% \]
% In each of these cases, the polynomial has only simple negative zeros.

% For the remaining connected minuscule posets, we combine the palindromicity classifications of Ding and Dong with known results and explicit calculations to verify real-rootedness in every palindromic case. 
% This completes the proof of Theorem~\ref{thm:main-minuscule}.

The minuscule conjecture naturally raises the question of whether similar coefficient properties persist for the broader class of Peck posets.
Following Stanley~\cite{Stanley1984Quotients}, a finite graded poset is \defi{Peck} if it is rank-symmetric, rank-unimodal, and strongly Sperner. 
Here strongly Sperner means that, for every positive integer $r$, the maximum cardinality of a union of $r$ antichains equals the sum of the $r$ largest rank numbers.
Ding and Dong proposed the following conjecture for Peck posets.

\begin{conj}[{\cite[Conjecture C]{DingDong2019}}]\label{conj:DD-Peck}
	If $P$ is a connected finite Peck poset, then $\mathcal N_P(x)$ is log-concave.
\end{conj}

We disprove this conjecture by constructing infinitely many connected Peck posets whose antichain polynomials are not even unimodal in Section~\ref{sec:minuscule-consequences}.
For example, on such poset has the antichain polynomial
\(
1+16x+15x^2+20x^3+15x^4+6x^5+x^6.
\)

The rest of the paper is organized as follows.
In Section~\ref{sec:defect-one}, we give a formula for the number of antichains, counted by size, of an arbitrary poset,
and establishes shadow-localization results.
Section~\ref{sec:three-chains} focuses on the product of three chains $[k]\times[m]\times[n]$.
We point out that the palindromicity of its antichain polynomials occurs only for $k=1,2$.
By constructing a bijection between antichains and labeled Dyck paths, we also give a layer-refined enumeration of the antichains in $[2]\times[m]\times[n]$.
This result is applied to prove a conjectured determinantal formula for a family of Zhang-Zhang polynomials.
We also obtain real-rootedness and stability results when the shortest chain has length at most two.
Section~\ref{sec:Hn} deals with the the product $[k]\times H_n$ by classifying palindromic cases and proving real-rootedness of its chain polynomials.
In Section~\ref{sec:minuscule-consequences}, we first complete the proof of Theorem~\ref{thm:main-minuscule},
and then construct infinitely many connected Peck posets whose antichain polynomials are not unimodal,
disproving Conjecture \ref{conj:DD-Peck}.
Appendix~\ref{app:technical-counts} contains the technical shadow classifications and counting arguments.
Appendix~\ref{app:elliptic-obstruction} proves the required Diophantine obstruction using elliptic logarithms and provides a computational certificate.

\section{A coefficient criterion for palindromicity}
\label{sec:defect-one}

Let $P$ be a finite poset and put $d=w(P)$. 
Since $[x^0]\mathcal N_P(x)=1$ and $[x^d]\mathcal N_P(x)$ counts the maximum antichains of $P$, palindromicity forces $P$ to have a unique maximum antichain.
To obtain further restrictions, we express the coefficients near degree $d$ in terms of antichains outside a fixed maximum antichain.

\subsection{Coefficients relative to a maximum antichain}
\label{subsec:central-completion}

Fix a maximum antichain $L$ of $P$ with $|L|=d$. 
For $B\in\Ant(P\setminus L)$, define
\[
\Gamma_L(B)
=
\{z\in L:z\text{ is comparable with some }b\in B\},
\]
and set $\delta_L(B)=|\Gamma_L(B)|-|B|$. 
We call $\delta_L(B)$ the \defi{$L$-defect} of $B$.
The set $B\sqcup\bigl(L\setminus\Gamma_L(B)\bigr)$
is an antichain of cardinality $d-\delta_L(B)$. 
Since $L$ has maximum cardinality $d$, we must have
$d-\delta_L(B)\le d$, and hence $\delta_L(B)\ge0$.
Moreover, if $L$ is the unique maximum antichain and $B$ is nonempty,
then $B\sqcup(L\setminus\Gamma_L(B))$ is distinct from $L$.
Its cardinality is therefore strictly less than $d$, which implies that
$\delta_L(B)\ge1$. In partucular, a nonempty antichain $B\subseteq P\setminus L$ is called an
\defi{$L$-defect-one antichain} if $\delta_L(B)=1$.

Thus the $L$-defect measures how far the canonical completion of $B$ falls short of a maximum antichain. 
The following theorem translates this completion procedure into an expansion of $\mathcal N_P(x)$, from which its coefficients near the leading term can be determined.
\begin{thm}
\label{thm:central-completion}
Let $L$ be a maximum antichain of $P$ with $|L|=d$. 
Then
\begin{equation}
\label{eq:central-completion}
\mathcal N_P(x)
=
\sum_{B\in\Ant(P\setminus L)}
x^{|B|}(1+x)^{d-|\Gamma_L(B)|}.
\end{equation}
Moreover, for every integer $0\le j\le d$,
\begin{equation}
\label{eq:near-top-coefficients}
[x^{d-j}]\mathcal N_P(x)
=
\sum_{B\in\Ant(P\setminus L)}
\binom{d-|\Gamma_L(B)|}{j-\delta_L(B)},
\end{equation}
where $\binom{a}{b}=0$ whenever $b<0$ or $b>a$.
\end{thm}

\begin{proof}
Let $\mathcal A=B\sqcup C$ be an antichain of $P$, where $B=\mathcal A\setminus L$ and $C=\mathcal A\cap L$. 
Then $B\in\Ant(P\setminus L)$. 
Moreover, no element of $C$ is comparable with an element of $B$, so $C\subseteq L\setminus\Gamma_L(B)$.

Conversely, fix $B\in\Ant(P\setminus L)$ and choose any subset $C\subseteq L\setminus\Gamma_L(B)$. 
Since $L$ is an antichain, the elements of $C$ are pairwise incomparable. 
By the definition of $\Gamma_L(B)$, no element of $C$ is comparable with an element of $B$. 
Hence $B\sqcup C$ is an antichain of $P$.

Thus, for a fixed $B\in\Ant(P\setminus L)$, the possible choices for $\mathcal A\cap L$ are precisely the subsets of $L\setminus\Gamma_L(B)$.
Their contribution to $\mathcal N_P(x)$ is
\[
\sum_{C\subseteq L\setminus\Gamma_L(B)}x^{|B|+|C|}
=
x^{|B|}(1+x)^{d-|\Gamma_L(B)|}.
\]
Summing over all $B\in\Ant(P\setminus L)$ proves~\eqref{eq:central-completion}.
Fix $B\in\Ant(P\setminus L)$. 
The coefficient of $x^{d-j}$ in the corresponding summand of~\eqref{eq:central-completion} is
\[
\binom{d-|\Gamma_L(B)|}{d-j-|B|}
=
\binom{d-|\Gamma_L(B)|}{j-\delta_L(B)},
\]
where the equality follows from $\delta_L(B)=|\Gamma_L(B)|-|B|$. 
Summing these contributions over all $B\in\Ant(P\setminus L)$ proves~\eqref{eq:near-top-coefficients}.
\end{proof}

We now specialize~\eqref{eq:near-top-coefficients} to the first coefficient below the leading term. 
Suppose that $L$ is the unique maximum antichain of $P$, and define
\[
D_1(P,L)
:=
\bigl|\{B\in\Ant(P\setminus L):\delta_L(B)=1\}\bigr|.
\]
Thus $D_1(P,L)$ is the number of defect-one antichains relative to $L$.

\begin{cor}
\label{cor:defect-one-coefficient}
If $L$ is the unique maximum antichain of $P$, then
\begin{equation}
\label{eq:defect-one-coefficient}
[x^{d-1}]\mathcal N_P(x)=d+D_1(P,L).
\end{equation}
In particular, if $\mathcal N_P(x)$ is palindromic, then $D_1(P,L)=|P|-d$.
\end{cor}

\begin{proof}
Taking $j=1$ in Theorem~\ref{thm:central-completion} gives
\[
[x^{d-1}]\mathcal N_P(x)
=
\sum_{B\in\Ant(P\setminus L)}
\binom{d-|\Gamma_L(B)|}{1-\delta_L(B)}.
\]
For $B=\varnothing$, we have $\Gamma_L(B)=\varnothing$ and $\delta_L(B)=0$, so the corresponding summand is $\binom d1=d$.

Suppose that $B$ is nonempty. We have $\delta_L(B)\ge1$, since $L$ is the unique maximum antichain. 
The corresponding summand equals $1$ when $\delta_L(B)=1$ and equals $0$ when $\delta_L(B)\ge2$. 
This proves~\eqref{eq:defect-one-coefficient}.

Finally, $[x]\mathcal N_P(x)=|P|$, since every singleton is an antichain. 
If $\mathcal N_P(x)$ is palindromic, then $[x^{d-1}]\mathcal N_P(x)=[x]\mathcal N_P(x)$, and therefore $D_1(P,L)=|P|-d$.
\end{proof}

Suppose that $P=P_0\sqcup\cdots\sqcup P_R$ is graded with rank function $\rho$, and that $L=P_r$ is its unique maximum antichain. 
Write $P_{<r}=\bigcup_{i<r}P_i$ and $P_{>r}=\bigcup_{i>r}P_i$, so that they consist of the elements in ranks strictly below and strictly above $L=P_r$, respectively; the following lemma shows that an $L$-defect-one antichain cannot meet both sides of $L$.

\begin{lem}
\label{lem:one-sided-defect}
For $B\in\Ant(P\setminus L)$, define $B^-:=B\cap P_{<r}$ and $B^+:=B\cap P_{>r}$. Then $\Gamma_L(B^-)\cap\Gamma_L(B^+)=\varnothing$ and
\begin{equation}
\label{eq:one-sided-defect-additivity}
\delta_L(B)=\delta_L(B^-)+\delta_L(B^+).
\end{equation}
In particular, if $\delta_L(B)=1$, exactly one of $B^-$ and $B^+$ is nonempty.
\end{lem}

\begin{proof}
Suppose that $z\in\Gamma_L(B^-)\cap\Gamma_L(B^+)$. 
Then there exist $u\in B^-$ and $v\in B^+$ such that $u$ and $v$ are both comparable with $z$. 
Since $\rho(u)<\rho(z)<\rho(v)$, we must have $u<z<v$. 
This contradicts the assumption that $B$ is an antichain.
Hence the two sets are disjoint.

Since $B=B^-\sqcup B^+$ and $\Gamma_L(B^-)\cap\Gamma_L(B^+)=\varnothing$, we also have $\Gamma_L(B)=\Gamma_L(B^-)\sqcup\Gamma_L(B^+)$.
Therefore
\[
\begin{aligned}
\delta_L(B)
&=|\Gamma_L(B)|-|B|\\
&=\bigl(|\Gamma_L(B^-)|-|B^-|\bigr)
+\bigl(|\Gamma_L(B^+)|-|B^+|\bigr)\\
&=\delta_L(B^-)+\delta_L(B^+),
\end{aligned}
\]
which prove \eqref{eq:one-sided-defect-additivity}

Now suppose that $\delta_L(B)=1$.
If both $B^-$ and $B^+$ were nonempty, the uniqueness of $L$ would imply $\delta_L(B^-)\ge1$ and $\delta_L(B^+)\ge1$.
Then~\eqref{eq:one-sided-defect-additivity} gives $\delta_L(B)\ge2$, contradicting $\delta_L(B)=1$.
Hence exactly one of $B^-$ and $B^+$ is nonempty.
\end{proof}

For the self-dual reduction, define
\begin{equation}
\label{eq:one-sided-defect-count}
D_1^-(P,L)
:=
\bigl|\{B\in\Ant(P_{<r}):\delta_L(B)=1\}\bigr|.
\end{equation}
Thus $D_1^-(P,L)$ counts the $L$-defect-one antichains lying strictly below $L$; the following theorem shows that, in the self-dual case, the same number lie above $L$.
\begin{thm}
\label{thm:self-dual-defect}
Suppose that $P$ is self-dual.
Then $R=2r$ and $D_1(P,L)=2D_1^-(P,L)$.
Consequently, if $\mathcal N_P(x)$ is palindromic, then $2D_1^-(P,L)=|P|-d$, where $d=|L|=w(P)$.
\end{thm}
\begin{proof}
Choose an order-reversing bijection $\tau\colon P\to P$.
It maps $P_i$ bijectively onto $P_{R-i}$ for every $0\le i\le R$,
so $P$ is rank-symmetric. In particular, $|P_{R-r}|=|P_r|=d$.
Thus $P_{R-r}$ is also a maximum antichain. Since $L=P_r$ is the
unique maximum antichain, we must have $R-r=r$, and hence $R=2r$.

It follows that $\tau$ fixes $L=P_r$ setwise and maps $P_{<r}$
bijectively onto $P_{>r}$. Let $B\in\Ant(P_{<r})$. For every $y\in L$,
\[
\tau(y)\in\Gamma_L(\tau(B))
\quad\Longleftrightarrow\quad
y\in\Gamma_L(B),
\]
because $\tau$ preserves comparability. Therefore $\Gamma_L(\tau(B))=\tau(\Gamma_L(B))$.
Since $\tau$ preserves cardinality, we obtain
\[
\delta_L(\tau(B))
=
|\Gamma_L(\tau(B))|-|\tau(B)|
=
|\Gamma_L(B)|-|B|
=
\delta_L(B).
\]
Hence $\tau$ induces a bijection between the $L$-defect-one antichains
below $L$ and those above $L$.

By Lemma~\ref{lem:one-sided-defect}, every $L$-defect-one antichain lies
entirely below or entirely above $L$. The two classes are therefore
disjoint and equinumerous, which gives $D_1(P,L)=2D_1^-(P,L)$.
Finally, if $\mathcal N_P(x)$ is palindromic, then
Corollary~\ref{cor:defect-one-coefficient} gives $D_1(P,L)=|P|-d$.
Combining the last two identities yields $2D_1^-(P,L)=|P|-d$,
as required.
\end{proof}

\subsection{Shadow localization below a maximum rank}
\label{subsec:rank-pushing}

Continue to assume that $P$ is graded and that $L=P_r$ is its unique maximum antichain.
We now locate the defect-one antichains below $L$ by pushing their lowest occupied ranks toward $L$.
For $X\subseteq P_i$, define its \defi{upper shadow} by
\[
\nabla X
:=
\{y\in P_{i+1}:x\lessdot y\text{ for some }x\in X\},
\]
where $x\lessdot y$ means that $y$ covers $x$.
For a nonempty antichain $B\in\Ant(P_{<r})$, define its lowest occupied rank by $\rho_{\min}(B)=\min\{\rho(b):b\in B\}$.
If $\rho_{\min}(B)\le r-2$, define
\[
B^\uparrow
:=
\bigl(B\setminus P_{\rho_{\min}(B)}\bigr)
\cup
\nabla\bigl(B\cap P_{\rho_{\min}(B)}\bigr).
\]
In other words, this operation replaces the lowest-rank part of $B$ by
its upper shadow, thereby moving it one rank closer to $L$, while leaving
the rest of $B$ unchanged. We call this operation a
\defi{rank-pushing step}. The following lemma shows that the resulting
set remains an antichain below $L$, preserves its comparable elements in
$L$, and gives a precise formula for the change in its $L$-defect.

\begin{lem}
\label{lem:rank-pushing}
The set $B^\uparrow$ is an antichain contained in $P_{<r}$, and $\Gamma_L(B^\uparrow)=\Gamma_L(B)$.
Moreover,
\begin{equation}
\label{eq:rank-pushing-defect}
\delta_L(B^\uparrow)
=
\delta_L(B)
-
\left(
\left|\nabla\bigl(B\cap P_{\rho_{\min}(B)}\bigr)\right|
-
\left|B\cap P_{\rho_{\min}(B)}\right|
\right).
\end{equation}
\end{lem}

\begin{proof}
The sets $B\setminus P_{\rho_{\min}(B)}$ and $\nabla\bigl(B\cap P_{\rho_{\min}(B)}\bigr)$ are antichains, since the former is contained in $B$ and the latter is contained in a single rank level.

Suppose that some $y\in B\setminus P_{\rho_{\min}(B)}$ is comparable with some $z\in\nabla\bigl(B\cap P_{\rho_{\min}(B)}\bigr)$. 
By the definition of the upper shadow, there exists $x\in B\cap P_{\rho_{\min}(B)}$ such that $x\lessdot z$. 
If $z\le y$, then $x<y$, which contradicts the fact that $x,y\in B$.
If $y\le z$, the fact that $y$ does not lie in the lowest occupied rank of $B$ suggests $\rho(y)\ge \rho_{\min}(B)+1=\rho(z)$. 
However, $y\le z$ implies $\rho(y)\le\rho(z)$. 
Hence $\rho(y)=\rho(z)$, and therefore $y=z$. 
This again gives $x<y$, contradicting that $B$ is an antichain.
Thus no element of $B\setminus P_{\rho_{\min}(B)}$ is comparable with an element of $\nabla\bigl(B\cap P_{\rho_{\min}(B)}\bigr)$, and hence $B^\uparrow$ is an antichain.

Secondly, the assumption $\rho_{\min}(B)\le r-2$ gives
\[
\nabla\bigl(B\cap P_{\rho_{\min}(B)}\bigr)
\subseteq P_{\rho_{\min}(B)+1}\subseteq P_{<r}.
\]
Since $B\setminus P_{\rho_{\min}(B)}\subseteq P_{<r}$, it follows that $B^\uparrow\subseteq P_{<r}$.

Suppose that $u\in L$ is comparable with some $x\in B\cap P_{\rho_{\min}(B)}$. 
Since $\rho(x)<r=\rho(u)$, we have $x<u$. 
Choose a saturated chain from $x$ to $u$, and let $z$ be the element immediately above $x$ on this chain. 
Then $z\in\nabla\bigl(B\cap P_{\rho_{\min}(B)}\bigr)$ and $z\le u$.
Thus $\Gamma_L(B\cap P_{\rho_{\min}(B)}) \subseteq \Gamma_L\bigl(\nabla(B\cap P_{\rho_{\min}(B)})\bigr)$.
Conversely, suppose that $u\in L$ is comparable with some $z\in\nabla(B\cap P_{\rho_{\min}(B)})$.
There exists $x\in B\cap P_{\rho_{\min}(B)}$ such that $x\lessdot z$. 
Since $z<u$, we have $x<u$, and hence $u\in\Gamma_L(B\cap P_{\rho_{\min}(B)})$. 
Therefore we have $\Gamma_L\bigl(B\cap P_{\rho_{\min}(B)}\bigr)=\Gamma_L\bigl(\nabla(B\cap P_{\rho_{\min}(B)})\bigr)$,
which implies $\Gamma_L(B^\uparrow)=\Gamma_L(B)$ by definition of $B^\uparrow$.

Finally, we prove the two parts forming $B^\uparrow$ are disjoint.
Assume that $y\in\bigl(B\setminus P_{\rho_{\min}(B)}\bigr)\cap\nabla\bigl(B\cap P_{\rho_{\min}(B)}\bigr)$, then there exists $x\in B\cap P_{\rho_{\min}(B)}$ such that $x\lessdot y$. 
Since $x,y\in B$, this contradicts the assumption that $B$ is an antichain. 
Therefore
\[
|B^\uparrow|
=
|B|
-
\left|B\cap P_{\rho_{\min}(B)}\right|
+
\left|\nabla\bigl(B\cap P_{\rho_{\min}(B)}\bigr)\right|.
\]
Consequently,
\[
\begin{aligned}
\delta_L(B^\uparrow)
&=
|\Gamma_L(B^\uparrow)|-|B^\uparrow|\\
&=
|\Gamma_L(B)|-|B|
+\left|B\cap P_{\rho_{\min}(B)}\right|
-\left|\nabla\bigl(B\cap P_{\rho_{\min}(B)}\bigr)\right|\\
&=
\delta_L(B)
-
\left(
\left|\nabla\bigl(B\cap P_{\rho_{\min}(B)}\bigr)\right|
-
\left|B\cap P_{\rho_{\min}(B)}\right|
\right).
\end{aligned}
\]
This proves the lemma.
\end{proof}
We say that $P$ has the \defi{upward matching property below $L=P_r$} if $|\nabla X|\ge |X|$
for every $X\subseteq P_i$ and every $i<r$. In particular, if the above inequality is strict for every nonempty $X$, we say that
$P$ has the \defi{strict upward matching property below $L$}.
\begin{thm}
\label{thm:weak-shadow-propagation}
Assume that $P$ has the upward matching property below $L$. 
Let $B\in\Ant(P_{<r})$ satisfy $\delta_L(B)=1$. 
If $\rho_{\min}(B)\le r-2$, then
\begin{equation}
\label{eq:rank-pushing-zero-growth}
\left|\nabla\bigl(B\cap P_{\rho_{\min}(B)}\bigr)\right|
=
\left|B\cap P_{\rho_{\min}(B)}\right|
\end{equation}
and $\delta_L(B^\uparrow)=1$.
Consequently, repeated rank pushing produces a defect-one antichain $B'\subseteq P_{r-1}$ satisfying $|\nabla B'|=|B'|+1$.
\end{thm}

\begin{proof}
Since $\rho_{\min}(B)\le r-2$, Lemma~\ref{lem:rank-pushing} shows that $B^\uparrow$ is an antichain contained in $P_{<r}$. 
Moreover,~\eqref{eq:rank-pushing-defect} gives
\[
1
=
\delta_L(B^\uparrow)
+
\left(
\left|\nabla\bigl(B\cap P_{\rho_{\min}(B)}\bigr)\right|
-
\left|B\cap P_{\rho_{\min}(B)}\right|
\right).
\]
The upward matching property shows that the second term on the right-hand side is nonnegative. 
It also implies that $B^\uparrow$ is nonempty.
Since $L$ is the unique maximum antichain, we have
$\delta_L(B^\uparrow)\ge1$. 
Hence~\eqref{eq:rank-pushing-zero-growth} holds and $\delta_L(B^\uparrow)=1$.

Each rank-pushing step strictly increases the lowest occupied rank.
The preceding argument may therefore be repeated until a defect-one antichain $B'\subseteq P_{r-1}$ is obtained. 
Since $B'$ lies in the rank immediately below $L=P_r$, its neighborhood in $L$ is precisely its upper shadow. 
Thus $|\nabla B'|-|B'|=\delta_L(B')=1$, which prove the theorem.
\end{proof}

\begin{cor}
\label{cor:strict-shadow-localization}
Suppose that $P$ has the strict upward matching property below $L$. 
Then the defect-one antichains contained in $P_{<r}$ are exactly the nonempty subsets $B\subseteq P_{r-1}$ satisfying $|\nabla B|=|B|+1$.
\end{cor}
\begin{proof}
Let $B\in\Ant(P_{<r})$ have defect one. 
If $\rho_{\min}(B)\le r-2$, Theorem~\ref{thm:weak-shadow-propagation} would give
the equality in~\eqref{eq:rank-pushing-zero-growth},
contradicting the strict-shadow assumption. 
Hence $\rho_{\min}(B)=r-1$, so $B\subseteq P_{r-1}$.

It remains to characterize the defect-one antichains contained in $P_{r-1}$. 
Let $B\subseteq P_{r-1}$ be nonempty. 
Since $P_{r-1}$ is a rank level, $B$ is automatically an antichain.
Suppose that $u\in\Gamma_L(B)$. 
Then $u\in L=P_r$ is comparable with some $b\in B\subseteq P_{r-1}$. 
Since $\rho(u)=\rho(b)+1$, we must have $b\lessdot u$. Hence $u\in\nabla B$. 
Conversely, every element of $\nabla B$ belongs to $P_r=L$ and is comparable with an element of $B$.
Therefore $\Gamma_L(B)=\nabla B$.
It follows that
\[
\delta_L(B)
=
|\Gamma_L(B)|-|B|
=
|\nabla B|-|B|.
\]
Consequently, $\delta_L(B)=1$ if and only if $|\nabla B|=|B|+1$. 
\end{proof}

\begin{rem}
\label{rem:shadow-criteria}
We record two standard ways to verify the shadow inequalities used above.

\begin{enumerate}[label=\textup{(\roman*)}]
\item
A graded poset $P$ has the \defi{normalized matching property} if
\[
\frac{|\nabla X|}{|P_{i+1}|}
\ge
\frac{|X|}{|P_i|}
\]
for every $X\subseteq P_i$ and every $0\le i<R$.
If $|P_0|<|P_1|<\cdots<|P_r|$, then every nonempty $X\subseteq P_i$ with $i<r$ satisfies
\[
|\nabla X|
\ge
\frac{|P_{i+1}|}{|P_i|}|X|
>
|X|.
\]

\item
Recall that a \defi{symmetric chain decomposition} of a graded poset of rank $R$ is a partition into saturated chains such that the minimum and maximum ranks of each chain sum to $R$.
Now suppose that $R=2r$ and that $P$ admits a symmetric chain decomposition $\mathscr C$.
For $X\subseteq P_i$ with $i<r$ and $x\in X$, let $C\in\mathscr{C}$ be the symmetric chain containing $x$.
By symmetry, $x$ has an immediate successor $x^+$ in $C$, and $x^+\in\nabla X$.
Distinct elements of $X$ lie in distinct chains and therefore have distinct successors. 
Hence $|\nabla X|\ge|X|$.
\end{enumerate}
\end{rem}

\section{Products of three chains}
\label{sec:three-chains}

Throughout this section, let $1\le k\le m\le n$. 
For $k\ge3$, the palindromicity classification follows from the coefficient criterion and shadow-localization results developed in Section~\ref{sec:defect-one}. 
The cases $k=1$ and $k=2$ are then treated by direct enumeration and a layer-refined enumeration, respectively. 

\subsection{\texorpdfstring{The case $k\ge3$}{The case k at least 3}}
\label{subsec:three-chain-palindromicity}

We begin by assuming that $3\le k\le m\le n$. 
As noted at the beginning of Section~\ref{sec:defect-one}, 
every palindromic antichain polynomial is monic.
By the monicity classification of Ding and Dong~\cite[Lemma~2.3]{DingDong2019}, there exists an integer $s$ such that
\begin{equation}
\label{eq:three-chain-monicity-parameters}
n=m+k-1-2s,
\qquad
0\le s\le \left\lfloor\frac{k-1}{2}\right\rfloor.
\end{equation}
Put $P=[k]\times[m]\times[n]$, with rank function $\rho(i,j,\ell)=i+j+\ell-3$.
By~\eqref{eq:three-chain-monicity-parameters}, the central rank of $P$ is 
\begin{equation}\label{eq:r=k+m-s-2}
r=k+m-s-2.
\end{equation}
We next describe the two rank levels $P_r$ and $P_{r-1}$ that enter the defect-one calculation.

Define
\begin{align*}
L_{k,m,s}
:&=
\{(i,j)\in[k]\times[m]:s+2\le i+j\le k+m-s\}\\
M_{k,m,s}
:&=
\{(i,j)\in[k]\times[m]:s+1\le i+j\le k+m-s-1\}.
\end{align*}
\begin{lem}
\label{lem:central-level-geometry}
Under the coordinate projection $(i,j,\ell)\mapsto(i,j)$, the rank levels $P_r$ and $P_{r-1}$ map bijectively onto $L_{k,m,s}$ and $M_{k,m,s}$, respectively.
Furthermore,
\[
|P_r|=km-s(s+1),
\qquad
|P_{r-1}|=km-s(s+1)-1.
\]
\end{lem}

\begin{proof}
If $(i,j,\ell)\in P_r$, then $\ell=k+m-s+1-i-j$ by \eqref{eq:r=k+m-s-2}.
Using~\eqref{eq:three-chain-monicity-parameters}, the condition $1\le\ell\le n$ is equivalent to $s+2\le i+j\le k+m-s$.
Thus the coordinate projection maps $P_r$ bijectively onto $L_{k,m,s}$.
By \eqref{eq:r=k+m-s-2} again, the same calculation for $P_{r-1}$ gives $\ell=k+m-s-i-j$.
The condition $1\le\ell\le n$ is equivalent to $s+1\le i+j\le k+m-s-1$.
Hence the coordinate projection maps $P_{r-1}$ bijectively onto $M_{k,m,s}$.

The complement of $L_{k,m,s}$ in $[k]\times[m]$ consists of the two corner triangles $i+j\le s+1$ and $i+j\ge k+m-s+1$, each of cardinality $\binom{s+1}{2}$.
Therefore
\[
|L_{k,m,s}|=km-2\binom{s+1}{2}=km-s(s+1).
\]

Similarly, the two corner triangles complementary to $M_{k,m,s}$ have
$\binom{s}{2}$ and $\binom{s+2}{2}$ elements. Hence
\[
|M_{k,m,s}|
=
km-\binom{s}{2}-\binom{s+2}{2}
=
km-s(s+1)-1.
\]
\end{proof}

Set $d=|P_r|=km-s(s+1)$.
Since $\mathcal N_P(x)$ is assumed to be palindromic, $P_r$ is the unique maximum antichain of $P$.
For the parameters above, the rank numbers of $P$ are strictly increasing below rank $r$. 
Moreover, products of chains have the normalized matching property~\cite{Engel1997}. 
Hence Remark~\ref{rem:shadow-criteria} gives strict shadow growth below $P_r$.

Under the identifications in Lemma~\ref{lem:central-level-geometry}, the upper shadow of a subset $S\subseteq M_{k,m,s}$ is
\[
\nabla S
=
\{(i,j)\in L_{k,m,s}:
(i,j)\in S,\ (i-1,j)\in S,\text{ or }(i,j-1)\in S\}.
\]
We call a nonempty subset $S\subseteq M_{k,m,s}$ \defi{tight} if $|\nabla S|=|S|+1$.

By Corollary~\ref{cor:strict-shadow-localization}, the defect-one antichains below $P_r$ are precisely the tight subsets of $M_{k,m,s}$.
Denote their number by
\[
T_{k,m,s}
=
\bigl|
\{S\subseteq M_{k,m,s}:S\ne\varnothing,\ |\nabla S|=|S|+1\}
\bigr|.
\]
We first consider the case $s\ge1$. 
Define the following three boundary subsets of $M_{k,m,s}$:
\[
C_0=\{(i,s+1-i):1\le i\le s\},
\]
\[
C_1=\{(i,m):1\le i\le k-s-1\},
\qquad
C_2=\{(k,j):1\le j\le m-s-1\}.
\]
We order $C_0$ and $C_1$ by increasing first coordinate and $C_2$ by increasing second coordinate. 
A consecutive segment of $C_i$ is a nonempty set of consecutive elements in this order.
Figure~\ref{fig:three-boundary-shadows} illustrates the three boundary sets and their upper shadows for $(k,m,s)=(5,6,2)$.

\begin{figure}[htbp]
\centering

\begin{tikzpicture}[x=.62cm,y=.62cm,line join=miter]

\newcommand{\drawcoordinatelabels}{%
\foreach \i in {1,...,5}
\node[font=\scriptsize] at ({\i-.5},-.30) {$\i$};

\foreach \j in {1,...,6}
\node[font=\scriptsize] at (-.30,{\j-.5}) {$\j$};

\node[font=\small] at (2.5,-.72) {$i$};
\node[font=\small] at (-.72,3) {$j$};
}

% Left panel: C_0, C_1, C_2 in M_{5,6,2}
\begin{scope}

\foreach \i in {1,...,5}{
\foreach \j in {1,...,6}{
\pgfmathtruncatemacro{\ijsum}{\i+\j}

\ifnum\ijsum<3\relax
\path[
fill=white,
draw=gray!55,
line width=.4pt
]
({\i-1},{\j-1}) rectangle (\i,\j);
\else
\ifnum\ijsum>8\relax
\path[
fill=white,
draw=gray!55,
line width=.4pt
]
({\i-1},{\j-1}) rectangle (\i,\j);
\else
\path[
fill=gray!12,
draw=gray!55,
line width=.4pt
]
({\i-1},{\j-1}) rectangle (\i,\j);
\fi
\fi
}
}

% C_0 = {(1,2),(2,1)}
\foreach \i/\j in {1/2,2/1}{
\path[
fill=red!30,
draw=gray!60,
line width=.4pt
]
({\i-1},{\j-1}) rectangle (\i,\j);
}

% C_1 = {(1,6),(2,6)}
\foreach \i/\j in {1/6,2/6}{
\path[
fill=blue!28,
draw=gray!60,
line width=.4pt
]
({\i-1},{\j-1}) rectangle (\i,\j);
}

% C_2 = {(5,1),(5,2),(5,3)}
\foreach \i/\j in {5/1,5/2,5/3}{
\path[
fill=green!30,
draw=gray!60,
line width=.4pt
]
({\i-1},{\j-1}) rectangle (\i,\j);
}

\drawcoordinatelabels

\node[font=\small] at (2.5,6.48)
{$C_0,C_1,C_2\subseteq M_{5,6,2}$};

\end{scope}

% Upper-shadow arrow
\draw[->,thick] (5.65,3) -- (7.3,3);
\node[above,font=\small] at (6.5,3) {$\nabla$};

% Right panel: upper shadows in L_{5,6,2}
\begin{scope}[shift={(8.5,0)}]

\foreach \i in {1,...,5}{
\foreach \j in {1,...,6}{
\pgfmathtruncatemacro{\ijsum}{\i+\j}

\ifnum\ijsum<4\relax
\path[
fill=white,
draw=gray!55,
line width=.4pt
]
({\i-1},{\j-1}) rectangle (\i,\j);
\else
\ifnum\ijsum>9\relax
\path[
fill=white,
draw=gray!55,
line width=.4pt
]
({\i-1},{\j-1}) rectangle (\i,\j);
\else
\path[
fill=gray!12,
draw=gray!55,
line width=.4pt
]
({\i-1},{\j-1}) rectangle (\i,\j);
\fi
\fi
}
}

% \nabla C_0 = {(1,3),(2,2),(3,1)}
\foreach \i/\j in {1/3,2/2,3/1}{
\path[
fill=red!30,
draw=gray!60,
line width=.4pt
]
({\i-1},{\j-1}) rectangle (\i,\j);
}

% \nabla C_1 = {(1,6),(2,6),(3,6)}
\foreach \i/\j in {1/6,2/6,3/6}{
\path[
fill=blue!28,
draw=gray!60,
line width=.4pt
]
({\i-1},{\j-1}) rectangle (\i,\j);
}

% \nabla C_2 = {(5,1),(5,2),(5,3),(5,4)}
\foreach \i/\j in {5/1,5/2,5/3,5/4}{
\path[
fill=green!30,
draw=gray!60,
line width=.4pt
]
({\i-1},{\j-1}) rectangle (\i,\j);
}

\drawcoordinatelabels

\node[font=\small] at (2.5,6.48)
{$\nabla C_0,\nabla C_1,\nabla C_2\subseteq L_{5,6,2}$};

\end{scope}

\end{tikzpicture}

\caption{
The boundary sets $C_0,C_1,C_2$ in $M_{5,6,2}$ and their upper shadows in $L_{5,6,2}$. The red, blue, and green cells correspond to $C_0,C_1$, and $C_2$, respectively. 
Each boundary set satisfies $|\nabla C_i|=|C_i|+1$.
}
\label{fig:three-boundary-shadows}
\end{figure}
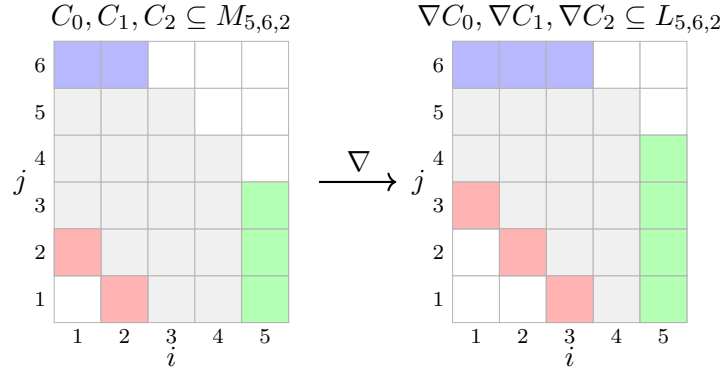

\begin{lem}
\label{lem:tight-set-classification}
Assume $s\ge1$. 
A nonempty subset $S\subseteq M_{k,m,s}$ is tight if and only if either $S=M_{k,m,s}$ or $S$ is a consecutive segment of one of $C_0,C_1,C_2$.
\end{lem}

\begin{proof}
The proof is given in Appendix~\ref{app:technical-counts}.
\end{proof}

Since a set of cardinality $q$ has $1+2+\dots+q=\binom{q+1}{2}$ nonempty consecutive
segments, Lemma~\ref{lem:tight-set-classification} gives
\begin{equation}
\label{eq:three-chain-tight-positive}
T_{k,m,s}
=
1+\binom{s+1}{2}
+\binom{k-s}{2}
+\binom{m-s}{2}.
\end{equation}

We now consider the case $s=0$. 
In this case,
\[
L_{k,m,0}=[k]\times[m],
\qquad
M_{k,m,0}=([k]\times[m])\setminus\{(k,m)\}.
\]
The tight subsets of $M_{k,m,0}$ have a different description.

\begin{lem}
\label{lem:tight-set-classification-zero}
A nonempty subset $S\subseteq M_{k,m,0}$ is tight if and only if one of the following holds:

\begin{enumerate}[label=\textup{(\roman*)}]

\item
$S=J\setminus\{(k,m)\}$ for an upper order ideal $J\subseteq[k]\times[m]$ with $|J|\ge2$;

\item
$S$ is a consecutive segment of $\{(i,m):1\le i\le k-2\}$;

\item
$S$ is a consecutive segment of $\{(k,j):1\le j\le m-2\}$.
\end{enumerate}
\end{lem}

\begin{proof}
The proof is given in Appendix~\ref{app:technical-counts}.
\end{proof}
\begin{thm}\label{thm:kmnpal3}
The antichain polynomial $\mathcal N_{[k]\times[m]\times[n]}(x)$ is not palindromic for $3\le k\le m\le n$.
\end{thm}
\begin{proof}
By Lemma~\ref{lem:tight-set-classification-zero}, the upper order ideals of $[k]\times[m]$ are in bijection with lattice paths consisting of $k$ horizontal steps and $m$ vertical steps. 
Hence there are $\binom{k+m}{k}$ such ideals.
Excluding the empty ideal and $\{(k,m)\}$, and then counting the consecutive segments in the last two families, we obtain
\begin{equation}
	\label{eq:three-chain-tight-zero}
	T_{k,m,0}
	=
	\binom{k+m}{k}-2
	+\binom{k-1}{2}
	+\binom{m-1}{2}.
\end{equation}

We now compare the linear and the next-to-leading coefficient of $\mathcal N_P(x)$.
Recall that $P=[k]\times[m]\times[n]$ and that $d=|P_r|=km-s(s+1)$.
Since every one-element subset of $P$ is an antichain, by \eqref{eq:three-chain-monicity-parameters}, we have
\begin{equation}
	\label{eq:three-chain-linear-coefficient}
	\bigl[x\bigr]\mathcal N_P(x)
	=
	|P|
	=
	kmn
	=
	km(m+k-1-2s).
\end{equation}
We next determine the next-to-leading coefficient.
The poset $P=[k]\times[m]\times[n]$ is self-dual under the order-reversing involution $(i,j,\ell)\mapsto(k+1-i,m+1-j,n+1-\ell)$,
which fixes the maximum antichain $L=P_r$ and interchanges the ranks below and above $P_r$. 
Since the tight subsets of $M_{k,m,s}$ enumerate the defect-one antichains below $P_r$, Corollary~\ref{cor:defect-one-coefficient} and Theorem~\ref{thm:self-dual-defect} give
\begin{equation}
	\label{eq:three-chain-second-leading-coefficient}
	\bigl[x^{d-1}\bigr]\mathcal N_P(x)=d+2T_{k,m,s}.
\end{equation}
If $s\ge1$, substituting~\eqref{eq:three-chain-tight-positive} into~\eqref{eq:three-chain-second-leading-coefficient} gives
\[
\begin{aligned}
	\bigl[x^{d-1}\bigr]\mathcal N_P(x)
	&=
	km-s(s+1)
	+2\left(
	1+\binom{s+1}{2}
	+\binom{k-s}{2}
	+\binom{m-s}{2}
	\right)\\
	&=
	km+(k-s)(k-s-1)+(m-s)(m-s-1)+2\\
	&\le km+(k-1)(k-2)+(m-1)(m-2)+2.
\end{aligned}
\]
Because $n\ge m$, \eqref{eq:three-chain-linear-coefficient} gives $\bigl[x\bigr]\mathcal N_P(x)=kmn\ge km^2$.
Therefore, using $3\le k\le m$, we obtain
\[
\begin{aligned}
	\bigl[x\bigr]\mathcal N_P(x)
	-
	\bigl[x^{d-1}\bigr]\mathcal N_P(x)
	&\ge
	km(m-1)
	-(k-1)(k-2)
	-(m-1)(m-2)-2\\
	&\ge
	3m(m-1)-2(m-1)(m-2)-2\\
	&=
	m^2+3m-6>0.
\end{aligned}
\]

If $s=0$, then $d=km$ and~\eqref{eq:three-chain-linear-coefficient} gives $\bigl[x\bigr]\mathcal N_P(x)=km(k+m-1)$.
Substituting~\eqref{eq:three-chain-tight-zero} into~\eqref{eq:three-chain-second-leading-coefficient}, we obtain
\[
\begin{aligned}
	\bigl[x^{km-1}\bigr]\mathcal N_P(x)
	&=
	km
	+2\left(
	\binom{k+m}{k}-2
	+\binom{k-1}{2}
	+\binom{m-1}{2}
	\right)\\
	&=
	km+2\binom{k+m}{k}-4
	+(k-1)(k-2)+(m-1)(m-2).
\end{aligned}
\]
Since $k,m\ge3$, the unimodality of the binomial coefficients gives $\binom{k+m}{k}\ge\binom{k+m}{3}$.
Moreover,
\[
(k+m)(k+m-1)-3km
=
(m-k)^2+(k-1)(m-1)-1>0.
\]
Thus $(k+m)(k+m-1)> 3km$.
Multiplying this inequality by $(k+m-2)/3$, we obtain
\[
2\binom{k+m}{k}
\ge
2\binom{k+m}{3}
=
\frac{(k+m)(k+m-1)(k+m-2)}{3}
>
km(k+m-2).
\]
Since $(k-1)(k-2)+(m-1)(m-2)\ge4$, we conclude that
\[
\bigl[x^{km-1}\bigr]\mathcal N_P(x)
>
km+km(k+m-2)
=
km(k+m-1)
=
\bigl[x\bigr]\mathcal N_P(x).
\]
Thus, in both cases, the coefficient of $x$ differs from the second-highest coefficient. This prove theorem.	
\end{proof}

\subsection{Explicit antichain polynomial for \texorpdfstring{$k\le2$}{k at most 2}}
\label{subsec:k-at-most-two}
We now turn to the cases $k=1$ and $k=2$. 
We begin with $k=1$, whose enumeration will also be used for $k=2$.
Let $A$ be an $r$-element antichain in $[m]\times[n]$.
Since two elements with the same first coordinate are comparable, their first coordinates are distinct. 
We write
\[
A=\{(i_1,j_1),\ldots,(i_r,j_r)\},
\qquad
i_1<\cdots<i_r.
\]
The antichain condition forces $j_1>\cdots>j_r$.
Conversely, every set of this form is an antichain.
Thus the number of $r$-element antichains is $\binom{m}{r}\binom{n}{r}$.
Consequently,
\begin{equation}\label{eq:N1mnx}
\mathcal N_{[1]\times[m]\times[n]}(x)
=
\mathcal N_{[m]\times[n]}(x)
=
\sum_{r=0}^{m}\binom{m}{r}\binom{n}{r}x^r.
\end{equation}
We next consider $[2]\times[m]\times[n]$.
For an antichain
$A\subseteq[2]\times[m]\times[n]$, write
\[
A_1=A\cap\bigl(\{1\}\times[m]\times[n]\bigr),
\qquad
A_2=A\cap\bigl(\{2\}\times[m]\times[n]\bigr).
\]
Define
\begin{equation}
\label{eq:layer-refined-antichain-polynomial}
\mathcal N_{m,n}(u,v)
:=
\sum_{A\in\Ant([2]\times[m]\times[n])}
u^{|A_1|}v^{|A_2|}.
\end{equation}
Thus $\bigl[u^pv^q\bigr]\mathcal N_{m,n}(u,v)$ is the number of antichains having $p$ elements in the first layer and $q$ elements in the second layer.
Moreover, $\mathcal N_{m,n}(x,x)=\mathcal N_{[2]\times[m]\times[n]}(x)$.

The following theorem gives an explicit formula for the number of antichains with prescribed layer sizes. We state it first to motivate the auxiliary lemmas needed for its proof.
\begin{thm}
\label{thm:fixed-layer-enumeration}
For all nonnegative integers $p$ and $q$,
\begin{equation}
\label{eq:fixed-layer-enumeration}
\begin{aligned}
\bigl[u^pv^q\bigr]\mathcal N_{m,n}(u,v)
&=
\binom{m}{p}\binom{n}{p}
\binom{m}{q}\binom{n}{q}\\
&\quad-
\binom{m+1}{p+1}\binom{n-1}{p-1}
\binom{m-1}{q-1}\binom{n+1}{q+1}.
\end{aligned}
\end{equation}
Here $\binom ab=0$ whenever $b<0$ or $b>a$.
\end{thm}

In order to prove Theorem~\ref{thm:fixed-layer-enumeration}, we establish a bijection between antichains and labeled Dyck paths. 
We first explain how a Dyck path determines two auxiliary $D/U$-words and how the two label sequences are assigned to these words. 
Let $U=(1,1)$ and $D=(1,-1)$. 
A Dyck path is a path starting at $(0,0)$, ending on the $x$-axis, and never going below it.
A peak is an occurrence of $UD$, and a valley is an occurrence of $DU$.

For integers $p\ge1$ and $q\ge0$, let $\mathcal D_{p,q}$ be
the set of Dyck paths of length $2p+2q+2$ with exactly $q+1$ peaks.
In particular, $\mathcal D_{p,0}$ consists of the unique path
$U^{p+1}D^{p+1}$.
Each $P\in\mathcal D_{p,q}$ has a unique decomposition $P=U^{b_0+1}D^{a_0+1}\cdots U^{b_q+1}D^{a_q+1}$, where 
\begin{equation}\label{eq:aibi}
\sum_{i=0}^{q}a_i=\sum_{i=0}^{q}b_i=p
\end{equation}
and all $a_i,b_i$ are nonnegative.

An $(m,n)$-labeling of Dyck path $P$ consists of two weakly increasing sequences
\[
1\le\xi_1\le\cdots\le\xi_{p+q}\le m,
\qquad
1\le\eta_1\le\cdots\le\eta_{p+q}\le n.
\]
Place $\xi_1,\ldots,\xi_{p+q}$, from left to right, under the letters of $D^{a_0}UD^{a_1}U\cdots UD^{a_q}$.
The labels must increase strictly, except that the two labels under a consecutive pair $DU$ may be equal.
Similarly, place $\eta_1,\ldots,\eta_{p+q}$ under the letters of $D^{b_0}UD^{b_1}U\cdots UD^{b_q}$.

Again the labels must increase strictly, except that the two labels under a consecutive pair $UD$ may be equal.
The following example illustrates how the two label sequences are assigned to a Dyck path.

\begin{exm}
Consider the Dyck path
\[
P
=
U^3D^3U^2DU^2D^3U^2D^2
\in\mathcal D_{5,3},
\]
where $(a_0,a_1,a_2,a_3)=(2,0,2,1)$ and $(b_0,b_1,b_2,b_3)=(2,1,1,1)$.
Take
\[
\xi=(1,2,2,3,4,5,5,7),
\qquad
\eta=(1,2,3,3,4,4,5,5).
\]

The two labelings are displayed as follows:
\[
\begin{array}{c|cccccccc}
&D&D&U&U&D&D&U&D\\ \hline
\xi&1&2&2&3&4&5&5&7
\end{array}
\qquad
\begin{array}{c|cccccccc}
&D&D&U&D&U&D&U&D\\ \hline
\eta&1&2&3&3&4&4&5&5
\end{array}.
\]
In the first array, equal consecutive labels occur only under the pairs $DU$ in positions $2,3$ and $6,7$. 
In the second array, they occur only under the pairs $UD$ in positions $3,4$, $5,6$, and $7,8$. 
Hence $(\xi,\eta)$ is a $(7,6)$-labeling of $P$.
\end{exm}
We next give some lemmas used in proof of Theorem~\ref{thm:fixed-layer-enumeration} latter.
\begin{lem}
\label{lem:labeled-dyck-bijection}
The antichains $A\subseteq[2]\times[m]\times[n]$ satisfying $|A\cap\bigl(\{1\}\times[m]\times[n]\bigr)|=p$ and $|A\cap\bigl(\{2\}\times[m]\times[n]\bigr)|=q$ are in bijection with the $(m,n)$-labeled paths in $\mathcal D_{p,q}$.
\end{lem}

\begin{proof}
Let $A$ be an antichain satisfying the conditions of the lemma, and set
\[
A_1
=
A\cap\bigl(\{1\}\times[m]\times[n]\bigr),
\qquad
A_2
=
A\cap\bigl(\{2\}\times[m]\times[n]\bigr).
\]
Then $|A_1|=p$ and $|A_2|=q$.
Each of $A_1$ and $A_2$ is an antichain in a copy of $[m]\times[n]$. 
We write
\[
A_1
=
\{(1,x_i,y_i):1\le i\le p\},
\qquad
x_1<\cdots<x_p,
\qquad
y_1>\cdots>y_p,
\]
\[
A_2
=
\{(2,x'_j,y'_j):1\le j\le q\},
\qquad
x'_1<\cdots<x'_q,
\qquad
y'_1>\cdots>y'_q.
\]
Merge the increasing sequences $x_1<\cdots<x_p$ and $x'_1<\cdots<x'_q$ into a single nondecreasing sequence. Record entries from the first sequence by $D$ and entries from the second sequence by $U$. 
Whenever $x_i=x'_j$, place $x_i$ before $x'_j$, so that the tie is recorded by $DU$.

Let $\alpha$ be the resulting path, and let $\xi_1\le\cdots\le\xi_{p+q}$ be the corresponding merged sequence of first coordinates.
The path $\alpha$ has $p$ down-steps and $q$ up-steps, and hence runs from $(0,0)$ to $(p+q,q-p)$. 
Moreover, $1\le\xi_1\le\cdots\le\xi_{p+q}\le m$.
Since both $(x_i)$ and $(x'_j)$ are strictly increasing, equal consecutive terms in the merged sequence can arise only from a tie $x_i=x'_j$.
Such a tie is recorded by $DU$, and therefore occurs at a valley of $\alpha$.

We treat the second coordinates similarly. 
Set $z_i=n+1-y_i$ for all $1\le i\le p$, and $z'_j=n+1-y'_j$ for all $1\le j\le q$.
The sequences $(z_i)$ and $(z'_j)$ are strictly increasing.
Merge them into a single nondecreasing sequence, recording entries from the first sequence by $D$ and entries from the second sequence by $U$.
Whenever $z_i=z'_j$, place $z'_j$ before $z_i$, so that the tie is recorded by $UD$.

Let $\beta$ be the resulting path, and let $\eta_1\le\cdots\le\eta_{p+q}$ be the merged sequence. 
Then $\beta$ runs from $(0,0)$ to $(p+q,q-p)$, and $1\le\eta_1\le\cdots\le\eta_{p+q}\le n$.
Equal consecutive terms can arise only from a tie $z_i=z'_j$. 
Such a tie is recorded by $UD$, and therefore occurs at a peak of $\beta$.

We next translate the antichain condition into a relation between $\alpha$ and $\beta$.
For $1\le j\le q$, set $R_j:=|\{i:x_i\le x'_j\}|$ and $S_j:=|\{i:z_i<z'_j\}|$.
Since $(x_i)$ is strictly increasing, the condition $x_i\le x'_j$ holds precisely for $1\le i\le R_j$. 
Similarly, the condition $z_i\ge z'_j$ holds precisely for $S_j<i\le p$.
Since $y_i\le y'_j$ is equivalent to $z_i\ge z'_j$, we have $(1,x_i,y_i)\le(2,x'_j,y'_j)$ for some $i$ if and only if $S_j<R_j$.
Thus
\begin{equation}
\label{eq:labeled-path-antichain-condition}
A\text{ is an antichain}
\quad\Longleftrightarrow\quad
R_j\le S_j
\quad\text{for every }1\le j\le q.
\end{equation}

By the two tie-breaking rules, $R_j$ is the number of down-steps preceding the $j$th up-step of $\alpha$, while $S_j$ is the number of down-steps preceding the $j$th up-step of $\beta$.
Therefore $R_j\le S_j$ for every $j$ if and only if $\alpha$ lies weakly above $\beta$.

Write $\alpha=D^{a_0}UD^{a_1}U\cdots UD^{a_q}$ and $\beta=D^{b_0}UD^{b_1}U\cdots UD^{b_q}$.
Associate with $A$ the path
\[
P
=
U^{b_0+1}D^{a_0+1}
\cdots
U^{b_q+1}D^{a_q+1}.
\]
For $0\le k<q$, the height of $P$ at the end of its $(k+1)$st down-run is
\[
\sum_{i=0}^{k}(b_i-a_i)
=
S_{k+1}-R_{k+1}
\ge0.
\]
According to \eqref{eq:aibi}, the height at the end of the final down-run is $0$.
The minimum height within each down-run occurs at its endpoint.
Therefore $P$ never goes below the $x$-axis. 
It has length $2p+2q+2$ and exactly $q+1$ peaks, so $P\in\mathcal D_{p,q}$.

The two words associated with $P\in\mathcal D_{p,q}$ are precisely $\alpha$ and $\beta$.
The preceding observations show that $(\xi,\eta)$ satisfies the labeling conditions. 
Hence the construction sends $A$ to an $(m,n)$-labeled path $(P,\xi,\eta)$.

Conversely, let $P\in\mathcal D_{p,q}$ be equipped with an $(m,n)$-labeling $(\xi,\eta)$. 
Its run decomposition uniquely determines
\[
\alpha
=
D^{a_0}UD^{a_1}U\cdots UD^{a_q},
\qquad
\beta
=
D^{b_0}UD^{b_1}U\cdots UD^{b_q}.
\]
Split the $\xi$-labels according to the steps of $\alpha$.
The labels under the down-steps give $x_1,\ldots,x_p$, and those under the up-steps give $x'_1,\ldots,x'_q$. 
Equality is allowed only at valleys of $\alpha$, where a down-step is followed by an up-step. 
It follows that
\[
x_1<\cdots<x_p,
\qquad
x'_1<\cdots<x'_q.
\]
Similarly, split the $\eta$-labels according to the steps of $\beta$.
The labels under the down-steps give $z_1,\ldots,z_p$, and those under the up-steps give $z'_1,\ldots,z'_q$. 
Since equality is allowed only at peaks of $\beta$, we obtain $z_1<\cdots<z_p$ and $z'_1<\cdots<z'_q$.
Set $y_i=n+1-z_i$ and $y'_j=n+1-z'_j$.
Then
\[
y_1>\cdots>y_p,
\qquad
y'_1>\cdots>y'_q.
\]
These four coordinate sequences determine $A_1$ and $A_2$ uniquely.

Since $P$ is a Dyck path, for every $0\le k<q$, all the partial sums $\sum_{i=0}^{k}(b_i-a_i)$ are nonnegative.
Hence $\alpha$ lies weakly above $\beta$, so $R_j\le S_j$ for every $j$. 
By~\eqref{eq:labeled-path-antichain-condition}, no element of $A_1$ lies below an element of $A_2$.
Each layer is itself an antichain, and hence $A_1\cup A_2$ is an antichain. 
The two constructions are inverse, which proves the claimed bijection.
\end{proof}

The following example illustrates the bijection by explicitly recovering an antichain from a labeled Dyck path.

\begin{exm}
\label{exm:labeled-dyck-path}
For the path $P$ in the preceding example,
\[
P
=
U^3D^3U^2DU^2D^3U^2D^2
\in\mathcal D_{5,3},
\]
take $\xi=(1,2,2,3,4,5,5,7)$ and $\eta=(1,2,3,3,4,4,5,5)$.
The resulting labeled words are
\[
\begin{array}{c|cccccccc}
\alpha_P&D&D&U&U&D&D&U&D\\ \hline
\xi&1&2&2&3&4&5&5&7
\end{array}
\qquad
\begin{array}{c|cccccccc}
\beta_P&D&D&U&D&U&D&U&D\\ \hline
\eta&1&2&3&3&4&4&5&5
\end{array}.
\]

The labels under the down-steps of $\alpha_P$ and $\beta_P$ give $(x_1,\ldots,x_5)=(1,2,4,5,7)$ and $(z_1,\ldots,z_5)=(1,2,3,4,5)$,
while the labels under the up-steps give $(x'_1,x'_2,x'_3)=(2,3,5)$ and $(z'_1,z'_2,z'_3)=(3,4,5)$.
Since $n=6$, we recover the second coordinates from $y_i=7-z_i$ and $y'_j=7-z'_j$.
The corresponding antichain in $[2]\times[7]\times[6]$ is therefore
\[
\begin{aligned}
A_1
&=
\{(1,1,6),(1,2,5),(1,4,4),(1,5,3),(1,7,2)\},\\
A_2
&=
\{(2,2,4),(2,3,3),(2,5,2)\}.
\end{aligned}
\]
\end{exm}

We next count the labelings of a fixed path.
For $P\in\mathcal D_{p,q}$, define
\[
\ell_D(P)
:=
|\{i:0\le i<q,\ a_i>0\}|,
\qquad
\ell_U(P)
:=
|\{i:0< i\le q,\ b_i>0\}|.
\]
Thus $\ell_D(P)$ is the number of nonfinal down-runs of $P$ having length at least $2$, while $\ell_U(P)$ is the number of noninitial up-runs having length at least $2$.

\begin{lem}
\label{lem:labeled-dyck-enumeration}
For $p,q\ge1$,
\[
\bigl[u^pv^q\bigr]\mathcal N_{m,n}(u,v)
=
\sum_{P\in\mathcal D_{p,q}}
\binom{m+\ell_D(P)}{p+q}
\binom{n+\ell_U(P)}{p+q}.
\]
\end{lem}

\begin{proof}
Fix $P\in\mathcal D_{p,q}$. 
A valley of $\alpha_P$ occurs precisely for an index $0\le i<q$ with $a_i>0$. 
Thus equality between consecutive $\xi$-labels is allowed at exactly $\ell_D(P)$ positions.

If a valley occurs between the $t$-th and $(t+1)$-st steps of $\alpha_P$, add $1$ to each of $\xi_{t+1},\ldots,\xi_{p+q}$.
This gives a bijection between the possible $\xi$-labelings and the strictly increasing sequences of length $p+q$ in $[m+\ell_D(P)]$. 
Hence there are $\binom{m+\ell_D(P)}{p+q}$ possible $\xi$-labelings.

Similarly, $\beta_P$ has exactly $\ell_U(P)$ peaks, so the possible $\eta$-labelings are counted by $\binom{n+\ell_U(P)}{p+q}$.
The two labelings are independent. 
Summing their product over $P\in\mathcal D_{p,q}$ and applying Lemma~\ref{lem:labeled-dyck-bijection} proves the result.
\end{proof}

To enumerate the Dyck paths according to their long up- and down-runs,
we realize them as distinguished cyclic shifts of unrestricted paths.
We use the following form of the cycle lemma
\cite{DvoretzkyMotzkin1947,Raney1960}.

\begin{lem}[Cycle lemma]\label{lem:cycle-lemma}
Let $v=v_1v_2\cdots v_N$ be a cyclic word in two letters $U$ and $D$, with weights $1$ and $-1$, respectively.
Suppose that the total weight of $v$ is $-1$.
Then there exists a unique index $i\in\{0,1,\ldots,N-1\}$ such that
\[
\sum_{j=1}^k \operatorname{wt}(v_{i+j})\ge0
\qquad(1\le k<N),
\]
where the indices are taken modulo $N$.
\end{lem}

For $p\ge1$, $q\ge0$, and $i,j\ge0$, let $E_{p,q}(i,j)$ denote the number of paths $P\in\mathcal D_{p,q}$ having exactly $i$ up-runs of length at least $2$ and exactly $j$ nonfinal down-runs of length at least $2$.

\begin{lem}
\label{lem:auxiliary-run-enumeration}
For $p\ge1$, $q\ge0$, and $i,j\ge0$,
\[
E_{p,q}(i,j)
=
\frac{1}{j+1}
\binom{q+1}{i}\binom{p-1}{i-1}
\binom{q}{j}\binom{p}{j}.
\]
\end{lem}

\begin{proof}
Let $\mathcal R_{p,q}(i,j)$ be the set of paths $R=U^{b_0+1}D^{a_0+1}\cdots U^{b_q+1}D^{a_q+1}$,
where $\sum_{h=0}^q a_h=\sum_{h=0}^q b_h=p$, having exactly $i$ positive entries among $b_0,\ldots,b_q$ and exactly $j$ positive entries among $a_0,\ldots,a_{q-1}$. 
It is noted that these paths are allowed to go below the $x$-axis.

There are $\binom{q+1}{i}\binom{p-1}{i-1}$ choices for $(b_0,\ldots,b_q)$. 
Indeed, we first choose the $i$ positions occupied
by the positive entries.
Once these positions are fixed, their values form a composition of $p$ into $i$ positive parts, of which there are $\binom{p-1}{i-1}$.

Similarly, there are $\binom{q}{j}\binom{p}{j}$ choices for $(a_0,\ldots,a_q)$. 
Here we choose the $j$ positive entries among $a_0,\ldots,a_{q-1}$, while $a_q$ is allowed to be zero.
Consequently,
\begin{equation}\label{eq:Rpqij}
|\mathcal R_{p,q}(i,j)|
=
\binom{q+1}{i}\binom{p-1}{i-1}
\binom{q}{j}\binom{p}{j}.
\end{equation}
For $R\in\mathcal R_{p,q}(i,j)$, append one down-step to its final down-run, and denote the resulting word by $\widehat R$. 
This word has total height $-1$ and exactly $j+1$ down-runs of length at least $2$.

By Lemma~\ref{lem:cycle-lemma}, $\widehat R$ has a unique cyclic shift all of whose proper initial segments have nonnegative height.
Since its total height is $-1$, this shift begins with $U$, ends with $D$, and has height $0$ immediately before its last step.

The cyclic cut therefore occurs between a down-run and an up-run, so no run is split. 
Deleting the last step gives a Dyck path. 
Since this Dyck path ends with $D$, the deleted step belonged to a down-run of length at least $2$.
The cyclic shift preserves the run lengths, and only the final down-run is shortened. 
The resulting Dyck path therefore has exactly $i$ up-runs of length at least $2$ and exactly $j$ nonfinal down-runs of length at least $2$. 
Hence it is counted by $E_{p,q}(i,j)$.

Conversely, let $P$ be counted by $E_{p,q}(i,j)$ and append one down-step to its final down-run.
Denote the resulting word by $\widehat P$.
The resulting word has exactly $j+1$ down-runs of length at least $2$. 
Choose one of these runs, cut immediately after it, and read the word cyclically. 
The chosen run becomes the final down-run. 
Deleting its last step then gives an element of $\mathcal R_{p,q}(i,j)$.
Since $P$ is a Dyck path, $\widehat P$ is the unique cyclic shift selected by Lemma~\ref{lem:cycle-lemma}.
Hence applying the preceding construction recovers $P$.

The $j+1$ choices produce distinct paths.
If two different choices produced the same path, then two different cyclic shifts of $\widehat P$ would give the same word. 
Hence $\widehat P=W^d$ for some word $W$ and some integer $d\ge2$. 
The total height of $\widehat P$ would then be divisible by $d$, contradicting the fact that it is $-1$.
Therefore we obtain $|\mathcal R_{p,q}(i,j)|=(j+1)E_{p,q}(i,j)$.
Combining this identity with \eqref{eq:Rpqij} proves the result.
\end{proof}

We are now ready to complete the proof of Theorem~\ref{thm:fixed-layer-enumeration}.

\begin{proof}[Proof of Theorem~\ref{thm:fixed-layer-enumeration}]
Assume first that $p,q\ge1$.
Fix $r,s\ge0$, and consider a path
\[
P
=
U^{b_0+1}D^{a_0+1}\cdots
U^{b_q+1}D^{a_q+1}
\in\mathcal D_{p,q}
\]
satisfying $\ell_D(P)=r$ and $\ell_U(P)=s$.

Suppose first that $b_0=0$. 
Since the initial up-run has length $1$ and $P$ is a Dyck path, we must also have $a_0=0$. 
Thus $P$ begins with the peak $UD$. 
Deleting this initial peak gives a path in $\mathcal D_{p,q-1}$ having exactly $s$ up-runs of length at least $2$ and exactly $r$ nonfinal down-runs of length at least $2$.
Consequently, the number of paths in this case is $E_{p,q-1}(s,r)$.

Now suppose that $b_0>0$. 
The initial up-run then has length at least $2$, so $P$ has altogether $s+1$ up-runs of length at least $2$.
There are $E_{p,q}(s+1,r)$ paths with these run statistics. 
Among them, those with $b_0=0$ are counted by $E_{p,q-1}(s+1,r)$ after deletion of their initial peak.
Hence the number of paths in the second case is $E_{p,q}(s+1,r)-E_{p,q-1}(s+1,r)$.

It follows that
\[
\bigl|\{P\in\mathcal D_{p,q}:
\ell_D(P)=r,\ \ell_U(P)=s\}\bigr|
=
E_{p,q-1}(s,r)
+
E_{p,q}(s+1,r)
-
E_{p,q-1}(s+1,r).
\]
Substituting the formula in Lemma~\ref{lem:auxiliary-run-enumeration} and simplifying with Pascal's identity and the standard relations between adjacent binomial coefficients, we obtain
\begin{equation}
\label{eq:dyck-run-distribution}
\begin{aligned}
&\bigl|\{P\in\mathcal D_{p,q}:
\ell_D(P)=r,\ \ell_U(P)=s\}\bigr|\\
=&
\binom{p}{r}\binom{q}{r}
\binom{p}{s}\binom{q}{s}
-
\binom{q+1}{r+1}\binom{p-1}{r-1}
\binom{q-1}{s-1}\binom{p+1}{s+1}.
\end{aligned}
\end{equation}

By Lemma~\ref{lem:labeled-dyck-enumeration}, grouping the paths according to $\ell_D(P)$ and $\ell_U(P)$ gives
\[
\begin{aligned}
\bigl[u^pv^q\bigr]\mathcal N_{m,n}(u,v)
=
\sum_{r,s\ge0}
&\bigl|\{P\in\mathcal D_{p,q}:
\ell_D(P)=r,\ \ell_U(P)=s\}\bigr|
\binom{m+r}{p+q}\binom{n+s}{p+q}.
\end{aligned}
\]
Substituting~\eqref{eq:dyck-run-distribution} and separating the $r$- and $s$-sums, we obtain
\[
\begin{aligned}
\bigl[u^pv^q\bigr]\mathcal N_{m,n}(u,v)
={}&
\left(
\sum_{r\ge0}
\binom{p}{r}\binom{q}{r}
\binom{m+r}{p+q}
\right)
\left(
\sum_{s\ge0}
\binom{p}{s}\binom{q}{s}
\binom{n+s}{p+q}
\right)\\
&-
\left(
\sum_{r\ge0}
\binom{q+1}{r+1}\binom{p-1}{r-1}
\binom{m+r}{p+q}
\right)
\left(
\sum_{s\ge0}
\binom{p+1}{s+1}\binom{q-1}{s-1}
\binom{n+s}{p+q}
\right).
\end{aligned}
\]

We use the following two Vandermonde-type identities:
\[
\begin{aligned}
\sum_{h\ge0}
\binom{a}{h}\binom{b}{h}
\binom{M+h}{a+b}
&=
\binom{M}{a}\binom{M}{b},\\
\sum_{h\ge0}
\binom{a+1}{h+1}\binom{b-1}{h-1}
\binom{M+h}{a+b}
&=
\binom{M-1}{a-1}\binom{M+1}{b+1}.
\end{aligned}
\]

Apply the first identity with $(a,b,M)=(p,q,m)$ and $(p,q,n)$, and the second with $(a,b,M)=(q,p,m)$ and $(p,q,n)$. 
This gives~\eqref{eq:fixed-layer-enumeration}.
If $q=0$, the antichain lies entirely in the first layer, giving $\binom{m}{p}\binom{n}{p}$ possibilities. 
Similarly, if $p=0$, there are $\binom{m}{q}\binom{n}{q}$ possibilities. 
Thus the formula also holds in these cases.
\end{proof}

Equation~\eqref{eq:fixed-layer-enumeration} admits the following factorization.
Set
\begin{equation}
\label{eq:Qmn-definition}
Q_{m,n}(x)
:=
\sum_{r\ge0}
\frac{1}{r+1}\binom{m}{r}\binom{n}{r}x^r,
\end{equation}
and define
\begin{equation}
\label{eq:AMmn-definitions}
M_{m,n}(x):=xQ_{m,n}'(x),
\qquad
A_{m,n}(x):=Q_{m,n}(x)+xQ_{m,n}'(x).
\end{equation}

\begin{cor}
\label{cor:layer-refined-factorization}

For all positive integers $m$ and $n$,
\begin{equation}
\label{eq:layer-refined-factorization}
\mathcal N_{m,n}(u,v)
=
A_{m,n}(u)A_{m,n}(v)
-
\frac{(m+1)(n+1)}{mn}
M_{m,n}(u)M_{m,n}(v).
\end{equation}
\end{cor}

\begin{proof}
From~\eqref{eq:Qmn-definition} and~\eqref{eq:AMmn-definitions},
\begin{equation}
\label{eq:AMmn-coefficients}
\begin{aligned}
A_{m,n}(x)
&=
\sum_{r\ge0}\binom{m}{r}\binom{n}{r}x^r,\\
M_{m,n}(x)
&=
\sum_{r\ge0}
\frac{r}{r+1}
\binom{m}{r}\binom{n}{r}x^r.
\end{aligned}
\end{equation}

By Theorem~\ref{thm:fixed-layer-enumeration}, we obtain
\[
\begin{aligned}
\mathcal N_{m,n}(u,v)
&=
\sum_{p,q\ge0}
\Biggl[
\binom{m}{p}\binom{n}{p}
\binom{m}{q}\binom{n}{q}-
\binom{m+1}{p+1}\binom{n-1}{p-1}
\binom{m-1}{q-1}\binom{n+1}{q+1}
\Biggr]u^pv^q\\
&=
\left(
\sum_{p\ge0}
\binom{m}{p}\binom{n}{p}u^p
\right)
\left(
\sum_{q\ge0}
\binom{m}{q}\binom{n}{q}v^q
\right)\\
&\quad-
\frac{(m+1)(n+1)}{mn}
\left(
\sum_{p\ge0}
\frac{p}{p+1}
\binom{m}{p}\binom{n}{p}u^p
\right)
\left(
\sum_{q\ge0}
\frac{q}{q+1}
\binom{m}{q}\binom{n}{q}v^q
\right)\\
&=
A_{m,n}(u)A_{m,n}(v)
-
\frac{(m+1)(n+1)}{mn}
M_{m,n}(u)M_{m,n}(v),
\end{aligned}
\]
which is~\eqref{eq:layer-refined-factorization}.
\end{proof}

\subsection{Application to Zhang-Zhang polynomials}
\label{subsec:ZZ-polynomials}
We now apply Corollary~\ref{cor:layer-refined-factorization} to prove the determinantal formula in Theorem~\ref{thm:HLW-determinant}.

A \defi{benzenoid graph} is a finite $2$-connected plane graph formed by a simply connected union of cells in the hexagonal lattice.
A \defi{Clar cover} of a benzenoid graph $G$ is a spanning subgraph whose connected components are either single edges or hexagonal faces.
Let $z(G,r)$ be the number of Clar covers having exactly $r$ hexagonal components. 
Following Zhang and Zhang~\cite{ZhangZhang1996}, the \defi{Zhang-Zhang polynomial}, or the \defi{ZZ polynomial}, of $G$ is
\[
ZZ(G,x)=\sum_{r\ge0}z(G,r)x^r.
\]
For positive integers $\ell,m,n$, let $O(\ell,m,n)$ denote the hexagonal flake whose successive side lengths are $\ell,m,n,\ell,m,n$.
For example, $O(1,2,3)$ is the $2\times3$ parallelogram benzenoid,
while $O(2,2,2)$ is coronene; see Figure~\ref{fig:hexagonal-flakes}.

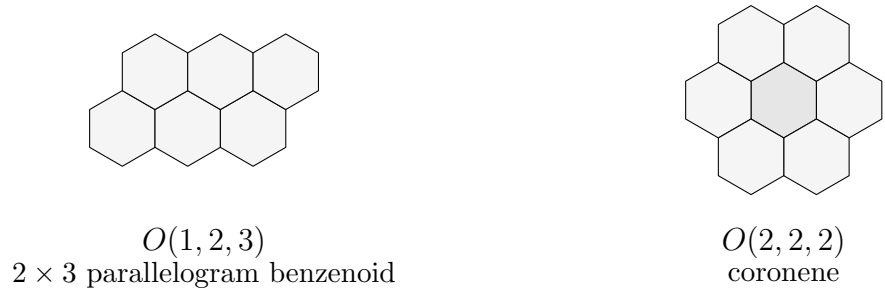
\begin{figure}[htbp]

\centering

\begin{tikzpicture}[
line join=round,
line cap=round,
scale=1.25
]

\newcommand{\hexcell}[2]{%
	\path[draw,fill=gray!8]
	(#1,{#2+.40}) --
	({#1+.346},{#2+.20}) --
	({#1+.346},{#2-.20}) --
	(#1,{#2-.40}) --
	({#1-.346},{#2-.20}) --
	({#1-.346},{#2+.20}) -- cycle;
}

% O(1,2,3): a 2 by 3 parallelogram benzenoid
\begin{scope}[xshift=-4cm]
	
	% The two rows are centered about y=0.
	\foreach \x/\y in {
		0/-.30,
		.692/-.30,
		1.384/-.30,
		.346/.30,
		1.038/.30,
		1.730/.30
	}{
		\hexcell{\x}{\y}
	}
	
	\node at (.865,-1.50) {$O(1,2,3)$};
	\node[font=\small] at (.865,-1.85)
	{$2\times3$ parallelogram benzenoid};
	
\end{scope}

% O(2,2,2): coronene
\begin{scope}[xshift=3cm]
	
	% Central hexagon
	\path[draw,fill=gray!20]
	(0,.40) --
	(.346,.20) --
	(.346,-.20) --
	(0,-.40) --
	(-.346,-.20) --
	(-.346,.20) -- cycle;
	
	% Six surrounding hexagons
	\foreach \x/\y in {
		.692/0,
		-.692/0,
		.346/.60,
		-.346/.60,
		.346/-.60,
		-.346/-.60
	}{
		\hexcell{\x}{\y}
	}
	
	\node at (0,-1.50) {$O(2,2,2)$};
	\node[font=\small] at (0,-1.85)
	{coronene};
	
\end{scope}

\end{tikzpicture}

\caption{Two examples of hexagonal flakes. 
The flake $O(1,2,3)$ is a $2\times3$ parallelogram benzenoid, while $O(2,2,2)$ is coronene.}

\label{fig:hexagonal-flakes}
\end{figure}

For the flake $O(\ell,m,n)$, the poset occurring in the construction of Langner and Witek is isomorphic to $[\ell]\times[m]$.
Specializing~\cite[Theorem~24]{LangnerWitek2022} to this flake gives
\[
ZZ(O(\ell,m,n),x)
=
\sum_{A\subseteq[\ell]\times[m]}
\Omega_A^\circ(n)(1+x)^{|A|},
\]
where $\Omega_A^\circ(n)$ denotes the number of strictly order-preserving maps from the induced subposet $A$ to $[n]$.
We now give the bijection relating strictly order-preserving maps to antichains in a product with a chain. 
Set $Q=[\ell]\times[m]$.
Fix a subset $A\subseteq Q$ and a strictly order-preserving map $\varphi:A\to[n]$. 
Define
\[
\mathcal B(A,\varphi)
:=
\bigl\{
(i,j,n+1-\varphi(i,j)):(i,j)\in A
\bigr\}.
\]

We claim that $\mathcal B(A,\varphi)$ is an antichain in $Q\times[n]$. 
Consider two distinct elements $(i,j),(i',j')\in A$.
If they are incomparable in $Q$, then the corresponding elements of $\mathcal B(A,\varphi)$ are also incomparable. 
If $(i,j)<(i',j')$ in $Q$, then the strict order-preservation of $\varphi$ gives $\varphi(i,j)<\varphi(i',j')$.
Consequently, $n+1-\varphi(i,j)>n+1-\varphi(i',j')$.
Thus the first two coordinates increase while the third coordinate decreases, so $\mathcal B(A,\varphi)$ is an antichain.

Conversely, let $\mathcal B$ be an antichain in $Q\times[n]$, and let $A$ be its projection onto $Q$. This projection is injective. 
Because if $(i,j,k)$ and $(i,j,k')$ were two distinct elements of $\mathcal B$ with the same first two coordinates, then either $k<k'$ or $k>k'$, and the two elements would be comparable in $Q\times[n]$.

It follows that, for every $(i,j)\in A$, there is a unique $f(i,j)\in[n]$ such that $(i,j,f(i,j))\in\mathcal B$.
If $(i,j)<(i',j')$ in $A$, then $f(i,j)>f(i',j')$.
Indeed, if $f(i,j)\le f(i',j')$, then $(i,j,f(i,j))\le(i',j',f(i',j'))$, contradicting the assumption that $\mathcal B$ is an antichain.
Therefore the map $\varphi:A\to[n]$ defined by $\varphi(i,j)=n+1-f(i,j)$ is strictly order-preserving.

These two constructions are inverse to each other and preserve cardinality. 
Thus, for each $A\subseteq Q$, the antichains in $Q\times[n]$ whose projection onto $Q$ is $A$ are counted by $\Omega_A^\circ(n)$. 
Consequently,
\[
\mathcal N_{Q\times[n]}(z)
=
\sum_{A\subseteq Q}
\Omega_A^\circ(n)z^{|A|}.
\]
Taking $Q=[\ell]\times[m]$ and $z=x+1$ gives
\[
\sum_{A\subseteq[\ell]\times[m]}
\Omega_A^\circ(n)(1+x)^{|A|}
=
\mathcal N_{[\ell]\times[m]\times[n]}(x+1).
\]
This correspondence preserves cardinality. 
Therefore
\begin{equation}
\label{eq:ZZ-antichain}
ZZ(O(\ell,m,n),x)
=
\mathcal N_{[\ell]\times[m]\times[n]}(x+1).
\end{equation}
We now specialize~\eqref{eq:ZZ-antichain} to $\ell=2$ to prove Theorem~\ref{thm:HLW-determinant}.

\begin{proof}[Proof of Theorem~\ref{thm:HLW-determinant}]
A direct index shift in~\eqref{eq:HLW-entries}, together
with~\eqref{eq:AMmn-coefficients}, gives
\[
J_{m,n}(x)=A_{m,n}(x+1),
\quad
K_{m,n}(x)=\frac{n+1}{m}M_{m,n}(x+1),
\quad
K_{n,m}(x)=\frac{m+1}{n}M_{m,n}(x+1).
\]
Therefore, by~\eqref{eq:ZZ-antichain}
and~\eqref{eq:layer-refined-factorization},
\[
\begin{aligned}
ZZ(O(2,m,n),x)
&=
\mathcal N_{[2]\times[m]\times[n]}(x+1)\\
&=
A_{m,n}(x+1)^2
-
\frac{(m+1)(n+1)}{mn}
M_{m,n}(x+1)^2\\
&=
J_{m,n}(x)^2-K_{m,n}(x)K_{n,m}(x)\\
&=
\det
\begin{pmatrix}
J_{m,n}(x) & K_{m,n}(x)\\
K_{n,m}(x) & J_{m,n}(x)
\end{pmatrix}.
\end{aligned}
\]
\end{proof}

\subsection{Palindromicity for \texorpdfstring{$k\le2$}{k at most 2}}
\label{subsec:small-k-palindromicity}

We now determine the palindromic cases when the shortest chain has cardinality at most two.
Together with Subsection~\ref{subsec:three-chain-palindromicity}, this gives the complete palindromicity classification for products of three chains.

\begin{lem}
\label{lem:small-k-palindromicity}
Let $1\le m\le n$.
\begin{enumerate}[label=\textup{(\roman*)}]
\item
The polynomial $\mathcal N_{[1]\times[m]\times[n]}(x)$ is palindromic if and only if $n=m$.

\item
The polynomial $\mathcal N_{[2]\times[m]\times[n]}(x)$ is palindromic if and only if $n=m+1$.
\end{enumerate}
\end{lem}

\begin{proof}
For $k=1$, \eqref{eq:N1mnx} shows the constant coefficient of $\mathcal N_{[1]\times[m]\times[n]}(x)$ is $1$, while its leading coefficient is $\binom nm$.
Hence palindromicity forces $n=m$.
Conversely, when $n=m$, it is easy to verify $\mathcal N_{[1]\times[m]\times[m]}(x)$ is palindromic.

We now consider $k=2$. 
By~\eqref{eq:layer-refined-factorization},
\begin{equation}
\label{eq:k-two-leading-coefficient}
\begin{aligned}
\bigl[x^{2m}\bigr]
\mathcal N_{[2]\times[m]\times[n]}(x)
&=
\binom nm^2
-
\frac{(m+1)(n+1)}{mn}
\left(
\frac{m}{m+1}\binom nm
\right)^2\\
&=
\frac{1}{m+1}
\binom nm
\binom{n-1}{m}.
\end{aligned}
\end{equation}

If $n=m$, then~\eqref{eq:k-two-leading-coefficient} shows that the
coefficient of $x^{2m}$ vanishes, while a direct coefficient extraction
from~\eqref{eq:layer-refined-factorization} gives $\bigl[x^{2m-1}\bigr]\mathcal N_{[2]\times[m]\times[m]}(x)=2$.
Hence the polynomial has degree $2m-1$ and leading coefficient $2$.
Since its constant coefficient is $1$, it is not palindromic.

Now suppose that $n>m$. 
Then~\eqref{eq:k-two-leading-coefficient} is positive, so the polynomial has degree $2m$.
Palindromicity therefore requires its right-hand side to equal $1$.
If $n=m+1$, the right-hand side of~\eqref{eq:k-two-leading-coefficient} is
\[
\frac{1}{m+1}
\binom{m+1}{m}\binom mm
=
1.
\]
If $n\ge m+2$, then
\[
\frac{1}{m+1}
\binom nm
\binom{n-1}{m}
\ge
\binom nm
>1.
\]
Thus palindromicity is possible only when $n=m+1$.

It remains to prove that $n=m+1$ is sufficient.
Write $Q_{m,m+1}(x)=\sum_{r=0}^{m}q_rx^r$.
By definition,
\[
q_r
=
\frac{1}{r+1}
\binom mr\binom{m+1}{r}
=
\frac{m+1}{(r+1)(m+1-r)}
\binom mr^2.
\]
Hence
\begin{equation}
\label{eq:Q-reciprocity}
x^mQ_{m,m+1}(x^{-1})=Q_{m,m+1}(x).
\end{equation}

Differentiating~\eqref{eq:Q-reciprocity} and using~\eqref{eq:Q-reciprocity} again to eliminate 
$Q_{m,m+1}(x^{-1})$, we obtain
\begin{equation}
\label{eq:Qprime-reciprocity}
x^{m-1}Q_{m,m+1}'(x^{-1})
=
mQ_{m,m+1}(x)-xQ_{m,m+1}'(x).
\end{equation}
Now put $\lambda=\sqrt{\frac{m+2}{m}}$ and $c_\pm=1\pm\lambda$,
and define
\[
H_\pm(x)=Q_{m,m+1}(x)+c_\pm xQ_{m,m+1}'(x).
\]
By~\eqref{eq:layer-refined-factorization}, we obtain
\[
\mathcal N_{[2]\times[m]\times[m+1]}(x)
=
H_-(x)H_+(x).
\]
Using~\eqref{eq:Q-reciprocity} and~\eqref{eq:Qprime-reciprocity}, we have
\[
\begin{aligned}
x^mH_\pm(x^{-1})
&=
Q_{m,m+1}(x)
+
c_\pm\bigl(mQ_{m,m+1}(x)-xQ_{m,m+1}'(x)\bigr)\\
&=
(1+mc_\pm)Q_{m,m+1}(x)-c_\pm xQ_{m,m+1}'(x).
\end{aligned}
\]
Since $\lambda^2=\frac{m+2}{m}$, a direct calculation gives $-c_+=(1+mc_+)c_-$, $-c_-=(1+mc_-)c_+$,
and $(1+mc_+)(1+mc_-)=1$.
Therefore
\[
x^mH_+(x^{-1})
=
(1+mc_+)H_-(x),
\qquad
x^mH_-(x^{-1})
=
(1+mc_-)H_+(x).
\]
Multiplying these two identities yields
\[
\begin{aligned}
x^{2m}
\mathcal N_{[2]\times[m]\times[m+1]}(x^{-1})
&=
x^{2m}H_-(x^{-1})H_+(x^{-1})\\
&=
(1+mc_-)(1+mc_+)
H_+(x)H_-(x)\\
&=
\mathcal N_{[2]\times[m]\times[m+1]}(x).
\end{aligned}
\]
Thus $\mathcal N_{[2]\times[m]\times[m+1]}(x)$ is palindromic.
\end{proof}

\begin{thm}
\label{thm:three-chain-palindromicity-classification}
Let $1\le k\le m\le n$. 
Then $\mathcal N_{[k]\times[m]\times[n]}(x)$ is palindromic if and only if $(k,m,n)=(1,r,r)$ or $(2,r,r+1)$
for some positive integer $r$.
\end{thm}

\begin{proof}
For $k\ge3$, non-palindromicity was proved in Theorem~\ref{thm:kmnpal3}.
The cases $k=1$ and $k=2$ follow from Lemma~\ref{lem:small-k-palindromicity}.
\end{proof}

\subsection{Real-rootedness and stability for \texorpdfstring{$k\le2$}{k at most 2}}
\label{subsec:small-k-real-rootedness}
We first establish a general result on polynomials of the form $Q(x)+cxQ'(x)$.

\begin{thm}
\label{thm:euler-perturbation}
Let $Q$ be a real polynomial of degree $d\ge1$ with distinct negative zeros $\rho_1<\rho_2<\cdots<\rho_d<0$,
and let
\[
G_c(x)=Q(x)+cxQ'(x),
\qquad c\in\mathbb R.
\]
If $c=0$, then $G_c=Q$. 
If $c\ne0$, all zeros of $G_c$ are simple, and their locations are as follows.

\begin{enumerate}[label=\textup{(\roman*)}]
\item
If $c>0$, then $G_c$ has one zero in each of
$(\rho_1,\rho_2),\ldots,
(\rho_{d-1},\rho_d),(\rho_d,0)$.

\item
If $-1/d<c<0$, then $G_c$ has one zero in each of
$(-\infty,\rho_1),
(\rho_1,\rho_2),\ldots,
(\rho_{d-1},\rho_d)$.

\item
If $c=-1/d$, then $\deg G_c=d-1$, and $G_c$ has one zero in each of $(\rho_1,\rho_2),\ldots,(\rho_{d-1},\rho_d)$.

\item
If $c<-1/d$, then $G_c$ has one positive zero and one zero in each of $(\rho_1,\rho_2),\ldots,(\rho_{d-1},\rho_d)$.
\end{enumerate}
\end{thm}

\begin{proof}
The case $c=0$ is immediate, so assume $c\ne0$.
Since $G_c(\rho_i)=c\rho_iQ'(\rho_i)\ne0$, no zero of $G_c$ is a zero of $Q$.
Hence
\begin{equation}
\label{eq:euler-root-equation}
G_c(x)=0
\quad\Longleftrightarrow\quad
\Phi(x):=\frac{xQ'(x)}{Q(x)}=-\frac1c.
\end{equation}

Since the zeros of $Q$ are $\rho_1,\ldots,\rho_d$, we may write $Q(x)
=
a\prod_{i=1}^{d}(x-\rho_i)$ for some nonzero constant $a$.
Taking the logarithmic derivative of the above factorization and multiplying by $x$, we obtain
\[
\Phi(x)
=
\frac{xQ'(x)}{Q(x)}
=
\sum_{i=1}^{d}\frac{x}{x-\rho_i}.
\]
Differentiating gives $\Phi'(x)>0$. 
Therefore $\Phi$ is strictly increasing on every component of $\mathbb R\setminus\{\rho_1,\ldots,\rho_d\}$.

The behavior of $\Phi$ is summarized in the following table.
For an interval $I=(a,b)$, the second and third columns record the limits of $\Phi$ at its left and right endpoints, respectively:
\[
\begin{array}{c|c|c|c}
I=(a,b)
&
\displaystyle \lim_{x\to a^+}\Phi(x)
&
\displaystyle \lim_{x\to b^-}\Phi(x)
&
\Phi(I)
\\ \hline
(-\infty,\rho_1)
&
d
&
+\infty
&
(d,\infty)
\\[1mm]
(\rho_i,\rho_{i+1})
&
-\infty
&
+\infty
&
\mathbb R
\\[1mm]
(\rho_d,0)
&
-\infty
&
0
&
(-\infty,0)
\\[1mm]
(0,\infty)
&
0
&
d
&
(0,d)
\end{array}
\]
where $1\le i<d$.

We now compare these ranges with the value $-1/c$.

\begin{enumerate}[label=\textup{(\roman*)}]
\item
If $c>0$, then $-\frac1c<0$.
Hence~\eqref{eq:euler-root-equation} has one solution in each of $(\rho_1,\rho_2),\ldots,(\rho_{d-1},\rho_d),(\rho_d,0)$,
and no other solutions.

\item
If $-1/d<c<0$, then $-\frac1c>d$.
Hence~\eqref{eq:euler-root-equation} has one solution in $(-\infty,\rho_1)$ and one in each of $(\rho_1,\rho_2),\ldots, (\rho_{d-1},\rho_d)$, and no other solutions.

\item
If $c=-1/d$, then $-\frac1c=d$.
Since $\Phi((-\infty,\rho_1))=(d,\infty)$, equation~\eqref{eq:euler-root-equation} has no solution in $(-\infty,\rho_1)$.
It has, however, one solution in each interval $(\rho_1,\rho_2),\ldots,(\rho_{d-1},\rho_d)$.

\item
If $c<-1/d$, then $0<-\frac1c<d$.
Hence~\eqref{eq:euler-root-equation} has one solution in $(0,\infty)$ and one in each interval $(\rho_1,\rho_2),\ldots,(\rho_{d-1},\rho_d)$, and no other solutions.
\end{enumerate}
We next verify that all the zeros found above are simple.
Since $G_c(x)=Q(x)\bigl(1+c\Phi(x)\bigr)$,
let $x_0$ be a zero of $G_c$.
Then $1+c\Phi(x_0)=0$.
Differentiating the above factorization gives
\[
G_c'(x)
=
Q'(x)\bigl(1+c\Phi(x)\bigr)
+
cQ(x)\Phi'(x).
\]
Hence
\[
G_c'(x_0)
=
cQ(x_0)\Phi'(x_0)\ne0,
\]
so every zero of $G_c$ is simple.

Finally, we determine the degree of $G_c$ in the critical case $c=-1/d$.
Write $Q(x)=\sum_{j=0}^{d} a_jx^j$ and $a_d\ne0$.
Then $xQ'(x)=\sum_{j=1}^{d} ja_jx^j$.
Hence $G_c(x)=\sum_{j=0}^{d}(1+cj)a_jx^j$.
Thus the leading coefficient vanishes precisely when $1+cd=0$.

If $c=-\frac1d$, then the coefficient of $x^d$ vanishes.
On the other hand, the coefficient of $x^{d-1}$ is
\[
\left(1-\frac{d-1}{d}\right)a_{d-1}
=
\frac{a_{d-1}}{d}.
\]
Since $a_{d-1}=-a_d\sum_{i=1}^{d}\rho_i\ne0$, we obtain $\deg G_{-1/d}=d-1$.
This completes the proof of the theorem.
\end{proof}

We use the following classical result on Jacobi polynomials due to Szegő~\cite[Theorem~3.3.1]{Szego1975}.

\begin{thm}[{\cite[Theorem~3.3.1]{Szego1975}}]
\label{thm:jacobi-zeros}
Let $\alpha,\beta>-1$ and let $d$ be a positive integer.
The Jacobi polynomial
\[
P_d^{(\alpha,\beta)}(t)
=
\sum_{r=0}^{d}
\binom{d+\alpha}{d-r}
\binom{d+\beta}{r}
\left(\frac{t-1}{2}\right)^r
\left(\frac{t+1}{2}\right)^{d-r}
\]
has exactly $d$ distinct zeros, all lying in the open interval $(-1,1)$.
\end{thm}

We next establish the relation between $Q_{m,n}(x)$ defined in~\eqref{eq:Qmn-definition} and Jacobi polynomials.
\begin{lem}
\label{lem:Qmn-jacobi}
For $1\le m\le n$,
\begin{equation}
\label{eq:Qmn-jacobi}
Q_{m,n}(x)
=
\frac{(1-x)^m}{m+1}
P_m^{(1,n-m)}
\left(\frac{1+x}{1-x}\right).
\end{equation}
In particular, $Q_{m,n}(x)$ has $m$ distinct negative zeros.
\end{lem}

\begin{proof}
Taking $d=m$, $\alpha=1$ and $\beta=n-m$ in Theorem~\ref{thm:jacobi-zeros}, we obtain
\[
P_m^{(1,n-m)}(t)
=
\sum_{r=0}^{m}
\binom{m+1}{r+1}
\binom nr
\left(\frac{t-1}{2}\right)^r
\left(\frac{t+1}{2}\right)^{m-r}.
\]

Setting $t=(1+x)/(1-x)$, we get
\[
\begin{aligned}
\frac{(1-x)^m}{m+1}
P_m^{(1,n-m)}
\left(\frac{1+x}{1-x}\right)
&=
\sum_{r=0}^{m}
\frac{1}{m+1}
\binom{m+1}{r+1}
\binom nr x^r \\
&=
\sum_{r=0}^{m}
\frac{1}{r+1}
\binom mr
\binom nr x^r \\
&=
Q_{m,n}(x).
\end{aligned}
\]
Since $n-m\ge0$, Theorem~\ref{thm:jacobi-zeros} shows that $P_m^{(1,n-m)}$ has $m$ distinct zeros in $(-1,1)$.
The inverse transformation $x=\frac{t-1}{t+1}$ maps the interval $(-1,1)$ bijectively onto $(-\infty,0)$.
Therefore the $m$ distinct zeros of $P_m^{(1,n-m)}$ correspond to $m$ distinct negative zeros of $Q_{m,n}(x)$.
\end{proof}

We now return to the antichain polynomials of products of three chains.

\begin{thm}
\label{thm:short-chain-real-rootedness}
Let $1\le k\le2$ and $k\le m\le n$.
Then every zero of $\mathcal N_{[k]\times[m]\times[n]}(x)$ is simple and negative. 
\end{thm}

\begin{proof}
Suppose that $k=1$.
Since $Q_{m,n}$ has $m$ distinct negative zeros, we may apply Theorem~\ref{thm:euler-perturbation} with $c=1$.
It follows that $\mathcal N_{[1]\times[m]\times[n]}(x)=Q_{m,n}(x)+xQ_{m,n}'(x)$
has $m$ distinct negative zeros.

Now suppose that $k=2$.
Set $\lambda=\sqrt{\frac{(m+1)(n+1)}{mn}}$.
By~\eqref{eq:layer-refined-factorization},
\[
\mathcal N_{[2]\times[m]\times[n]}(x)
=
\bigl(Q_{m,n}(x)+(1-\lambda)xQ_{m,n}'(x)\bigr)
\bigl(Q_{m,n}(x)+(1+\lambda)xQ_{m,n}'(x)\bigr).
\]
Since $m\le n$, we have
\[
1<\lambda
=
\sqrt{\left(1+\frac1m\right)
\left(1+\frac1n\right)}
\le
1+\frac1m.
\]
Since $Q_{m,n}$ has $m$ distinct negative zeros, Theorem~\ref{thm:euler-perturbation} applies to both factors.
The factor $Q_{m,n}(x)+(1+\lambda)xQ_{m,n}'(x)$ has $m$ distinct negative zeros.
For the other factor, if $m<n$, then $-\frac1m<1-\lambda<0$, so $Q_{m,n}(x)+(1-\lambda)xQ_{m,n}'(x)$ also has $m$ distinct negative zeros.
If $m=n$, then $1-\lambda=-\frac1m$, and this factor has degree $m-1$ with $m-1$ distinct negative zeros.

It remains to show that the two factors have no common zero.
Suppose that $x_0$ is a common zero of the two factors.
Then
\[
Q_{m,n}(x_0)+(1-\lambda)x_0Q_{m,n}'(x_0)=0,
\qquad
Q_{m,n}(x_0)+(1+\lambda)x_0Q_{m,n}'(x_0)=0.
\]
Subtracting the two equations gives $2\lambda x_0Q_{m,n}'(x_0)=0$.
Since both factors have constant term $Q_{m,n}(0)=1$, we have
$x_0\ne0$. 
Hence $Q_{m,n}'(x_0)=0$.
Substituting back into either equation yields $Q_{m,n}(x_0)=0$,
contradicting the fact that all zeros of $Q_{m,n}$ are simple.
Therefore the two factors have no common zero.
It follows that all zeros of $\mathcal N_{[2]\times[m]\times[n]}(x)$
are simple and negative.
\end{proof}

Combining Theorems~\ref{thm:three-chain-palindromicity-classification}
and~\ref{thm:short-chain-real-rootedness} gives the following
consequence.

\begin{cor}
\label{cor:three-chain-gamma-positive}
Let $1\le k\le m\le n$.
If $\mathcal N_{[k]\times[m]\times[n]}(x)$ is palindromic, then all its zeros are simple and negative.
Consequently, it is $\gamma$-positive.
\end{cor}

\begin{proof}
By Theorem~\ref{thm:three-chain-palindromicity-classification}, the palindromic cases are precisely $(k,m,n)=(1,r,r)$ or $(2,r,r+1)$,
for some positive integer $r$.
Theorem~\ref{thm:short-chain-real-rootedness} therefore implies that all zeros are simple and negative.
Since an antichain polynomial has nonnegative coefficients, palindromicity together with real-rootedness implies $\gamma$-positivity.
\end{proof}

In fact, the factorization~\eqref{eq:layer-refined-factorization} yields a stronger bivariate property.
Recall that a real multivariate polynomial is \defi{real stable} if it does not vanish whenever all its variables have positive imaginary parts; see~\cite{BorceaBranden2009}.
We shall use the following general fact.

\begin{lem}
\label{lem:euler-kernel-stability}
Let $Q$ be a nonzero real polynomial all of whose zeros are negative, and define
\[
M(x):=xQ'(x),
\qquad
A(x):=Q(x)+xQ'(x).
\]
Then, for every $\kappa>0$, the bivariate polynomial
\[
F_\kappa(u,v)
:=
A(u)A(v)-\kappa M(u)M(v)
\]
is real stable.
\end{lem}

\begin{proof}
If $Q$ is constant, then $F_\kappa(u,v)=Q^2$ is a nonzero constant and hence real stable. 
We may therefore assume that $\deg Q\ge1$.
Set $d=\deg Q$ and write $Q(x)=C\prod_{i=1}^{d}(x-\rho_i)$ with $\rho_i<0$,
where the zeros are repeated according to multiplicity.
For $z$ in the upper half-plane, set
\[
R(z):=\frac{zQ'(z)}{Q(z)}
=\sum_{i=1}^{d}\frac{z}{z-\rho_i}.
\]
Since
\[
\operatorname{Im}\frac{z}{z-\rho_i}
=
\frac{-\rho_i\operatorname{Im}z}{|z-\rho_i|^2}>0,
\]
we have $\operatorname{Im}R(z)>0$.
Since $A(z)=Q(z)+zQ'(z)=Q(z)(1+R(z))$, and all zeros of $Q$ are negative real numbers, we have $Q(z)\ne0$ whenever $\operatorname{Im}z>0$.
Moreover, $\operatorname{Im}R(z)>0$, so $1+R(z)\ne0$.
Hence $A(z)\ne0$ in the upper half-plane.
Therefore
\[
\frac{M(z)}{A(z)}
=
\frac{R(z)}{1+R(z)}.
\]
Since
\[
\operatorname{Im}
\left(
\frac{R(z)}{1+R(z)}
\right)
=
\frac{\operatorname{Im}R(z)}
{|1+R(z)|^2}>0,
\]
$M(z)/A(z)$ also lies in the upper half-plane.

Suppose that $u$ and $v$ lie in the upper half-plane and that $F_\kappa(u,v)=0$.
Since $A(u)A(v)\ne0$, we obtain
\begin{equation}\label{eq:MAAM}
\frac{M(u)}{A(u)}
\frac{M(v)}{A(v)}
=
\frac1\kappa.
\end{equation}
Let $\theta_1,\theta_2\in(0,\pi)$ be the respective arguments of
$\frac{M(u)}{A(u)}$ and $\frac{M(v)}{A(v)}$.
The argument of their product is congruent to
$\theta_1+\theta_2$ modulo $2\pi$.
Since $0<\theta_1+\theta_2<2\pi$,
this product cannot be a positive real number, contradicting \eqref{eq:MAAM} is positive.

Therefore $F_\kappa(u,v)\ne0$ whenever
$\operatorname{Im}u,\operatorname{Im}v>0$, and hence
$F_\kappa$ is real stable.
\end{proof}

\begin{thm}
\label{thm:layer-refined-stability}
For all positive integers $m$ and $n$, the layer-refined antichain polynomial $\mathcal N_{m,n}(u,v)$ defined in~\eqref{eq:layer-refined-antichain-polynomial} is real stable.
\end{thm}

\begin{proof}
By~\eqref{eq:layer-refined-factorization},
\[
\mathcal N_{m,n}(u,v)
=
A_{m,n}(u)A_{m,n}(v)
-
\frac{(m+1)(n+1)}{mn}
M_{m,n}(u)M_{m,n}(v),
\]
where, as in~\eqref{eq:AMmn-definitions},
$M_{m,n}(x)=xQ_{m,n}'(x)$ and
$A_{m,n}(x)=Q_{m,n}(x)+xQ_{m,n}'(x)$.
Since all zeros of $Q_{m,n}$ are negative by
Lemma~\ref{lem:Qmn-jacobi}, Lemma~\ref{lem:euler-kernel-stability},
applied with $\kappa=(m+1)(n+1)/(mn)$, shows that
$\mathcal N_{m,n}(u,v)$ is real stable.
\end{proof}

\section{The folded family \texorpdfstring{$[k]\times H_n$}{[k] x Hn}}
\label{sec:Hn}

Recall that $H_n=\{(i,j)\in[n]\times[n]: i\le j\}$, ordered componentwise.
We first record several structural properties of $H_n$ that will be used throughout this section.

\subsection{\texorpdfstring{Structural properties of $H_n$}{Structural properties of Hn}}
\label{subsec:Hn-structure}
The poset $H_n$ is naturally graded by $r(i,j)=i+j-2$, and hence has rank $2n-2$.
Therefore $[k]\times H_n$ is graded of rank $k+2n-3$.
\begin{prop}
\label{prop:Hn-self-dual}
The poset $[k]\times H_n$ is self-dual under the involution
\[
(a,i,j)
\longmapsto
(k+1-a,n+1-j,n+1-i).
\]
\end{prop}

\begin{proof}
Let the involution denoted by $\varphi$.
Since $i\le j$, we have $n+1-j\le n+1-i$, so $\varphi(a,i,j)\in[k]\times H_n$.
Moreover, if $(a,i,j)\le(a',i',j')$, then $\varphi(a',i',j')\le \varphi(a,i,j)$.
Thus $\varphi$ reverses the order.
Since $\varphi^2$ is the identity, it is an order-reversing involution.
\end{proof}

We shall also use the following standard product theorem for symmetric chain orders; see Anderson~\cite[Theorem~3.6.1]{Anderson2002} and Engel~\cite[Theorem~5.1.5]{Engel1997}.

\begin{thm}[{\cite[Theorem~5.1.5]{Engel1997}}]
\label{thm:product-SCD}
If $P$ and $Q$ admit symmetric chain decompositions, then so does $P\times Q$.
\end{thm}

\begin{prop}
\label{prop:Hn-SCD}
The poset $[k]\times H_n$ admits a symmetric chain decomposition for every positive integer $k$.
\end{prop}

\begin{proof}
For $0\le q\le \left\lfloor\frac{n-1}{2}\right\rfloor$, define
\[
\mathcal C_q
=
\{(i,i+2q):1\le i\le n-2q\}
\cup
\{(i,i+2q+1):1\le i\le n-2q-1\}.
\]
The elements of $\mathcal C_q$ form the saturated chain,
\[
(1,1+2q)
<
(1,2+2q)
<
(2,2+2q)
<
(2,3+2q)
<
\cdots
<
(n-2q,n).
\]
Its minimum and maximum elements have ranks $2q$ and $2n-2q-2$, respectively, whose sum is $2n-2=\operatorname{rank}(H_n)$.
Thus $\mathcal C_q$ is symmetric.

For each $q$, the chain $\mathcal C_q$ consists precisely of the elements $(i,j)\in H_n$ satisfying $j-i=2q$ or $j-i=2q+1$.
Every integer $d\in\{0,1,\ldots,n-1\}$ belongs to exactly one pair $\{2q,2q+1\}$, 
namely the pair with $q=\lfloor d/2\rfloor$.
Hence the chains $\mathcal C_q$ are pairwise disjoint and their union
is $H_n$.

Since $[k]$ is a chain, it admits a symmetric chain decomposition.
Therefore Theorem~\ref{thm:product-SCD} implies that $[k]\times H_n$ admits a symmetric chain decomposition.
\end{proof}

\subsection{The case \texorpdfstring{$k\ge3$}{k at least 3}}
\label{subsec:Hn-large-k}

We now consider the case $k\ge3$.
Since the constant term of $\mathcal N_{[k]\times H_n}(x)$ is $1$, palindromicity implies monicity.
By Ding and Dong~\cite[Lemma~2.3]{DingDong2019}, the monic cases satisfy
\begin{equation}
\label{eq:H-monicity-parameter}
k=2n-1-4s
\end{equation}
for some nonnegative integer $s$.
Under~\eqref{eq:H-monicity-parameter}, the poset $[k]\times H_n$ has middle rank $r_0=2(n-s-1)$.

By Proposition~\ref{prop:Hn-SCD}, the poset $[k]\times H_n$ admits a symmetric chain decomposition.
Hence, by Remark~\ref{rem:shadow-criteria}, every $X\subseteq([k]\times H_n)_t$ with $t<r_0$ satisfies $|\nabla X|\ge |X|$.
Let $L=([k]\times H_n)_{r_0}$.
Equivalently,
\[
L
=
\{(a,i,j)\in[k]\times H_n:
a+i+j=2n-2s+1\}.
\]
Since $[k]\times H_n$ admits a symmetric chain decomposition, the middle rank $L$ is a maximum antichain.
Moreover, $\mathcal N_{[k]\times H_n}(x)$ is monic, so the maximum antichain is unique.

Let $w=|L|=w([k]\times H_n)$.
Since $[k]\times H_n$ is self-dual by Proposition~\ref{prop:Hn-self-dual}, Theorem~\ref{thm:self-dual-defect} gives
\begin{equation}
\label{eq:H-two-sided-defect}
D_1([k]\times H_n,L)
=
2D_1^-([k]\times H_n,L),
\end{equation}
where $D_1^-([k]\times H_n,L)$, defined in~\eqref{eq:one-sided-defect-count}, denotes the number of defect-one antichains lying strictly below $L$.
Combining~\eqref{eq:H-two-sided-defect} with Corollary~\ref{cor:defect-one-coefficient}, we see that palindromicity requires
\begin{equation}
\label{eq:H-palindromicity-defect-condition}
2D_1^-([k]\times H_n,L)
=
|[k]\times H_n|-w.
\end{equation}

By Theorem~\ref{thm:weak-shadow-propagation}, if $B\in\Ant(([k]\times H_n)_{<r_0})$ has $L$-defect one and $\rho_{\min}(B)\le r_0-2$, then its lowest-rank part $X=B\cap([k]\times H_n)_{\rho_{\min}(B)}$ satisfies $|\nabla X|=|X|$.
Thus the remaining task is to determine the nonempty sets below the middle rank for which the weak shadow inequality is an equality.

\begin{prop}
\label{prop:H-zero-growth}
Let $k=2n-1-4s\ge3$, and let $r_0=2(n-s-1)$ be the middle rank of $[k]\times H_n$.
Suppose that $t<r_0$ and that
$X\subseteq([k]\times H_n)_t$ is nonempty.
Then $|\nabla X|=|X|$ if and only if, for some integers $u$ and $q$ with $1\le u\le q\le s$,
\[
t=k+2q-3
\qquad\text{and}\qquad
X=\{(k,i,2q-i):u\le i\le q\}.
\]
In this case,
\[
\nabla X
=
\{(k,i,2q+1-i):u\le i\le q\}.
\]
\end{prop}

\begin{proof}
The proof is given in Appendix~\ref{app:technical-counts}.
\end{proof}

The preceding classification immediately localizes defect-one antichains to the two ranks immediately below the middle rank.

\begin{cor}
\label{cor:H-two-level-localization}
Let $B\in\Ant(([k]\times H_n)_{<r_0})$ have $L$-defect one.
Then
\begin{equation}
\label{eq:H-two-level-localization}
B
\subseteq
([k]\times H_n)_{r_0-2}
\cup
([k]\times H_n)_{r_0-1}.
\end{equation}
Moreover, if
$B\cap([k]\times H_n)_{r_0-2}\ne\varnothing$, then for some
integer $u$ with $1\le u\le s$,
\begin{equation}
\label{eq:H-lowest-rank-part}
B\cap([k]\times H_n)_{r_0-2}
=
\{(k,i,2s-i):u\le i\le s\}.
\end{equation}
\end{cor}

\begin{proof}
Let $t=\rho_{\min}(B)$ and put $X=B\cap([k]\times H_n)_t$.
Suppose that $t<r_0-2$.
By Theorem~\ref{thm:weak-shadow-propagation}, $|\nabla X|=|X|$, so Proposition~\ref{prop:H-zero-growth} gives $t=k+2q-3$ for some integer $q$ with $1\le q\le s$.

Apply one rank-pushing step to the lowest-rank part $X$ of $B$ and set $B^\uparrow=(B\setminus X)\cup\nabla X$.
By Theorem~\ref{thm:weak-shadow-propagation}, $B^\uparrow$ is again an antichain of $L$-defect one.
Since $X\ne\varnothing$ and $|\nabla X|=|X|$, the upper shadow
$\nabla X$ is nonempty and is contained in
$([k]\times H_n)_{t+1}$.
Therefore the lowest occupied rank of $B^\uparrow$ is exactly $t+1$.
Since $t+1\le r_0-2$, Theorem~\ref{thm:weak-shadow-propagation} shows that its lowest-rank part $X'=B^\uparrow\cap([k]\times H_n)_{t+1}$ satisfies $|\nabla X'|=|X'|$.
Hence Proposition~\ref{prop:H-zero-growth} gives
\[
t+1=k+2q'-3
\]
for some integer $q'$ with $1\le q'\le s$.
This is impossible, since $t=k+2q-3$ and $t+1=k+2q'-3$ imply that $t$ and $t+1$ have the same parity, whereas consecutive integers have opposite parity.
Therefore $t\ge r_0-2$.
Since $B\subseteq([k]\times H_n)_{<r_0}$, this proves~\eqref{eq:H-two-level-localization}.

Now suppose that $B\cap([k]\times H_n)_{r_0-2}\ne\varnothing$.
Then $\rho_{\min}(B)=r_0-2$.
By Theorem~\ref{thm:weak-shadow-propagation}, the set $X=B\cap([k]\times H_n)_{r_0-2}$ satisfies $|\nabla X|=|X|$.
Hence Proposition~\ref{prop:H-zero-growth} gives $r_0-2=k+2q-3$
for some integer $q$ with $1\le q\le s$.
On the other hand, $r_0-2=k+2s-3$.
Therefore $q=s$.
It follows that~\eqref{eq:H-lowest-rank-part} holds for some integer $u$ with $1\le u\le s$.
This completes the proof.
\end{proof}

\begin{lem}
\label{lem:Hn-central-geometry}
Under the projection $(a,i,j)\longmapsto(i,j)$, the middle rank $([k]\times H_n)_{r_0}$ is identified with
\[
L_{n,s}
=
\{(i,j)\in H_n:
2s+2\le i+j\le 2n-2s\},
\]
while the rank $([k]\times H_n)_{r_0-1}$ is identified with
\[
M_{n,s}
=
\{(i,j)\in H_n:
2s+1\le i+j\le 2n-2s-1\}.
\]
Moreover,
\[
|L_{n,s}|
=
\binom{n+1}{2}-2s(s+1),
\qquad
|M_{n,s}|=|L_{n,s}|-1.
\]
\end{lem}

\begin{proof}
For $(a,i,j)\in([k]\times H_n)_{r_0}$, we have $a=r_0+3-i-j$.
Thus $1\le a\le k$ is equivalent to $r_0+3-k\le i+j\le r_0+2$.
Using $r_0=2n-2s-2$ and~\eqref{eq:H-monicity-parameter}, this becomes
\[
2s+2\le i+j\le2n-2s.
\]
Hence the projection $(a,i,j)\mapsto(i,j)$ identifies the middle rank with $L_{n,s}$.

Similarly, for $(a,i,j)\in([k]\times H_n)_{r_0-1}$, we have $a=r_0+2-i-j$, and the condition $1\le a\le k$ becomes
\[
2s+1\le i+j\le2n-2s-1.
\]
Thus the rank $r_0-1$ is identified with $M_{n,s}$.

For every positive integer $t\le n+1$, the elements
$(i,j)\in H_n$ satisfying $i+j\le t$ are obtained by choosing $1\le i\le\left\lfloor\frac t2\right\rfloor$ and then $i\le j\le t-i$.
Hence their number is
\[
\sum_{i=1}^{\lfloor t/2\rfloor}(t-2i+1)
=
\left\lfloor\frac{t^2}{4}\right\rfloor.
\]

For $L_{n,s}$, the elements of $H_n$ lying outside $L_{n,s}$ are precisely those satisfying $i+j\le 2s+1$ or $i+j\ge 2n-2s+1$.
Under the involution $(i,j)\longmapsto(n+1-j,n+1-i)$, the second region is mapped bijectively onto the first.
Hence both regions contain
\[
\left\lfloor\frac{(2s+1)^2}{4}\right\rfloor
=
s(s+1)
\]
elements.
Therefore
\[
|L_{n,s}|
=
\binom{n+1}{2}-2s(s+1).
\]

For $M_{n,s}$, the elements outside it satisfy $i+j\le2s$ or $i+j\ge2n-2s$.
The first region is counted by $t=2s$.
Under the same involution, the second region is mapped onto
\[
\{(i,j)\in H_n:i+j\le2s+2\},
\]
which is counted by $t=2s+2$.
Thus
\[
\begin{aligned}
|M_{n,s}|
&=
\binom{n+1}{2}
-\left\lfloor\frac{(2s)^2}{4}\right\rfloor
-\left\lfloor\frac{(2s+2)^2}{4}\right\rfloor\\
&=
\binom{n+1}{2}-s^2-(s+1)^2\\
&=
|L_{n,s}|-1.
\end{aligned}
\]
This prove the lemma.
\end{proof}

Under the identifications in Lemma~\ref{lem:Hn-central-geometry}, the upper shadow of a subset $S\subseteq M_{n,s}$ is
\[
\nabla S
=
\{(i,j)\in L_{n,s}:
(i,j)\in S,\ (i-1,j)\in S,\text{ or }(i,j-1)\in S\}.
\]
Since $S$ lies in the rank immediately below $L$, its neighborhood in $L$ is precisely its upper shadow. 
Hence $\delta_L(S)=|\nabla S|-|S|$.
Therefore $S$ has $L$-defect one if and only if $|\nabla S|=|S|+1$.

Let
\[
T_{n,s}
=
\bigl|
\{S\subseteq M_{n,s}:S\ne\varnothing,\ |\nabla S|=|S|+1\}
\bigr|.
\]

\begin{prop}
\label{prop:H-tight-count}
Assume $s\ge1$. Then
\begin{equation}
\label{eq:H-tight-count}
T_{n,s}
=
(s+1)^2
+
\binom{n-2s}{2}
+
n-s
-
\mathbf 1_{\{n=2s+2\}}.
\end{equation}
Here $\mathbf 1_{\{n=2s+2\}}$ denotes the indicator of the condition $n=2s+2$.
\end{prop}

\begin{proof}
The classification of the subsets satisfying
$|\nabla S|=|S|+1$ is given in
Appendix~\ref{app:technical-counts}, from which the formula follows.
\end{proof}

Let $C_{n,s}$ denote the number of $L$-defect-one antichains $B\subseteq([k]\times H_n)_{<r_0}$ satisfying
\[
B\cap([k]\times H_n)_{r_0-2}\ne\varnothing.
\]

\begin{prop}
\label{prop:H-cross-level-count}
Assume $s\ge1$. Then
\begin{equation}
\label{eq:H-cross-level-count}
C_{n,s}
=
\frac{s(2s^2+15s+19)}{6}
-
s\,\mathbf 1_{\{n=2s+2\}},
\end{equation}
where $\mathbf 1_{\{n=2s+2\}}$ denotes the indicator of the condition
$n=2s+2$.
\end{prop}

\begin{proof}
The enumeration is given in Appendix~\ref{app:technical-counts}.
\end{proof}

Combining the two counts gives the total one-sided defect-one contribution.

\begin{cor}
\label{cor:H-total-defect-one}
Assume $s\ge1$. Then
\begin{equation}
\label{eq:H-total-defect-one}
\begin{aligned}
D_1^-([k]\times H_n,L)
={}&
(s+1)^2
+\binom{n-2s}{2}
+n-s\\
&+
\frac{s(2s^2+15s+19)}{6}
-(s+1)\mathbf 1_{\{n=2s+2\}}.
\end{aligned}
\end{equation}
\end{cor}

\begin{proof}
By Corollary~\ref{cor:H-two-level-localization}, every $L$-defect-one antichain below $L$ either lies entirely in rank $r_0-1$ or meets rank $r_0-2$.
The former are counted by~\eqref{eq:H-tight-count}, and the latter by~\eqref{eq:H-cross-level-count}.
Hence
\[
D_1^-([k]\times H_n,L)=T_{n,s}+C_{n,s},
\]
and the stated formula follows.
\end{proof}

Assume $s\ge1$, and set $m=n-2s-1$.
By~\eqref{eq:H-monicity-parameter}, $k=2m+1$, so $m\ge1$ because $k\ge3$.

Recall that $L=([k]\times H_n)_{r_0}$ is the unique maximum antichain and that $w=|L|$ is the width of $[k]\times H_n$.
By Lemma~\ref{lem:Hn-central-geometry},
\[
w
=
\binom{n+1}{2}-2s(s+1).
\]
Moreover,
\[
|[k]\times H_n|
=
(2m+1)\binom{n+1}{2}.
\]
Hence the necessary condition~\eqref{eq:H-palindromicity-defect-condition} becomes
\begin{equation}
\label{eq:H-palindromicity-ms-condition}
2D_1^-([k]\times H_n,L)
=
(2m+1)\binom{n+1}{2}-w.
\end{equation}
We now distinguish two cases according to the value of $m$.

Suppose first that $m=1$. Then $k=3$ and $n=2s+2$. 
Using~\eqref{eq:H-total-defect-one}, we obtain
\[
|[k]\times H_n|-w-2D_1^-([k]\times H_n,L)
=
-\frac{s(s-1)(2s+5)}{3}.
\]
Hence palindromicity requires $s(s-1)(2s+5)=0$. 
Since $s\ge1$, this forces $s=1$, so the only possible pair in this 
case is $(k,n)=(3,4)$.

Assume henceforth that $m\ge2$. If 
$\mathcal N_{[k]\times H_n}(x)$ is palindromic, then substituting
\eqref{eq:H-total-defect-one} into
the left hand side of \eqref{eq:H-palindromicity-ms-condition} yields the necessary equation
\begin{equation}
\label{eq:H-cubic-obstruction}
F(m,s):=
3m^3+12m^2s+6m^2+12ms^2+18ms-3m
-2s^3-15s^2-31s-12=0.
\end{equation}
The following lemma shows that this necessary equation has no integer 
solutions in the relevant range and therefore rules out all the 
remaining cases.

\begin{lem}
\label{lem:cubic-obstruction}
For the polynomial $F$ defined in~\eqref{eq:H-cubic-obstruction}, 
one has $F(m,s)\ne0$ for all integers $m\ge2$ and $s\ge1$.
\end{lem}

\begin{proof}
The proof is given in Appendix~\ref{app:elliptic-obstruction}.
\end{proof}

The preceding discussion shows that, when $s\ge1$, palindromicity is possible only for $(k,n)=(3,4)$.
It remains to consider the boundary case $s=0$.
Then~\eqref{eq:H-monicity-parameter} gives $k=2n-1$, and the middle rank is $r_0=2n-2$.

\begin{prop}
\label{prop:H-zero-boundary-count}
Let $n\ge2$ and $k=2n-1$.
Then
\begin{equation}
\label{eq:H-zero-boundary-count}
D_1^-([k]\times H_n,L)
=
2^{n+1}+\binom{n-1}{2}-5.
\end{equation}
\end{prop}

\begin{proof}
The proof is given in Appendix~\ref{app:technical-counts}.
\end{proof}

Recall that $L=([k]\times H_n)_{r_0}$ is the unique maximum antichain and that $w=|L|$ is the width of $[k]\times H_n$.
Since $s=0$, Lemma~\ref{lem:Hn-central-geometry} gives $w=\binom{n+1}{2}$.
Moreover, $|[k]\times H_n|=(2n-1)\binom{n+1}{2}$.
Hence~\eqref{eq:H-palindromicity-defect-condition} becomes
\[
2D_1^-([k]\times H_n,L)
=
(2n-2)\binom{n+1}{2}.
\]
By~\eqref{eq:H-zero-boundary-count}, the necessary condition for palindromicity reduces to
\begin{equation}
\label{eq:H-boundary-obstruction}
g(n):=2^{n+2}-n^3+n^2-2n-8=0.
\end{equation}

For the function $g$ defined in~\eqref{eq:H-boundary-obstruction}, the first few values are
\[
\begin{array}{c|cccc}
n&2&3&4&5\\ \hline
g(n)&0&0&0&10
\end{array}
\]

For $n\ge5$, $g(n+1)-g(n)=2^{n+2}-3n^2-n-2$.
At $n=5$, this equals $46$. 
Moreover,
\[
\bigl(g(n+2)-g(n+1)\bigr)
-\bigl(g(n+1)-g(n)\bigr)
=
2^{n+2}-6n-4>0
\qquad(n\ge5).
\]
Hence $g(n+1)-g(n)>0$ for all $n\ge5$.
Therefore $g(n)$ is strictly increasing for $n\ge5$.
Since $g(5)=10>0$, the equation $g(n)=0$ has no solution with $n\ge5$.

Therefore the only solutions are $n=2,3,4$, which yield $(k,n)=(3,2),(5,3),(7,4)$.

\begin{thm}
\label{thm:H-palindromicity-large-k}
Let $k\ge3$ and $n\ge1$. 
Then $\mathcal N_{[k]\times H_n}(x)$ is palindromic if and only if
\begin{equation}\label{eq:rangekn}
(k,n)\in\{(3,2),(5,3),(3,4),(7,4)\}.
\end{equation}
\end{thm}

\begin{proof}
The necessity follows from the preceding discussion.
For the converse, a direct enumeration gives
\begin{align*}
\mathcal N_{[3]\times H_2}(x)
&=
(1+x)^3+6x(1+x),\\
\mathcal N_{[5]\times H_3}(x)
&=
\mathcal N_{[3]\times H_4}(x)
=
(1+x)^6
+24x(1+x)^4
+54x^2(1+x)^2
+8x^3,\\
\mathcal N_{[7]\times H_4}(x)
&=
(1+x)^{10}
+60x(1+x)^8
+560x^2(1+x)^6
\\
\quad&+1280x^3(1+x)^4
+690x^4(1+x)^2
+40x^5.
\end{align*}
Hence all four polynomials are palindromic.
\end{proof}

\subsection{Palindromicity}
\label{subsec:Hn-palindromicity}

We now complete the palindromicity classification for $\mathcal N_{[k]\times H_n}(x)$.
It remains only to consider the case $k=1$.

\begin{prop}
\label{prop:Hn-antichain-polynomial}
For every $n\ge1$,
\begin{equation}
\label{eq:Hn-antichain-polynomial}
\mathcal N_{H_n}(x)
=
\sum_{q=0}^{\lfloor(n+1)/2\rfloor}
\binom{n+1}{2q}x^q.
\end{equation}
In particular, $\mathcal N_{H_n}(x)$ is palindromic if and only if $n$ is odd.
\end{prop}

\begin{proof}
Let $A=\{(i_1,j_1),\ldots,(i_q,j_q)\}$ be a $q$-element antichain of $H_n$. 
Since two elements having the same first coordinate are comparable, we may order the elements so that $i_1<\cdots<i_q$.
The antichain condition then forces $j_1>\cdots>j_q$.
Since $(i_t,j_t)\in H_n$ for every $t$, we have $i_t\le j_t$.
It follows that
\[
i_1<\cdots<i_q<j_q+1<j_{q-1}+1<\cdots<j_1+1
\]
is a $2q$-element subset of $[n+1]$.

Conversely, given $1\le a_1<\cdots<a_{2q}\le n+1$, define $i_t=a_t$ and $j_t=a_{2q+1-t}-1$, for $1\le t\le q$.
Then $i_1<\cdots<i_q$ and $j_1>\cdots>j_q$, and $i_q=a_q<a_{q+1}=j_q+1$, so $i_q\le j_q$. 
Hence $(i_t,j_t)\in H_n$ for every $t$, and the resulting set is an antichain.

These two constructions are inverse. Therefore the number of $q$-element antichains of $H_n$ is $\binom{n+1}{2q}$, which proves~\eqref{eq:Hn-antichain-polynomial}.

Suppose first that $n=2r-1$ is odd. 
Then
\[
\mathcal N_{H_{2r-1}}(x)
=
\sum_{q=0}^{r}\binom{2r}{2q}x^q,
\]
and $\binom{2r}{2q}=\binom{2r}{2(r-q)}$.
Thus the polynomial is palindromic.

If $n=2r$ is even, then the polynomial has degree $r$ and leading coefficient $\binom{2r+1}{2r}=2r+1$, whereas its constant coefficient is $1$.
Hence it is not palindromic.
\end{proof}

\begin{thm}
\label{thm:H-palindromicity-classification}
Let $k,n$ be positive integers.
Then $\mathcal N_{[k]\times H_n}(x)$ is palindromic if and only if
either $k=1$ and $n$ is odd
or $(k,n)$ satisfies \eqref{eq:rangekn}
\end{thm}
\begin{proof}
The case $k=1$ follows from
Proposition~\ref{prop:Hn-antichain-polynomial}.

Now suppose that $k\ge2$.
Palindromicity implies monicity, and Ding and Dong's monicity classification gives $k=2n-1-4s$ for some $s\ge0$.
Thus $k$ is odd, so $k\ne2$.
Hence $k\ge3$, and the result follows from Theorem~\ref{thm:H-palindromicity-large-k}.
\end{proof}

\subsection{Real-rootedness}
\label{subsec:Hn-real-rootedness}
We now prove that every palindromic antichain polynomial in the family $[k]\times H_n$ has only simple negative zeros.
We begin with $H_n$ itself.

\begin{prop}
\label{prop:Hn-real-rootedness}
For every $n\ge1$, the zeros of $\mathcal N_{H_n}(x)$ are
\begin{equation}
\label{eq:Hn-roots}
-\tan^2\left(
\frac{(2j-1)\pi}{2(n+1)}
\right),
\qquad
1\le j\le
\left\lfloor\frac{n+1}{2}\right\rfloor.
\end{equation}
In particular, all zeros of $\mathcal N_{H_n}(x)$ are simple and negative.
\end{prop}

\begin{proof}
By~\eqref{eq:Hn-antichain-polynomial}, $\mathcal N_{H_n}(x)=\sum_{q\ge0}\binom{n+1}{2q}x^q$.
For $y\in\mathbb R$, we have
\[
\mathcal N_{H_n}(-y^2)
=
\sum_{q\ge0}
(-1)^q\binom{n+1}{2q}y^{2q}.
\]
By the binomial theorem,
\[
\frac{(1+iy)^{n+1}+(1-iy)^{n+1}}{2}
=
\sum_{q\ge0}
\binom{n+1}{2q}(iy)^{2q}
=
\sum_{q\ge0}
(-1)^q\binom{n+1}{2q}y^{2q}.
\]
Hence
\[
\mathcal N_{H_n}(-y^2)
=
\frac{(1+iy)^{n+1}+(1-iy)^{n+1}}{2}.
\]

Since $\mathcal N_{H_n}(0)=1$, zero is not a root. 
Because $x=-y^2$, the values $y$ and $-y$ give the same value of $x$, so it suffices to consider $y>0$.
For such $y$, we have
\[
1+iy
=
\sqrt{1+y^2}\,
e^{i\arctan y}
\qquad
1-iy
=
\sqrt{1+y^2}\,
e^{-i\arctan y}.
\]
Therefore
\[
\begin{aligned}
\mathcal N_{H_n}(-y^2)
&=
\frac{
	(1+y^2)^{(n+1)/2}e^{i(n+1)\arctan y}
	+
	(1+y^2)^{(n+1)/2}e^{-i(n+1)\arctan y}
}{2}\\
&=
(1+y^2)^{(n+1)/2}
\cos\bigl((n+1)\arctan y\bigr).
\end{aligned}
\]
Since $(1+y^2)^{(n+1)/2}$ is positive, we have $\mathcal N_{H_n}(-y^2)=0$ if and only if $\cos\bigl((n+1)\arctan y\bigr)=0$.
As $y>0$, we have $0<\arctan y<\frac{\pi}{2}$.
Hence the possible zeros of the cosine are precisely
\[
(n+1)\arctan y
=
\frac{(2j-1)\pi}{2},
\qquad
1\le j\le
\left\lfloor\frac{n+1}{2}\right\rfloor.
\]
It follows that
\[
y
=
\tan\left(
\frac{(2j-1)\pi}{2(n+1)}
\right),
\qquad
1\le j\le
\left\lfloor\frac{n+1}{2}\right\rfloor.
\]
Since $x=-y^2$, the corresponding zeros of $\mathcal N_{H_n}(x)$ are precisely those in~\eqref{eq:Hn-roots}.
These numbers are negative and distinct. 
Moreover, $\deg\mathcal N_{H_n}(x)=\left\lfloor\frac{n+1}{2}\right\rfloor$, so the list in~\eqref{eq:Hn-roots} contains all its zeros. 
In particular, every zero is simple and negative.
\end{proof}

It remains to treat the four exceptional palindromic cases in \eqref{eq:rangekn}
Brändén~\cite[Lemma~4.1]{Branden2004} and Gal~\cite[Remark~3.1.1]{Gal2005} independently showed that nonpositive real-rootedness passes from a $\gamma$-polynomial to the corresponding palindromic polynomial. 
We use the following strict version, which also preserves simplicity.

\begin{lem}
\label{lem:gamma-root-transfer}
Let $f(x)$ be a palindromic polynomial of degree $d$, with $\gamma$-polynomial $\Gamma_f(t)$ defined by
\[
f(x)
=
(1+x)^d
\Gamma_f\left(\frac{x}{(1+x)^2}\right).
\]
If $\Gamma_f(t)$ has $\lfloor d/2\rfloor$ distinct negative zeros, then $f(x)$ has $d$ distinct negative zeros.
\end{lem}

\begin{proof}
By assumption, $\Gamma_f(t)=c\prod_{j=1}^{\lfloor d/2\rfloor}(t-\theta_j)$, where $c\ne0$ and $\theta_j$ are distinct negative real numbers.
Hence
\[
\begin{aligned}
f(x)
&=
c(1+x)^d
\prod_{j=1}^{\lfloor d/2\rfloor}
\left(
\frac{x}{(1+x)^2}-\theta_j
\right)\\
&=
c(1+x)^{d-2\lfloor d/2\rfloor}
\prod_{j=1}^{\lfloor d/2\rfloor}
\bigl(x-\theta_j(1+x)^2\bigr).
\end{aligned}
\]

For each $j$, 
\[ 
x-\theta_j(1+x)^2 
=
-\theta_jx^2+(1-2\theta_j)x-\theta_j. 
\] 
Its discriminant is $1-4\theta_j>0$, so it has two distinct real zeros. 
By Vieta's formulas, the product of the two roots is $1$, so they have the same sign. 
Their sum is $(1-2\theta_j)/\theta_j<0$, since $\theta_j<0$.
Therefore each quadratic factor has two distinct negative zeros.
Moreover, the zeros arising from distinct $\theta_j$ are disjoint.
Indeed, if $x$ were a common zero corresponding to
$\theta_j\ne\theta_\ell$, then $x\ne-1$ and $\theta_j=\frac{x}{(1+x)^2}=\theta_\ell$, a contradiction.

If $d$ is odd, the remaining factor $1+x$ contributes the additional
zero $x=-1$, which is not a zero of any quadratic factor. Hence $f(x)$
has $d$ distinct negative zeros.
\end{proof}

It remains to consider the four exceptional palindromic cases in Theorem~\ref{thm:H-palindromicity-classification}.
Write
$\Gamma_{k,n}(t)$ for the $\gamma$-polynomial of $\mathcal N_{[k]\times H_n}(x)$. 
The explicit formulas obtained above give
\[
\begin{aligned}
&\Gamma_{3,2}(t)
=1+6t,\\
&\Gamma_{5,3}(t)
=
\Gamma_{3,4}(t)
=1+24t+54t^2+8t^3,\\
&\Gamma_{7,4}(t)
=1+60t+560t^2+1280t^3+690t^4+40t^5.
\end{aligned}
\]

The polynomial $\Gamma_{3,2}(t)$ has the single negative zero $-1/6$.
For the cubic polynomial, direct caculation shows that $\Gamma_{5,3}$ has a zero in each of the three disjoint negative intervals $(-7,-6)$, $(-1/2,-2/5)$ and $(-1/20,-1/25)$.
Since its degree is $3$, these are all its zeros, and they are distinct.

Similarly,
The five zeros of $\Gamma_{7,4}(x)$ lie in the pairwise disjoint intervals
\[
(-16,-15),\quad
\left(-\frac32,-\frac75\right),\quad
\left(-\frac12,-\frac25\right),\quad
\left(-\frac17,-\frac18\right),\quad
\left(-\frac1{40},-\frac1{50}\right).
\]
In particular, all its zeros are real, distinct, and negative.

Lemma~\ref{lem:gamma-root-transfer} now shows that the antichain polynomials corresponding to all four exceptional pairs have only simple negative zeros. 
Together with Proposition~\ref{prop:Hn-real-rootedness} and Theorem~\ref{thm:H-palindromicity-classification}, this proves the
following.

\begin{thm}
\label{thm:H-palindromic-real-rootedness}
If $\mathcal N_{[k]\times H_n}(x)$ is palindromic, then all its zeros are simple and negative.
\end{thm}

\section{Minuscule consequences and Peck counterexamples}
\label{sec:minuscule-consequences}

In this section, we complete the proof of
Theorem~\ref{thm:main-minuscule} by considering the remaining connected minuscule posets.
We then return to Conjecture~\ref{conj:DD-Peck} and construct an infinite family of counterexamples.

\subsection{The remaining minuscule posets}
By Proctor's classification~\cite{Proctor1984}, after the families $[m]\times[n]$ and $H_n$ treated in the preceding sections, it remains to consider three types of connected minuscule posets.

First,
\[
K_n=[n]\oplus([1]\sqcup[1])\oplus[n],
\]
where $\oplus$ denotes the ordinal sum and $\sqcup$ denotes the disjoint union. 

For a finite poset $P$, let $J(P)$ denote the poset of order ideals of $P$, ordered by inclusion. 
Inductively, define $J^1(P)=J(P)$ and $J^{r+1}(P)=J(J^r(P))$.
The two exceptional connected minuscule posets are
\[
J^2([2]\times[3])
\qquad\text{and}\qquad
J^3([2]\times[3]).
\]
Ding and Dong~\cite{DingDong2019} determined all the palindromic cases for these posets.

\begin{thm}[{\cite{DingDong2019}}]
\label{thm:DD-remaining-palindromicity}
Let $P$ be one of $K_n$, $J^2([2]\times[3])$ and $J^3([2]\times[3])$.
Then $\mathcal N_{[k]\times P}(x)$ is palindromic if and only if one of the following holds.
\[
\begin{cases}
P=K_n, & k\in\{1,2n+1\},\\
P=J^2([2]\times[3]), & k\in\{5,11\},\\
P=J^3([2]\times[3]), & k=1.
\end{cases}
\]
\end{thm}

We now prove real-rootedness in all the cases listed above.
\begin{proof}[Proof of Theorem~\ref{thm:main-minuscule}]
First consider $P=K_n$. 
When $k=1$, we have
\[
\mathcal N_{K_n}(x)=1+(2n+2)x+x^2.
\]
Its zeros are $-(n+1)\pm\sqrt{n(n+2)}$. 
Since $\sqrt{n(n+2)}<n+1$, both zeros are negative.

When $k=2n+1$, Ding and Dong~\cite[Proposition~3.2]{DingDong2019} identified the antichain polynomial with the type-$D$ Narayana polynomial
$\mathcal N_{[2n+1]\times K_n}(x)
=
\operatorname{Cat}(D_{2n+2},x)$.
The type-$D$ Narayana polynomials are real-rooted by Br\"and\'en~\cite[Theorem~7.1]{Branden2006}.
Since this polynomial has nonnegative coefficients and a nonzero constant term, none of its zeros is nonnegative. Hence all its zeros are negative.

Next let $Q=J^2([2]\times[3])$.
By Theorem~\ref{thm:DD-remaining-palindromicity}, it suffices to consider $k=5$ and $k=11$. 
Ding and Dong~\cite[Example~3.4]{DingDong2019} computed the corresponding $\gamma$-expansions:
\[
\mathcal N_{[5]\times Q}(x)
=
(1+x)^{10}
\Gamma_5\left(\frac{x}{(1+x)^2}\right),
\]
where $\Gamma_5(t)=1+70t+745t^2+1850t^3+1025t^4+62t^5$, and
\[
\mathcal N_{[11]\times Q}(x)
=
(1+x)^{16}
\Gamma_{11}\left(\frac{x}{(1+x)^2}\right),
\]
where
\[
\Gamma_{11}(t)
=
1+160t+4900t^2+49280t^3+194810t^4
+314720t^5+193760t^6+35840t^7+860t^8.
\]

For $\Gamma_5(t)$, by direct substitution, the five zeros of $\Gamma_5(x)$ lie, respectively, in the pairwise disjoint intervals
\[
(-15,-14),\quad
\left(-\frac32,-\frac75\right),\quad
\left(-\frac25,-\frac38\right),\quad
\left(-\frac19,-\frac1{10}\right),\quad
\left(-\frac1{50},-\frac1{60}\right).
\]

By the intermediate value theorem, the polynomial has a zero in each of these five pairwise disjoint intervals. 
Since all intervals are negative, these zeros are distinct and negative.
Similarly, $\Gamma_{11}(t)$ has a zero in each of the following pairwise disjoint negative intervals:
\[
\begin{gathered}
	(-36,-35),\quad
	\left(-4,-\frac{15}{4}\right),\quad
	\left(-\frac{13}{10},-\frac54\right),\quad
	\left(-\frac35,-\frac12\right),\\
	\left(-\frac3{10},-\frac14\right),\quad
	\left(-\frac{13}{100},-\frac3{25}\right),\quad
	\left(-\frac1{20},-\frac1{25}\right),\quad
	\left(-\frac1{100},-\frac1{125}\right).
\end{gathered}
\]
Thus $\Gamma_{11}(t)$ has a zero in each of eight disjoint negative intervals. 
Since its degree is $8$, all its zeros are distinct and negative.
Lemma~\ref{lem:gamma-root-transfer} now shows that $\mathcal N_{[5]\times Q}(x)$ and $\mathcal N_{[11]\times Q}(x)$ have only simple negative zeros.

It remains to consider $Q=J^3([2]\times[3])$.
By Theorem~\ref{thm:DD-remaining-palindromicity}, the only palindromic case is $k=1$. 
Ding and Dong~\cite[Proposition~3.1]{DingDong2019} computed
\[
\mathcal N_Q(x)=
1+27x+27x^2+x^3
=
(1+x)^3
\left(
1+24\frac{x}{(1+x)^2}
\right).
\]
Its $\gamma$-polynomial is therefore $\Gamma_1(t)=1+24t$, whose only zero is $-1/24$. 
Lemma~\ref{lem:gamma-root-transfer} shows that $\mathcal N_Q(x)$ has three distinct negative zeros.

By Proctor's classification, every connected minuscule poset belongs to one of the five families
\[
[m]\times[n],\qquad
H_n,\qquad
K_n,\qquad
J^2([2]\times[3]),\qquad
J^3([2]\times[3]).
\]
For $P=[m]\times[n]$, the result follows from Theorems~\ref{thm:three-chain-palindromicity-classification} and~\ref{thm:short-chain-real-rootedness}. 
The case $P=H_n$ follows from Theorem~\ref{thm:H-palindromic-real-rootedness}. 
The remaining three families were treated above. 
Hence, whenever $\mathcal N_{[k]\times P}(x)$ is palindromic, all its zeros are real and negative.
This completes the proof of Theorem~\ref{thm:main-minuscule}.
\end{proof}

\subsection{An infinite family of Peck counterexamples}
\label{subsec:Peck-counterexamples}
We now turn to Conjecture~\ref{conj:DD-Peck}.
In fact, the conjecture fails in a stronger form: there are infinitely many connected Peck posets whose antichain polynomials are not even unimodal.

For positive integers $r$ and $s$, let $A_s$ denote an antichain of cardinality $s$, and define
\[
P_{r,s}
:=
[r]\oplus A_s\oplus[r],
\]
where $\oplus$ denotes the ordinal sum of posets.

\begin{prop}
\label{prop:Peck-counterexample-family}

For all positive integers $r$ and $s$, the poset $P_{r,s}$ is a connected Peck poset, and
\begin{equation}
\label{eq:Peck-counterexample-polynomial}
\mathcal N_{P_{r,s}}(x)
=
(1+x)^s+2rx.
\end{equation}
Moreover, if $s\ge6$ and $r>\frac{s(s-3)}4$, then $\mathcal N_{P_{r,s}}(x)$ is not unimodal.

\end{prop}

\begin{proof}
The rank sizes of $P_{r,s}$ are
\[
\underbrace{1,\ldots,1}_{r\text{ times}},
\ s,\
\underbrace{1,\ldots,1}_{r\text{ times}}.
\]
Hence $P_{r,s}$ is rank-symmetric and rank-unimodal.

Moreover, the bipartite graph between any two consecutive ranks is either a single edge or a complete bipartite graph with one shore of size one. 
Thus $P_{r,s}$ has the normalized matching property.
Consequently, it is strongly Sperner, and hence $P_{r,s}$ is Peck.

Since elements belonging to different summands of the ordinal sum are comparable, every antichain is either contained in $A_s$ or consists of a single element from one of the two chains. 
Therefore~\eqref{eq:Peck-counterexample-polynomial} holds.
The first three nonconstant coefficients are
\[
[x]\mathcal N_{P_{r,s}}(x)=s+2r,
\qquad
[x^2]\mathcal N_{P_{r,s}}(x)=\binom{s}{2},
\qquad
[x^3]\mathcal N_{P_{r,s}}(x)=\binom{s}{3}.
\]
If $r>\frac{s(s-3)}4$ and $s\ge6$, then $s+2r>\binom{s}{2}$ and $\binom{s}{3}>\binom{s}{2}$.
Therefore
\[
[x]\mathcal N_{P_{r,s}}(x)
>
[x^2]\mathcal N_{P_{r,s}}(x)
<
[x^3]\mathcal N_{P_{r,s}}(x),
\]
so the coefficient sequence is not unimodal.
\end{proof}

%%%%%%%%%%%%%%%%%%%%%%%%%%%%%%%%%%%%%%%%%%%%%%%%%%%%%下面是附录
\newpage

\appendix
\section{Technical shadow classifications}
\label{app:technical-counts}

This appendix contains the boundary and zero-growth classifications used in Sections~\ref{sec:three-chains} and~\ref{sec:Hn}.
\subsection{Proof of Lemma~\ref{lem:tight-set-classification}}
\begin{proof}
Put
\[
M=M_{k,m,s},
\qquad
L=L_{k,m,s},
\qquad
e_1=(1,0),
\qquad
e_2=(0,1).
\]	
We first prove the sufficiency of this lemma. 
Let $S$ be a nonempty consecutive segment of one of $C_0,C_1,C_2$.
In each of the three cases, the upper shadow of a consecutive segment of length $|S|$ consists of exactly $|S|+1$ points. 
Hence $|\nabla S|=|S|+1$, so $S$ is tight.
Now consider $S=M$.
The definition of the two adjacent levels gives $\nabla M=L$. Moreover, Lemma~\ref{lem:central-level-geometry} implies $|L|=|M|+1$.
Thus $S$ is also tight.

We now prove the necessity. 
The idea is to decompose the defect $|\nabla S|-|S|$ into local row contributions. 
Tightness forces these contributions to vanish almost everywhere. 
According to whether one of the first $s$ rows is completely filled, 
this implies either that $S$ is a consecutive segment of one of $C_0,C_1,C_2$, or that $S=M$.

For $S\subseteq M$, define $U(S):=S\cup(S+e_1)\cup(S+e_2)$.
Then $\nabla S=U(S)\cap L$.
The points of $U(S)$ lying outside $L$ are precisely
$S\cap C_0$, $(S\cap C_1)+e_2$ and $(S\cap C_2)+e_1$,
all of which are pairwise disjoint. Therefore
\begin{equation}\label{eq:ShadowS-S}
|\nabla S|-|S| = |U(S)|-|S| -|S\cap C_0| -|S\cap C_1| -|S\cap C_2|.
\end{equation}

For $1\le j\le m$, let
$S_j=\{i\in[k]:(i,j)\in S\}$
be the set of first coordinates of the points in the $j$-th row of $S$, and put $S_0=S_{m+1}=\varnothing$.
For a finite set $A\subseteq\mathbb Z$, let $c(A)$ be the number of nonempty consecutive components in $A$, with convention $c(\varnothing)=0$. 
Define $U(S)_j:=S_j\cup S_j^+\cup S_{j-1}$, where $S_j^+:=\{i+1:i\in S_j\}$.
Since $|S_j\cup S_j^+|=|S_j|+c(S_j)$,
we have
\begin{equation}\label{eq:Us-S}
|U(S)_j|-|S_j| = c(S_j) + |S_{j-1}\setminus(S_j\cup S_j^+)|.
\end{equation}
Define $\lambda_j:=\mathbf 1_{\{j\le s,\ s+1-j\in S_j\}}$ and $\rho_j:=\mathbf 1_{\{j\le m-s-1,\ k\in S_j\}}$. 
Since the $m$-th row of $M$ is $C_1$, we get
\[
\sum_{j=1}^m\lambda_j=|S\cap C_0|,
\qquad
\sum_{j=1}^m\rho_j=|S\cap C_2|,
\qquad|S_m|=|S\cap C_1|.
\]
Substituting them into \eqref{eq:ShadowS-S}, then using \eqref{eq:Us-S} gives the exact decomposition
\[
\begin{aligned}
|\nabla S|-|S|
&=
\sum_{j=1}^m c(S_j)
+\sum_{j=2}^m
|S_{j-1}\setminus(S_j\cup S_j^+)|
+|S_m|-\sum_{j=1}^m\lambda_j
-|S_m|-\sum_{j=1}^m\rho_j\\
&=
\sum_{j=1}^m
\bigl(c(S_j)-\lambda_j-\rho_j\bigr)
+
\sum_{j=2}^m
|S_{j-1}\setminus(S_j\cup S_j^+)|\\
&=
\sum_{j=1}^m a_j+\sum_{j=2}^m b_j.
\end{aligned}
\]
where $a_j:=c(S_j)-\lambda_j-\rho_j$ for $1\le j\le m$, and $b_j:=|S_{j-1}\setminus(S_j\cup S_j^+)|$ for $2\le j\le m$.

The first coordinates of the points in the $j$-th row of $M$
form the interval
\[
I_j=
\left[
\max\{1,s+1-j\},
\min\{k,k+m-s-1-j\}
\right].
\]
For $1\le j\le s$, we have $I_j=[s+1-j,k]$.
The points corresponding to its left and right endpoints lie in $C_0$ and $C_2$, respectively.
Thus $\lambda_j+\rho_j\le2$. If $S_j\ne\varnothing$, 
then $c(S_j)\ge1$, and hence
\[
a_j=c(S_j)-\lambda_j-\rho_j\ge-1.
\]
Moreover, equality $a_j=-1$ can occur only when $c(S_j)=1$ and $\lambda_j=\rho_j=1$.
In this case, $S_j$ is a single consecutive component containing both
endpoints of $I_j$, so necessarily $S_j=I_j$.
Conversely, if $S_j=I_j$, then $c(S_j)=1$, $\lambda_j=\rho_j=1$,
and therefore $a_j=-1$. Hence
\[
a_j=-1
\quad\Longleftrightarrow\quad
S_j=I_j.
\]
In particular, for $1\le j\le s$, if $S_j\ne I_j$,
then $a_j\ge0$.
For $s<j\le m-s-1$, we have
$a_j=c(S_j)-\rho_j\ge0$.
For $m-s\le j\le m$, we have $a_j=c(S_j)\ge0$.
Also, $b_j\ge0$ for $2\le j\le m$.

\medskip
\noindent
\textit{Case 1. $S_j\ne I_j$ for every $1\le j\le s$.}

Then every term in the defect decomposition is nonnegative.
Let $q=\max\{j:S_j\ne\varnothing\}$.
Assume that $S$ is tight, so $|\nabla S|-|S|=1$.

\smallskip
\noindent
\textit{Subcase 1.1. $q=m$.}

In this case, $S$ meets the top boundary $C_1$.
Then $a_m=c(S_m)\ge1$.
Hence $a_m=1$ and every other term in the defect decomposition
vanishes. Since $a_{m-1}=0$ and $m-1\ge m-s$, we have $S_{m-1}=\varnothing$.
The equality $b_{m-1}=0$ then gives $S_{m-2}=\varnothing$,
and continuing downward gives $S_j=\varnothing$ for every $1\le j<m$.
Thus $S=S_m\times\{m\}$, and $a_m=1$ says that $S_m$ is a single consecutive
component of $I_m=[1,k-s-1]$.
Therefore $S$ is a nonempty interval in $C_1$.

\smallskip
\noindent
\textit{Subcase 1.2. $m-s\le q<m$.}

In this case, the highest nonempty row of $S$ lies just below the
top boundary $C_1$.
Since $S_{q+1}=\varnothing$, we have $b_{q+1}=|S_q|\ge1$,
while $a_q=c(S_q)\ge1$.
Hence $|\nabla S|-|S|\ge2$,
which contradicts tightness. Thus this subcase cannot occur.

\smallskip
\noindent
\textit{Subcase 1.3. $s<q\le m-s-1$.}

In this case, the highest nonempty row of $S$ lies in the middle
region. Again, $b_{q+1}=|S_q|\ge1$.
Since the total defect is $1$, we must have $|S_q|=1$ and $a_q=0$.
Furthermore, every other $a_j$ and $b_j$ vanishes. 
In the middle range, $a_q=c(S_q)-\rho_q$,
so $S_q=\{k\}$.

If $2\le j\le q$ and $S_j=\{k\}$, then $b_j=0$ gives
\[
S_{j-1}
\subseteq
S_j\cup S_j^+
=
\{k,k+1\}.
\]
Since $S_{j-1}\subseteq[k]$, it follows that $S_{j-1}\subseteq\{k\}$.
Once an empty row occurs, the equality $b_j=0$ forces every lower
row to be empty. Hence the nonempty rows of $S$ are consecutive
and each consists of the single point $(k,j)$. Thus $S$ is a
nonempty interval in $C_2$.

\smallskip
\noindent
\textit{Subcase 1.4. $q\le s$.}

In this case, the highest nonempty row of $S$ lies in the lower
sloping region. As before, $b_{q+1}=|S_q|$, so tightness gives $|S_q|=1$ and $a_q=0$.
Since $q\le s$, the equality $a_q=0$ for a singleton means that
$S_q$ consists of exactly one endpoint of $I_q$.

If $S_q=\{k\}$, the argument of Subcase~1.3 shows that $S$ is a nonempty interval in $C_2$. 
Otherwise, $S_q=\{s+1-q\}$.
For $2\le j\le q$, if $S_j=\{s+1-j\}$,
then $b_j=0$ gives
\[
S_{j-1}
\subseteq
\{s+1-j,s+2-j\}.
\]
Since $I_{j-1}=[s+2-j,k]$,
we obtain $S_{j-1}\subseteq\{s+2-j\}$.
Again, once an empty row occurs, all lower rows are empty.
Hence the nonempty rows are consecutive singletons $(s+1-j,j)$,
and $S$ is a nonempty interval in $C_0$.

Thus, in Case~1, every tight set is a nonempty consecutive segment
of one of $C_0,C_1,C_2$.

\medskip
\noindent
\textit{Case 2. $S_h=I_h$ for some $1\le h\le s$.}

We now show that tightness forces $S=M$. 
Define
\[
\eta_j:=
\begin{cases}
-1,&2\le j\le s,\\
0,&s<j\le m-s-1,\\
1,&m-s\le j\le m.
\end{cases}
\]
For every $2\le j\le m$, we claim that
\[
a_j+b_j
=
\eta_j+|I_{j}\setminus S_{j}|-|I_{j-1}\setminus S_{j-1}|
+
|(I_{j-1}\setminus S_{j-1})\cap(S_j\cup S_j^+)|.
\]
To prove this identity, first observe that $|S_j\cup S_j^+|=|S_j|+c(S_j)$.
According to the position of the $j$-th row, we have
\[
a_j+|I_{j-1}\setminus (S_j\cup S_j^+)|
=
\begin{cases}
|I_j\setminus S_j|-1, & 2\le j\le s,\\[2mm]
|I_j\setminus S_j|, & s<j\le m-s-1,\\[2mm]
|I_j\setminus S_j|+1, & m-s\le j\le m.
\end{cases}
\]
Indeed, if $2\le j\le s$, then
\[
I_j=[s+1-j,k],
\qquad
I_{j-1}=[s+2-j,k],
\]
and the only possible elements of
$S_j\cup S_j^+$ outside $I_{j-1}$ are
$s+1-j$ and $k+1$, corresponding to
$\lambda_j$ and $\rho_j$, respectively.
If $s<j\le m-s-1$, then $I_{j-1}=I_j=[1,k]$,
and the only possible element of
$S_j\cup S_j^+$ outside $I_{j-1}$ is
$k+1$, corresponding to $\rho_j$.
Finally, if $m-s\le j\le m$, then $I_{j-1}$ is obtained from $I_j$
by adding one point at the right end, and $S_j\cup S_j^+\subseteq I_{j-1}$.

Thus the three cases can be written uniformly as
\[
a_j+|I_{j-1}\setminus (S_j\cup S_j^+)|
=
\eta_j+|I_j\setminus S_j|.
\]

On the other hand,
\[
\begin{aligned}
b_j
&=
|S_{j-1}\setminus (S_j\cup S_j^+)|\\
&=
|I_{j-1}\setminus (S_j\cup S_j^+)|
-
\bigl|
(I_{j-1}\setminus S_{j-1})
\setminus (S_j\cup S_j^+)
\bigr|.
\end{aligned}
\]
Moreover,
\[
\bigl|
(I_{j-1}\setminus S_{j-1})
\setminus (S_j\cup S_j^+)
\bigr|=
|I_{j-1}\setminus S_{j-1}|
-
\bigl|
(I_{j-1}\setminus S_{j-1})
\cap (S_j\cup S_j^+)
\bigr|.
\]
Combining these identities proves the claim.

Choose $h\le s$ such that $S_h=I_h$.
Since $a_j\ge-1$ for every $1\le j\le h$
and $b_j\ge0$, we have
\begin{equation}\label{eq:ajbj}
\sum_{j=1}^{h}a_j
+
\sum_{j=2}^{h}b_j
\ge-h.
\end{equation}

Since $|I_h\setminus S_h|=0$, summing the transition identity for
$j=h+1,\ldots,m$ gives
\begin{equation}\label{eq:tail-defect-sum}
\begin{aligned}
	\sum_{j=h+1}^{m}(a_j+b_j)
	&=
	\sum_{j=h+1}^{m}\eta_j
	+|I_m\setminus S_m|
	+\sum_{j=h+1}^{m}
	\bigl|
	(I_{j-1}\setminus S_{j-1})
	\cap (S_j\cup S_j^+)
	\bigr|\\
	&=
	h+1+|I_m\setminus S_m|
	+\sum_{j=h+1}^{m}
	\bigl|
	(I_{j-1}\setminus S_{j-1})
	\cap (S_j\cup S_j^+)
	\bigr|.
\end{aligned}
\end{equation}
Therefore, summing \eqref{eq:ajbj} and
\eqref{eq:tail-defect-sum}, we obtain
\[
|\nabla S|-|S|
\ge
-h+(h+1)
=
1.
\]
If $S$ is tight, equality must hold throughout. In particular, $|I_m\setminus S_m|=0$
and
\[
(I_{j-1}\setminus S_{j-1})
\cap (S_j\cup S_j^+)
=
\varnothing
\qquad
(h+1\le j\le m).
\]
Hence $S_m=I_m$.
We now argue by downward induction on $j$. Suppose that $S_j=I_j$ for some $j>h$. 
Then $S_j\cup S_j^+=I_j\cup I_j^+$ contains $I_{j-1}$. 
Hence
\[
I_{j-1}\setminus S_{j-1}
\subseteq
S_j\cup S_j^+.
\]
But $(I_{j-1}\setminus S_{j-1}) \cap (S_j\cup S_j^+)=\varnothing$ means $I_{j-1}\setminus S_{j-1}=\varnothing$.
Therefore we get $S_{j-1}=I_{j-1}$.

Starting with $j=m$ and proceeding downward, we conclude that $S_j=I_j$ for $h\le j\le m$.
Equality in \eqref{eq:ajbj} also forces $a_j=-1$ for every $1\le j\le h$,
and hence $S_j=I_j$, for every $1\le j\le h$.
Therefore $S=M$.

This completes the classification.
\end{proof}

\subsection{Proof of Lemma~\ref{lem:tight-set-classification-zero}}

\begin{proof}
For $s=0$, we have
\[
L_{k,m,0}=[k]\times[m],
\qquad
M_{k,m,0}=([k]\times[m])\setminus\{(k,m)\}.
\]
We first prove the necessity. 
Suppose that $S\subseteq M_{k,m,0}$ is tight. 
Then $\nabla S\setminus S=\{v\}$
for some $v\in[k]\times[m]$.
We distinguish two cases according to whether $v=(k,m)$.

If $v=(k,m)$, then every upper cover of an element of $S$ belongs to $S\cup\{(k,m)\}$. 
It follows that $S\cup\{(k,m)\}$ is an upper order ideal of $[k]\times[m]$. 
Since $S$ is nonempty, this ideal has at least two elements. 
Thus $S=J\setminus\{(k,m)\}$ for some upper order ideal $J\subseteq[k]\times[m]$ with $|J|\ge2$.

If $v\ne(k,m)$, then $(k,m)\notin\nabla S$ implies $(k-1,m)\notin S$ and $(k,m-1)\notin S$.
We claim that $S$ contains no $(i,j)$ with $i<k$ and $j<m$. Suppose otherwise. From $(i,j)$ to $(k,m)$, consider the two saturated
chains
\begin{align*}
&(i,j)<(i+1,j)<\cdots<(k,j)
<(k,j+1)<\cdots<(k,m),\\
&(i,j)<(i,j+1)<\cdots<(i,m)
<(i+1,m)<\cdots<(k,m).
\end{align*}
Apart from $(i,j)$ and $(k,m)$, these two chains are disjoint.
Since $v\ne(k,m)$, at least one of them does not contain $v$.
Starting from $(i,j)\in S$, every next element on this chain must belong to $S$. 
Otherwise, the first element leaving $S$ would belong to $\nabla S\setminus S=\{v\}$,
contrary to the choice of the chain. Hence the whole chain belongs to $S$, 
and in particular $(k,m)\in S$, 
contradicting $S\subseteq M_{k,m,0}$. This proves the claim.

From the claim, we have
\[
S\subseteq
\{(i,m):1\le i\le k-2\}
\cup
\{(k,j):1\le j\le m-2\}.
\]
Suppose that $S$ meets both sets. Let $(i,m)$ be the rightmost element
of $S$ in the top row and $(k,j)$ the highest element of $S$ in the
right column. Then $(i+1,m)\in\nabla S\setminus S$ and $(k,j+1)\in\nabla S\setminus S$.
These two elements are distinct, contradicting $|\nabla S\setminus S|=1$. Thus $S$ is contained in one of the two boundary chains.

If $S$ has $c$ consecutive components in the natural order on that chain, 
then each component contributes exactly one additional element to the upper shadow. 
Hence $|\nabla S|=|S|+c$. 
Since $S$ is tight, $c=1$. Thus $S$ is a nonempty interval in one of the two stated boundary chains.

We now prove the sufficiency. 
First, let $S=J\setminus\{(k,m)\}$,
where $J\subseteq[k]\times[m]$ is an upper order ideal with $|J|\ge2$.
Every nonempty upper order ideal of $[k]\times[m]$ contains $(k,m)$,
so $S$ is nonempty. 
Since $J\ne\{(k,m)\}$, the ideal $J$ contains an element covered by $(k,m)$, and hence $(k,m)\in\nabla S$.
Moreover, every upper cover of an element of $S$ belongs to $J$.
Therefore $\nabla S=J=S\cup\{(k,m)\}$, and hence $|\nabla S|=|S|+1$, that is, $S$ is tight.

Finally, let $S$ be a nonempty interval in one of the two stated
boundary chains. Its upper shadow consists of the points of $S$
together with exactly one additional point at the upper end of the
interval. That is to say, $|\nabla S|=|S|+1$.
Thus $S$ is tight.
\end{proof}

\subsection{Proof of Proposition~\ref{prop:H-tight-count}}
To prove Proposition~\ref{prop:H-tight-count}, we first deal with the auxiliary boundary for $[k]\times H_n$.
Throughout this subsection, let $s,m\ge1$ and put $n=m+2s+1$, $a=2s+1$, and $b=a+2m$. 
Recall that
\[
D_t=\{(i,j)\in H_n:i+j=t\}.
\]
Thus $D_a$ and $D_b$ are the two boundary diagonals of $M_{n,s}$, and $m=(b-a)/2$ measures the width between them.
Intervals such as $[u,w]$ denote intervals of integers. 
For $1\le r\le2m+1$, define
\[
I_r:=
\left[
\max\{1,r-m\},
\,s+\left\lfloor\frac{r+1}{2}\right\rfloor
\right]
\]
and $V_r=\{(r,i):i\in I_r\}$. Let $\mathcal G_{s,m}$ be the directed graph
with vertex set $V(\mathcal G_{s,m})=\bigsqcup_{r=1}^{2m+1}V_r$ and directed edges $(r,i)\longrightarrow(r+1,i)$ and $(r,i)\longrightarrow(r+1,i+1)$
whenever the indicated endpoints belong to $V(\mathcal G_{s,m})$.

The map $\Xi:V(\mathcal G_{s,m})\longrightarrow\bigcup_{r=1}^{2m+1}D_{a+r}$, 
defined by $\Xi(r,i):=(i,a+r-i)$ is an isomorphism from $\mathcal G_{s,m}$ onto the directed Hasse graph
induced by the diagonals $D_{a+1},D_{a+2},\ldots,D_{b+1}$. 
In particular, $V_1$ corresponds to $D_{a+1}$ and $V_{2m+1}$ corresponds to $D_{b+1}$.

For any $U\subseteq I_1$, we identify $U$ with the vertex set
$\{(1,i):i\in U\}\subseteq V_1$.
A vertex cut separating $U$ from
$V_{2m+1}$ is a subset $C\subseteq V(\mathcal G_{s,m})$ meeting every
directed path from a vertex of $U$ to a vertex of $V_{2m+1}$; the initial and
terminal vertices are allowed to belong to the cut.

The next proposition gives the exact minimum vertex-cut size between
an interval $U\subseteq V_1$ and $V_{2m+1}$, together with a complete
classification of the equality cases.
\begin{prop}\label{prop:shifted-strip-mincut}
Let $U=[u,w]\subseteq I_1$. Every vertex cut separating $U$ from $V_{2m+1}$ has cardinality at least $|U|=w-u+1$. Moreover, the cuts of cardinality $|U|$ are precisely the following.
\begin{enumerate}[label=\textup{(\arabic*)}]
\item If $w<s+1$, the only minimum cut is $C=U\subseteq V_1$.

\item If $w=s+1$ and $u>1$, the minimum cuts are
\[
C_q^-=
\{(1,i):u\le i<q\}
\cup
\{(2,i):q\le i\le s+1\},
\]
where $u\le q\le s+2$.

\item If $U=I_1$, the minimum cuts consist of the
cuts $C_q^-$ above, where $1\le q\le s+2$, together with
\[
C_\rho^+
=
\{(2m,i):m+\rho\le i\le m+s\}
\cup
\{(2m+1,i):m+1\le i\le m+\rho\},
\]
where $0\le \rho\le s+1$. In particular, the two families need not be disjoint when $m=1$.
\end{enumerate}
\end{prop}
\begin{proof}
For $1\le j\le s+1$ and $1\le r\le2m+1$, define $\xi_{r,j} := \left( r,\, j+\left\lfloor\frac{r-1}{2}\right\rfloor \right).$ It follows directly from the definition of $I_r$ that $\xi_{r,j}\in V_r$. Moreover,
\[
P_j:
\xi_{1,j}\longrightarrow \xi_{2,j}\longrightarrow
\cdots\longrightarrow \xi_{2m+1,j}
\]
is a directed path in $\mathcal G_{s,m}$. The paths $P_u,P_{u+1},\ldots,P_w$ are pairwise vertex-disjoint. Therefore every vertex cut separating $U$ from $V_{2m+1}$ must meet each of these paths, and hence has cardinality at least $w-u+1=|U|$.

Now let $C$ be a cut of cardinality $|U|$. Since the paths $P_u,\ldots,P_w$ are pairwise disjoint, $C$ contains exactly one vertex from each $P_j$, and it contains no vertex outside their union. Thus, for each $u\le j\le w$, there is a unique integer $h_j$ such that $C\cap P_j=\{\xi_{h_j,j}\}$.

We first record a switching property of the integers $h_j$. We claim that for $u\le j<w$, the following implications hold:
\begin{equation}\label{eq:switching-property}
\begin{aligned}
	r\ \text{odd},\quad h_j\ge r+1
	&\quad\Longrightarrow\quad h_{j+1}\ge r+1,\\
	r\ \text{even},\quad h_{j+1}\ge r+1
	&\quad\Longrightarrow\quad h_j\ge r+1.
\end{aligned}
\end{equation}
To prove the claim, consider the case: $r$ is odd. We suppose that $h_j\ge r+1$ but $h_{j+1}\le r$. 
There is a cross-edge $\xi_{r,j}\longrightarrow \xi_{r+1,j+1}$. 
Consequently, one may follow $P_j$ up to $\xi_{r,j}$, use this cross-edge, and then follow $P_{j+1}$ from $\xi_{r+1,j+1}$ to the terminal level. 
This path avoids the cut point on $P_j$ and $P_{j+1}$, since it leaves $P_j$ before level $h_j$ and then enters $P_{j+1}$ after level $h_{j+1}$.
This contradicts the assumption that $C$ is a cut. 
The proof for even $r$ is analogous. This proves the claim.

For every integer $t\ge1$, put
$A_t=\{j\in[u,w]:h_j\ge t\}$.
In particular, $A_t=\varnothing$ for $t>2m+1$.
Thus $A_t$ records the paths $P_j$ whose cut vertices lie at level $t$
or above.
By \eqref{eq:switching-property}, $A_t$ is a terminal interval of $[u,w]$ when $t$ is even; $A_t$ is an initial interval of $[u,w]$ when $t$
is odd. 
Let $\ell=\min_{u\le j\le w}h_j$. Then $A_\ell=[u,w]$ and $A_{\ell+1}\subsetneq[u,w]$. 
Since $A_{\ell+1}$ and $A_{\ell+2}$ have opposite orientations, we have
$A_{\ell+2}\subseteq A_{\ell+1}$ and $A_{\ell+1}\subsetneq [u,w]$. Thus the set $A_{\ell+2}$ must be empty.
Otherwise a nonempty initial interval of $[u,w]$ would be contained in a proper terminal interval, or vice versa, which is impossible.
Hence
\begin{equation}\label{eq:two-consecutive-levels}
h_j\in\{\ell,\ell+1\}
\quad
(u\le j\le w).
\end{equation}
Moreover, the set of indices for which $h_j=\ell+1$ is a suffix when $\ell$ is odd and a prefix when $\ell$ is even.

We now distinguish three cases.

\medskip
\noindent
\emph{Case 1: $w<s+1$.}

In this case, $\ell=1$. Suppose instead that $\ell\ge2$. Starting from $\xi_{1,w}$, use the edge $\xi_{1,w}\longrightarrow \xi_{2,w+1}$ and then follow the path $P_{w+1}$. Since $P_{w+1}$ contains no vertex of $C$, this gives a directed path from $U$ to the terminal level avoiding $C$, a contradiction.

By \eqref{eq:two-consecutive-levels}, the cut points lie in the first two levels, and $A_2=\{j:h_j=2\}$ is a terminal interval of $[u,w]$. If $A_2$ were nonempty, then $w\in A_2$, and the same path
\[
\xi_{1,w}\longrightarrow \xi_{2,w+1}
\longrightarrow\cdots\longrightarrow \xi_{2m+1,w+1}
\]
would avoid $C$. Thus $A_2=\varnothing$, and $h_j=1$ for every $j\in[u,w]$. Hence $C=U$.

\medskip
\noindent
\emph{Case 2: $w=s+1$ and $u>1$.}

In this case, we first show that $\ell\le2$. Suppose that $\ell\ge3$. Follow $P_u$ from $\xi_{1,u}$ to $\xi_{2,u}$, use the cross-edge $\xi_{2,u}\longrightarrow \xi_{3,u-1}$,  
and then follow $P_{u-1}$ to the terminal level. Since $u-1\notin U$, the path $P_{u-1}$ contains no vertex of $C$, so the resulting path avoids $C$, a contradiction. Hence $\ell\le2$. 

Suppose first that $\ell=1$. Then $A_2=\{j\in[u,s+1]:h_j=2\}$ is a terminal interval of $[u,s+1]$. 
Thus $ A_2=[q,s+1]$ for some $u\le q\le s+2$, where $q=s+2$ corresponds to $A_2=\varnothing$. 
Hence 
\[ 
h_j= 
\begin{cases} 
1,&u\le j<q,\\ 
2,&q\le j\le s+1, 
\end{cases} 
\] 
and therefore 
\[ 
C = \{\xi_{1,j}:u\le j<q\} \cup \{\xi_{2,j}:q\le j\le s+1\}, 
\] 
which is exactly $C_q^-$. Now suppose that $\ell=2$. Then $A_3=\{j\in[u,s+1]:h_j=3\}$ is an initial interval of $[u,s+1]$. 
If $A_3$ were nonempty, then $u\in A_3$, so $h_u=3$. 
The path 
\[ 
\xi_{1,u}\longrightarrow \xi_{2,u} \longrightarrow \xi_{3,u-1}\longrightarrow\cdots 
\] 
would then avoid $C$, a contradiction. Thus $A_3=\varnothing$, and hence $h_j=2$ for $u\le j\le s+1$, which implies $C=C_u^-$.  
Thus all minimum cuts are of the asserted form.

\medskip
\noindent
\emph{Case 3: $w=s+1$ and $u=1$.}

If $\ell=1$, the same argument as above shows that $C=C_q^-$ for some $1\le q\le s+2$. The cut with $q=1$, consisting of all vertices of $V_2$, corresponds to the case in which $h_j=2$ for every $j$. It remains to consider cuts lying farther from the bottom. Suppose first that $m\ge2$. For $3\le r\le2m-1$, put $z_r=(r,\left\lfloor\frac{r-1}{2}\right\rfloor)$.
The sequence
\begin{equation*}
\Pi:
\xi_{1,1}\longrightarrow \xi_{2,1}\longrightarrow
z_3\longrightarrow z_4\longrightarrow\cdots
\longrightarrow z_{2m-1}\longrightarrow
\xi_{2m,1}\longrightarrow \xi_{2m+1,1}
\end{equation*}
is a directed path in $\mathcal G_{s,m}$. Notice that, for $3\le r\le2m-1$, the vertex $z_r$ lies strictly to the left of all the vertices $\xi_{r,j}$ with $1\le j\le s+1$. 

If $\ell=2$ and $A_3\ne\varnothing$, then $A_3$ contains $1$, and hence $h_1=3$. Therefore $\xi_{2,1}\notin C$, and the path $\Pi$ avoids $C$, a contradiction. 

If $3\le\ell\le2m-2$, then all points of $C$ lie on the canonical paths $P_j$ in levels $\ell$ and $\ell+1$, whereas $\Pi$ passes through the vertices $z_\ell$ and $z_{\ell+1}$ in these levels. Hence $\Pi$ again avoids $C$.

If $\ell=2m-1$, then the set $A_{2m}$ is a proper terminal interval of $[1,s+1]$, or it is empty. In either case, $1\notin A_{2m}$, and hence $h_1=2m-1$. Thus $\xi_{2m,1}\notin C$, while the vertex of $\Pi$ in level $2m-1$ is $z_{2m-1}$, which lies outside all the canonical paths. Therefore $\Pi$ once again avoids $C$.

Consequently, apart from the bottom cuts, the only possibilities are $\ell=2m$ or $\ell=2m+1$. When $\ell=2m$, the set $A_{2m+1}=\{j:h_j=2m+1\}$ is an initial interval, say $A_{2m+1}=[1,\rho]$ for some $0\le\rho\le s$, where $\rho=0$ denotes the empty set. It follows that
\[
\begin{aligned}
C
&=
\{\xi_{2m,j}:\rho+1\le j\le s+1\}
\cup
\{\xi_{2m+1,j}:1\le j\le\rho\}\\
&=
\{(2m,i):m+\rho\le i\le m+s\}\cup\{(2m+1,i):m+1\le i\le m+\rho\}\\
&=C_\rho^+.
\end{aligned}
\]
When $\ell=2m+1$, all cut points lie in the terminal level, which is the cut $C_{s+1}^+$. When $m=1$, there are no intermediate levels between the first two and the last two levels. The same staircase description \eqref{eq:two-consecutive-levels} gives directly either a cut
$C_q^-$ or a cut $C_\rho^+$.

We now verify the sufficiency, namely, that all the displayed sets are indeed minimum cuts. For $C_q^-$, a path starting at an index $i<q$ meets the cut in the first level, whereas a path starting at an index $i\ge q$ meets the cut in the second level. For the upper cuts, consider the map
\[
\varphi(r,i):=(2m+2-r,i+m+1-r),
\]
which reverses every directed edge of $\mathcal G_{s,m}$ and interchanges the first and terminal levels. Moreover, $\varphi(C_q^-)=C_{q-1}^+$. It follows that every $C_\rho^+$ is also a vertex cut. Since all the displayed cuts have cardinality $|U|$, the proof is complete.
\end{proof}
We now use Proposition~\ref{prop:shifted-strip-mincut} to classify all subsets $S\subseteq M_{n,s}$ satisfying $|\nabla_HS|=|S|+1$. For $S\subseteq H_n$, define its upper-cover shadow by
\[
N_H^+(S)
:=
\{(i+1,j)\in H_n:(i,j)\in S\}
\cup
\{(i,j+1)\in H_n:(i,j)\in S\}.
\]
For $S\subseteq M_{n,s}$, we have $N_H^+(S)\subseteq L_{n,s}$. Under the projection from
$[k]\times H_n$ to $H_n$, the upper shadow of $S$ is $\nabla_HS = (S\setminus D_a)\cup N_H^+(S).$ A nonempty subset $S\subseteq M_{n,s}$ is called $H$-tight if $|\nabla_HS|=|S|+1$. Define the external boundary of $S$ by $\partial_HS:=\nabla_HS\setminus(S\setminus D_a)$. Consequently,
\begin{equation}\label{eq:H-defect-boundary}
|\nabla_HS|-|S|
=
|\partial_HS|-|S\cap D_a|.
\end{equation}

\begin{lem}\label{lem:H-boundary-cut}
Let $S\subseteq M_{n,s}$. Put
\[
U:=
\{i:(i,a-i)\in S\}
\cup
\{i+1:(i,a-i)\in S\}
\subseteq I_1.
\]
The set $\Xi^{-1}(\partial_HS)$ is a vertex cut separating $U$ from $V_{2m+1}$.

Moreover, if $x\in S\setminus D_a$, then every directed path from $\Xi^{-1}(x)$ to $V_{2m+1}$ meets $\Xi^{-1}(\partial_HS)$.
\end{lem}

\begin{proof}
Let $x_1\to\cdots\to x_\ell$ be a directed path from $U$ to
$V_{2m+1}$. By the definition of $U$, the point $\Xi(x_1)$ is an
upper cover of some point of $S\cap D_a$.

If $\Xi(x_1)\notin S$, then
$\Xi(x_1)\in\partial_HS$, and hence
$x_1\in\Xi^{-1}(\partial_HS)$. Otherwise, since
$\Xi(x_\ell)\in D_{b+1}$ while $S\subseteq M_{n,s}$, there is a first
index $j$ such that $\Xi(x_j)\notin S$. Then
$\Xi(x_{j-1})\in S$, and $\Xi(x_j)$ is an upper cover of
$\Xi(x_{j-1})$. Hence
$\Xi(x_j)\in\partial_HS$, so
$x_j\in\Xi^{-1}(\partial_HS)$.

Thus every directed path from $U$ to $V_{2m+1}$ meets
$\Xi^{-1}(\partial_HS)$.

The same first-exit argument, starting from $\Xi^{-1}(x)$ for
$x\in S\setminus D_a$, proves the second assertion.
\end{proof}
For $1\le r\le 2m+1$ and $i\in I_r$, let $v_{r,i}$ denote the vertex of $\mathcal G_{s,m}$ corresponding to $(i,a+r-i)\in D_{a+r}$. The following elementary lemma records a special two-connectivity property of the shifted triangular strip. Although it may also be viewed through the vertex version of Menger's theorem, we give a direct construction of the two paths. 
\begin{lem}\label{lem:H-two-disjoint-paths}
Let $v_{r,i}$ be a nonterminal vertex of $\mathcal G_{s,m}$. If both $v_{r+1,i}$ and $v_{r+1,i+1}$ exist, then there are two directed paths from $v_{r,i}$ to $V_{2m+1}$ which are vertex-disjoint except at $v_{r,i}$.
\end{lem}
\begin{proof}
Start one path at $v_{r+1,i}$ and, at each subsequent level, choose the successor having the smaller second coordinate. Start the other path at $v_{r+1,i+1}$ and, at each subsequent level, choose the successor having the larger second coordinate. We claim that the second-coordinate index of the first path always remains strictly smaller than that of the second path. Suppose that the two paths meet for the first time in some layer. Immediately before meeting, their indices must differ by one. For the left path to move to the common index, its smaller successor must be unavailable, while for the right path to move to the same index, its larger successor must be unavailable. This would imply that the corresponding layer $I_t$ contains only the common index.

However, directly from the definition of $I_t$, one has $|I_t|\ge s+1\ge2$ for every $1\le t\le2m+1$. This is a contradiction. Hence the two paths remain disjoint from their initial vertices onward.
\end{proof}
We now classify the $H$-tight subsets. This classification will also yield Proposition~\ref{prop:H-tight-count} by direct enumeration.
\begin{lem}\label{lem:H-tight-classification}
Let $s,m\ge1$, with $n=m+2s+1$, $a=2s+1$, and $b=a+2m$.
A nonempty subset $S\subseteq M_{n,s}$ is $H$-tight if and only if it belongs to one of the following families:
\begin{enumerate}[label=\textup{(\Roman*)}]
\item $S=\{(i,a-i):u\le i\le v\}$, where $1\le u\le v\le s$;
\item $S=\{(i,a-i):\ell\le i\le s\}\cup\{(i,a+1-i):q\le i\le s+1\}$, where $1\le\ell\le q\le s+1$;
\item $S=\{(i,n):u\le i\le v\}$, where $1\le u\le v\le m$;
\item $S=\{(i,i)\}$, where $s+2\le i\le n-s-1$;
\item $S=(M_{n,s}\setminus D_b)\cup\{(i,b-i):m\le i\le m+\rho-1\}$, where $0\le\rho\le s+1$.
\end{enumerate}
The five families are disjoint except when $m=1$: the type~\textup{(V)} set with $\rho=0$ coincides with the type~\textup{(II)} set with $\ell=q=1$.
\end{lem}
\begin{proof}
We first prove necessity, and we conduct the discussion by categorizing it into two scenarios.

\medskip
\noindent
\textbf{Case 1: $S\cap D_a\ne\varnothing$.}

Let $X=\{i:(i,a-i)\in S\cap D_a\}\subseteq[1,s]$, and let $c(X)$ denote the number of interval components of $X$. The upper covers of $S\cap D_a$ in $D_{a+1}$ have index set $U=X\cup(X+1)$. If the interval components of $X$ are $[u_1,v_1],\ldots,[u_c,v_c]$, then
\[
U=[u_1,v_1+1]\cup\cdots\cup[u_c,v_c+1],
\]
which implies
\begin{equation}\label{eq:U-component-count}
|U|=|X|+c(X)=|S\cap D_a|+c(X).
\end{equation}
By Lemma~\ref{lem:H-boundary-cut}, $\Xi^{-1}(\partial_HS)$ is a vertex cut separating $U$ from $V_{2m+1}$. The canonical paths $P_i$, $i\in U$, are pairwise vertex-disjoint. Therefore, we have $|\partial_HS|=|\Xi^{-1}(\partial_HS)|\ge |U|.$

On the other hand, by \eqref{eq:H-defect-boundary} and the tightness assumption, we have $|\partial_HS|=|S\cap D_a|+1$. 
Together with \eqref{eq:U-component-count}, this gives $|S\cap D_a|+1\ge |S\cap D_a|+c(X)$. 
Since $X\ne\varnothing$, we have $c(X)\ge1$, and therefore $c(X)=1$. 
Thus $X=[u,v]$ for some $1\le u\le v\le s$, and $U=[u,v+1]$. 
Moreover, $|\Xi^{-1}(\partial_HS)|=|U|$, so $\Xi^{-1}(\partial_HS)$ is a minimum vertex cut separating $U$ from $V_{2m+1}$. 
Suppose first that $v<s$. Then $v+1<s+1$, and Proposition~\ref{prop:shifted-strip-mincut} gives
\[
\partial_HS=\{(i,a+1-i):u\le i\le v+1\}.
\]

We claim that $S=S\cap D_a$. Indeed, if $(i,a+r-i)\in S$, for $r\ge2$, then the corresponding graph vertex $(r,i)$ has a directed path to $V_{2m+1}$ which does not meet $V_1$, and hence avoids $\Xi^{-1}(\partial_HS)$, which contradicts Lemma~\ref{lem:H-boundary-cut}. Similarly, if $(i,a+1-i)\in S$ with $i\notin U$, then the canonical path $P_i$ avoids $\Xi^{-1}(\partial_HS)=U$, giving the same contradiction. Finally, points corresponding to vertices of $\Xi^{-1}(\partial_HS)$ cannot belong to $S$, since $\partial_HS\cap(S\setminus D_a)=\varnothing$. Hence $S=S\cap D_a$ which is of type~\textup{(I)}.

Now suppose that $v=s$. Then $U=[u,s+1]$. If $\Xi^{-1}(\partial_HS)$ is a lower minimum cut, then $\Xi^{-1}(\partial_HS)=C_q^-$ for some $u\le q\le s+2$. If $q=s+2$, then $C_q^-=\{(1,i):u\le i\le s+1\},$ and the preceding argument again gives $S=S\cap D_a$, which is also of type~\textup{(I)}.

Suppose that $q\le s+1$. The points in $\{(i,a+1-i):q\le i\le s+1\}$ belong to the upper shadow of $S\cap D_a$, while their corresponding vertices under $\Xi$ do not belong to $C_q^-$. Hence these points are not in $\partial_HS$, and therefore they must belong to $S$. 

There are no further points in $S$. Indeed, if
$(i,a+1-i)\in S$ with $i<u$, then the canonical path $P_i$
from $\Xi^{-1}((i,a+1-i))$ to $V_{2m+1}$ avoids $C_q^-$.
If $(i,a+r-i)\in S$ with $r\ge2$, then
\[
\Xi^{-1}\bigl((i,a+r-i)\bigr)\notin C_q^-,
\]
since $C_q^-=\Xi^{-1}(\partial_HS)$ and
$(i,a+r-i)\in S$. Moreover, as $C_q^-$ is contained in the first
two layers, this vertex has a continuation to $V_{2m+1}$ avoiding
$C_q^-$. Both cases contradict Lemma~\ref{lem:H-boundary-cut}.
Therefore
\[
S=\{(i,a-i):u\le i\le s\}\cup\{(i,a+1-i):q\le i\le s+1\},
\]
which is of type~\textup{(II)}.

It remains to consider the upper minimum cuts. By Proposition~\ref{prop:shifted-strip-mincut}, they occur only when $U=I_1$, i.e., $S\cap D_a=D_a$. In this case, we have $\Xi^{-1}(\partial_HS)=C_\rho^+$ for some $0\le\rho\le s+1$. Since $C_\rho^+$ contains no vertices in $V_1,\ldots,V_{2m-1}$, all points of $D_{a+1}$ belong to $S$. Indeed, $D_{a+1}\subseteq\nabla_H(D_a)\subseteq\nabla_HS$ and
$D_{a+1}\cap\partial_HS=\varnothing$, so $D_{a+1}\subseteq S$.

Now suppose inductively that $D_{a+r}\subseteq S$ for some $1\le r<2m-1$. 
Since $\nabla_H(D_{a+r})\setminus D_{a+r}=D_{a+r+1}$,
we have $D_{a+r+1}\subseteq\nabla_HS$. 
Moreover, $C_\rho^+\cap V_{r+1}=\varnothing$, 
so $D_{a+r+1}\cap\partial_HS=\varnothing$. 
Hence $D_{a+r+1}\subseteq S$. Therefore $D_{a+1}\cup\cdots\cup D_{b-1}\subseteq S$. 
The graph layer $V_{2m}$ corresponds to the diagonal $D_b$. 
The vertices of $V_{2m}$ outside $C_\rho^+$ are exactly $\{(2m,i):m\le i\le m+\rho-1\}$. 
Consequently, 
\[
S\cap D_b=\{(i,b-i):m\le i\le m+\rho-1\}.
\]
Hence
\[
S=(M_{n,s}\setminus D_b)\cup\{(i,b-i):m\le i\le m+\rho-1\},
\]
which is of type~\textup{(V)}.

\medskip
\noindent
\textbf{Case 2: $S\cap D_a=\varnothing$.}

Since $S\subseteq M_{n,s}\setminus D_a$, the definition of $\nabla_H$ gives $S\subseteq\nabla_HS$. Hence the tightness condition is equivalent to $\nabla_HS=S\sqcup\{e\}$ for a unique point $e$. That is to say, $\partial_HS=\{e\}$. By Lemma~\ref{lem:H-boundary-cut}, every directed path in $\mathcal G_{s,m}$ starting at the graph vertex corresponding to a point of $S$ and ending in $V_{2m+1}$ must pass through the graph vertex corresponding to $e$.

Let $(i,j)\in S$ and suppose that $i<j<n$. Then $(i,j)$ has two distinct upper covers in $H_n$, namely $(i,j+1)$ and $(i+1,j)$. If $i+j=a+r$, then under the identification $\Xi$ these correspond respectively to $(r+1,i)$ and $(r+1,i+1)$ in $\mathcal G_{s,m}$. By Lemma~\ref{lem:H-two-disjoint-paths}, there are two directed paths from the graph vertex corresponding to $(i,j)$ to $V_{2m+1}$ which are vertex-disjoint except at their initial vertex. Since $e\notin S$, we have $e\ne (i,j)$, so the vertex corresponding to $e$ cannot lie on both paths. This is a contradiction. 
Therefore every point of $S$ lies either on the top boundary $j=n$ or on the folded diagonal $i=j$. 

Suppose first that $S$ contains a diagonal point $(i,i)$. 
Its unique upper cover in $H_n$ is $(i,i+1)$ which is strictly interior. 
Hence, after the first step from $(i,i)$ to $(i,i+1)$, 
Lemma~\ref{lem:H-two-disjoint-paths} provides two internally vertex-disjoint continuations from $(i,i+1)$ to the terminal level. 
It follows that no point strictly above $(i,i+1)$ can lie on every path from $(i,i)$ to the terminal level. 
Since every such path must meet the unique boundary point $e$, we obtain $e=(i,i+1)$.

We next claim that once there is $(i,i)\in S$, then it is the only point of $S$. First, no top-boundary point can belong to $S$, since every directed path starting from a point $(j,n)$ remains on the top boundary and therefore cannot pass through $(i,i+1)$. Now let $(j,j)\in S$ be another diagonal point. If $j>i$, then $(i,i+1)$ is not reachable from $(j,j)$. If $j<i$, the unique cover $(j,j+1)$ is strictly interior, and Lemma~\ref{lem:H-two-disjoint-paths} gives two internally vertex-disjoint continuations, at most one of which can pass through $(i,i+1)$. In either case there is a path from $(j,j)$ to the terminal level avoiding $e$, a contradiction. Hence we get $S=\{(i,i)\}$.

The condition $(i,i)\in M_{n,s}\setminus D_a$ gives $s+1\le i\le n-s-1$. When $i=s+1$, the singleton $\{(s+1,s+1)\}$ is already of type~\textup{(II)}, corresponding to $\ell=q=s+1$. The remaining diagonal singletons are therefore precisely those of type~\textup{(IV)}.

Finally, suppose that all points of $S$ lie on the top boundary. The top-boundary points of $M_{n,s}$ are $(1,n),(2,n),\ldots,(m,n)$, and their cover relation forms the chain
\[
(1,n)\longrightarrow(2,n)\longrightarrow\cdots\longrightarrow(m,n)\longrightarrow(m+1,n),
\]
where $(m+1,n)\in D_{b+1}$. 
Each interval component of $S$ along this chain, say
$\{(u,n),(u+1,n),\ldots,(v,n)\}$, contributes a distinct
external-boundary point $(v+1,n)$.
Since $|\partial_HS|=1$, the set $S$ has exactly one interval component. Hence $S=\{(i,n):u\le i\le v\}$ for some $1\le u\le v\le m$, which is of type~\textup{(III)}. This completes the proof of necessity.

For sufficiency, we verify the external boundary of each of the five displayed families. 

For type~\textup{(I)}, $\partial_HS=\{(i,a+1-i):u\le i\le v+1\}.$ 
Thus $|\partial_HS|=|S\cap D_a|+1$, and \eqref{eq:H-defect-boundary} gives $|\nabla_HS|=|S|+1$.

For type~\textup{(II)}, we have
\[
\partial_HS
=
\{(i,a+1-i):\ell\le i<q\}
\cup
\{(i,a+2-i):q\le i\le s+1\}.	
\]
Consequently, we get $|\partial_HS|=s-\ell+2=|S\cap D_a|+1$,
and hence $|\nabla_HS|=|S|+1$.

For type~\textup{(III)}, $\partial_HS=\{(v+1,n)\}$. 
Since $S\cap D_a=\varnothing$, equation \eqref{eq:H-defect-boundary} again gives $|\nabla_HS|=|S|+1$.

For type~\textup{(IV)}, $\partial_HS=\{(i,i+1)\}$, so the same conclusion follows.

For type~\textup{(V)}, we have
\[
\partial_HS
=
\{(i,b-i):m+\rho\le i\le m+s\}
\cup
\{(i,b+1-i):m+1\le i\le m+\rho\},
\]	
which implies
\[
|\partial_HS|
=s+1
=|D_a|+1
=|S\cap D_a|+1.
\]	
Hence $|\nabla_HS|=|S|+1$. Thus all five families are $H$-tight.
\end{proof}	
We can now prove Proposition~\ref{prop:H-tight-count}.
\begin{proof}[Proof of Proposition~\ref{prop:H-tight-count}]
By Lemma~\ref{lem:H-tight-classification}, the numbers of sets of the five types are respectively
\[
\binom{s+1}{2},\quad\binom{s+2}{2},\quad\binom{m+1}{2},\quad m-1,\quad s+2.
\]	
The five families are pairwise disjoint except when $m=1$. In that case, the type~\textup{(V)} set with $\rho=0$ is $M_{n,s}\setminus D_b=D_a\cup D_{a+1}$, which coincides with the type~\textup{(II)} set with $\ell=q=1$. This is the only duplication. Therefore
\[
\begin{aligned}
T_{n,s}
&=(s+1)^2+\binom{m+1}{2}+(m-1)+(s+2)-\mathbf 1_{\{m=1\}}\\
&=(s+1)^2+\binom{n-2s}{2}+n-s-\mathbf 1_{\{n=2s+2\}}.
\end{aligned}
\]	
\end{proof}

\subsection{Proof of Proposition~\ref{prop:H-zero-growth}}
\begin{proof} 
We first prove necessity, assuming $|\nabla X|=|X|$. The argument is divided into four steps. 
First, we project the rank level $([k]\times H_n)_t$ onto a diagonal strip of $H_n$ and translate the zero-growth condition into a boundary equality.  
Second, we show that $S$ meets the lower boundary diagonal $D_\lambda$ and that all the relevant shadow inequalities must be equalities.  
Next, we classify the possible form of $S\cap D_\lambda$ for which the upper-cover shadow has the same cardinality.  
Finally, we prove that the corresponding minimum vertex cut is unique, which forces $S$ to have no points outside $D_\lambda$. 

\medskip 
\noindent 
\textbf{Step 1: Projection to $H_n$ and translation of the zero-growth condition.} 

Under the projection $\pi_t:(a,i,j)\longmapsto(i,j)$,  
the rank level $([k]\times H_n)_t$ is identified with the diagonal strip 
\[ 
\mathcal B_\lambda 
= 
\left\{ 
(i,j)\in H_n: 
\lambda\le i+j\le\lambda+k-1 
\right\}, 
\] 
where $\lambda=t+3-k$. 
Indeed, for $(a,i,j)\in([k]\times H_n)_t$ we have $a=t+3-i-j$, and the condition $1\le a\le k$ is equivalent to $\lambda\le i+j\le\lambda+k-1$. Write $S:=\pi_t(X)\subseteq\mathcal B_\lambda$. 

Since an element corresponding to $(i,j)\in S$ has an upper cover obtained by increasing the first coordinate $a$ precisely when $i+j>\lambda$, we have 
\begin{equation}\label{eq:ZG-projected-shadow} 
\pi_{t+1}(\nabla X)=(S\setminus D_\lambda)\cup N_H^+(S). 
\end{equation} 
Since $\pi_t$ and $\pi_{t+1}$ are bijections on the corresponding rank levels, equation \eqref{eq:ZG-projected-shadow} gives 
\begin{equation}\label{eq:ZG-defect} 
|\nabla X|-|X|=|N_H^+(S)\setminus(S\setminus D_\lambda)|-|S\cap D_\lambda|, 
\end{equation} 
which implies
\begin{equation}\label{eq:ZG-zero-growth} 
|\nabla X|=|X| 
\quad\Longleftrightarrow\quad 
|N_H^+(S)\setminus(S\setminus D_\lambda)|=|S\cap D_\lambda|. 
\end{equation} 
Thus it remains only to determine when the condition on the right-hand side of \eqref{eq:ZG-zero-growth} holds.
\medskip 
\noindent 
\textbf{Step 2: Nonemptiness on $D_\lambda$ and equality in the shadow bounds.} 

Now we claim that $S\cap D_\lambda\ne\varnothing$. Otherwise, $S\cap D_\lambda=\varnothing$ gives $S=S\setminus D_\lambda$, and by \eqref{eq:ZG-defect} we would have 
\[
N_H^+(S)\setminus(S\setminus D_\lambda)=N_H^+(S)\setminus S=\varnothing.
\]
Choose a point of $S$ having maximal value of $i+j$. Since $t<r_0$, we have $i+j\le t+2\le r_0+1<2n$, so this point is not the maximum element $(n,n)$ of $H_n$ and therefore has an upper cover in $H_n$. By maximality of $i+j$, that upper cover does not belong to $S$, contradicting $N_H^+(S)\setminus(S\setminus D_\lambda)=\varnothing$. Thus $S\cap D_\lambda\ne\varnothing$. 

Set $U_\lambda:=N_H^+(S\cap D_\lambda)\subseteq D_{\lambda+1}$. We next show that $N_H^+(S)\setminus(S\setminus D_\lambda)$ is a vertex cut separating $U_\lambda$ from $D_{\lambda+k}$ in the directed Hasse graph of $H_n$; here a cut may contain initial or terminal vertices. Indeed, consider a directed path beginning at a point of $U_\lambda$ and ending in $D_{\lambda+k}$. If its first vertex does not belong to $S\setminus D_\lambda$, then it belongs to $N_H^+(S)\setminus(S\setminus D_\lambda)$. Otherwise, follow the path until it first leaves $S\setminus D_\lambda$. Such a vertex exists because $S\setminus D_\lambda\subseteq D_{\lambda+1}\cup\cdots\cup D_{\lambda+k-1}$, whereas the terminal vertex lies in $D_{\lambda+k}$. The first vertex outside $S\setminus D_\lambda$ is an upper cover of a point of $S$, and hence belongs to $N_H^+(S)\setminus(S\setminus D_\lambda)$. Therefore every $U_\lambda$ to $D_{\lambda+k}$ path meets $N_H^+(S)\setminus(S\setminus D_\lambda)$. 

To obtain a lower bound on the size of the vertex cut, we next construct
$|U_\lambda|$ pairwise vertex-disjoint directed paths from $U_\lambda$ to $D_{\lambda+k}$.
Write $k=2m+1$, so $m=n-2s-1\ge1$ and $n=m+2s+1$. For each $(i,\lambda+1-i)\in U_\lambda$, consider the directed path 
\[ 
P_i:\quad 
z_{0,i}\longrightarrow z_{1,i}\longrightarrow\cdots 
\longrightarrow z_{2m,i}, 
\] 
where 
\[ 
z_{\ell,i} 
= 
\left( 
i+\left\lfloor\frac{\ell}{2}\right\rfloor,\, 
\lambda+1+\ell-i-\left\lfloor\frac{\ell}{2}\right\rfloor 
\right), 
\qquad 0\le\ell\le2m. 
\] 
At odd steps the second coordinate is increased, and at even steps the first coordinate is increased. Since $t<r_0$, we have $\lambda\le2s+1$. Each initial point satisfies $1\le i\le\lfloor(\lambda+1)/2\rfloor\le s+1$. Along $P_i$, the first coordinate is at most $i+m\le s+m+1\le n$, and the second is at most $\lambda+m+1-i\le n$. After an even number of steps the second coordinate minus the first is $\lambda+1-2i\ge0$; after an odd number it is one larger. Hence all these vertices lie in $H_n$, and $z_{2m,i}\in D_{\lambda+k}$. Moreover, the paths $P_i$, for distinct $i$, are pairwise vertex-disjoint. Consequently, 
\begin{equation}\label{eq:ZG-cut-lower-bound} 
|N_H^+(S)\setminus(S\setminus D_\lambda)|\ge |U_\lambda|. 
\end{equation} 

On the other hand, $\lambda\le2s+1<n$, so every $(i,j)\in S\cap D_\lambda$ has $j<n$. The map $(i,j)\longmapsto(i,j+1)$ injects $S\cap D_\lambda$ into $U_\lambda$. Hence 
\begin{equation}\label{eq:ZG-diagonal-shadow} 
|U_\lambda|\ge |S\cap D_\lambda|. 
\end{equation} 
Combining \eqref{eq:ZG-zero-growth}, \eqref{eq:ZG-cut-lower-bound}, and \eqref{eq:ZG-diagonal-shadow}, we obtain 
\begin{equation}\label{eq:ZG-all-equalities} 
|N_H^+(S)\setminus(S\setminus D_\lambda)|=|U_\lambda|=|S\cap D_\lambda|. 
\end{equation} 
Thus it remains to classify the subsets $S\cap D_\lambda\subseteq D_\lambda$ satisfying
$|U_\lambda|=|S\cap D_\lambda|$.	

\medskip 
\noindent 
\textbf{Step 3: Classification of $S\cap D_\lambda$.} 

We now determine the possible form of $S\cap D_\lambda$ satisfying $|U_\lambda|=|S\cap D_\lambda|$.
Since $S\cap D_\lambda\ne\varnothing$, necessarily $\lambda\ge2$. 
Moreover, since $t<r_0$, $\lambda=t+3-k\le2s+1$. 
Thus the equality case is impossible when $s=0$. 
In particular, the diagonals under consideration do not meet the top boundary $j=n$. 

Suppose first that $\lambda=2q+1$ is odd. Then $D_\lambda=\{(i,2q+1-i):1\le i\le q\}$. 
Let $\operatorname{ind}_\lambda(S):=\{i:(i,\lambda-i)\in S\}$ be the index set of $S\cap D_\lambda$.
Let $c(\operatorname{ind}_\lambda(S))$ denote the number of interval components of $\operatorname{ind}_\lambda(S)$. 
Then $N_H^+(S\cap D_\lambda)$ has index set $\operatorname{ind}_\lambda(S)\cup(\operatorname{ind}_\lambda(S)+1)$, and therefore 
\[
|N_H^+(S\cap D_\lambda)|
=
|S\cap D_\lambda|+c(\operatorname{ind}_\lambda(S))
>
|S\cap D_\lambda|,
\]
contradicting \eqref{eq:ZG-all-equalities}. Thus $\lambda$ cannot be odd. 

Hence $\lambda=2q$ for some $q$. Then $D_\lambda=\{(i,2q-i):1\le i\le q\}$. The endpoint $(q,q)$ has only one upper cover, whereas every other point of $D_{2q}$ has two upper covers. 
In this case, we have $\operatorname{ind}_\lambda(S)=\{i:(i,2q-i)\in S\}$.
If $q\notin \operatorname{ind}_\lambda(S)$, 
then 
\[
|N_H^+(S\cap D_\lambda)|
=
|S\cap D_\lambda|+c(\operatorname{ind}_\lambda(S))
>
|S\cap D_\lambda|.
\]
If $q\in \operatorname{ind}_\lambda(S)$, then 
\[
|N_H^+(S\cap D_\lambda)|
=
|S\cap D_\lambda|+c(\operatorname{ind}_\lambda(S))-1.
\]
Hence $|N_H^+(S\cap D_\lambda)|=|S\cap D_\lambda|$ holds if and only if $c(\operatorname{ind}_\lambda(S))=1$ and $q\in \operatorname{ind}_\lambda(S)$. Therefore 
\begin{equation}\label{eq:ZG-lower-form} 
S\cap D_\lambda=\{(i,2q-i):u\le i\le q\} 
\end{equation} 
for some $1\le u\le q$. Since $2q=\lambda\le2s+1$, we have $q\le s$. Moreover, 
\[
U_\lambda=N_H^+(S\cap D_\lambda)=\{(i,2q+1-i):u\le i\le q\}. 
\]

\medskip 
\noindent 
\textbf{Step 4: Uniqueness of the minimum cut and elimination of points outside $D_\lambda$.} 

It remains to show that $S=S\cap D_\lambda$. We first prove that $U_\lambda$ is the unique minimum vertex cut separating $U_\lambda$ from $D_{2q+k}$. Let $C$ be such a vertex cut with $|C|=|U_\lambda|=q-u+1$. Suppose that $(c,2q+1-c)\notin C$ for some $u\le c\le q$. We construct $|C|+1$ directed paths from $U_\lambda$ to $D_{2q+k}$ which are pairwise vertex-disjoint except that two of them have the common initial vertex $(c,2q+1-c)$. 

In the first step, use the edges as follows:
\[
\begin{array}{rcll}
(i,2q+1-i)&\longrightarrow&(i,2q+2-i),
&u\le i<c,\\[1mm]
(c,2q+1-c)&\longrightarrow&(c,2q+2-c),\\
(c,2q+1-c)&\longrightarrow&(c+1,2q+1-c),\\[1mm]
(i,2q+1-i)&\longrightarrow&(i+1,2q+1-i),
&c<i\le q.
\end{array}
\]
Thus the first step reaches precisely the $q-u+2$ consecutive vertices from $(u,2q+2-u)$ to $(q+1,q+1)$.
From each $(j,2q+2-j)$ with $u\le j\le q+1$, continue along
\[
\left(
j+\left\lfloor\frac{\ell}{2}\right\rfloor,\,
2q+2+\ell-j-\left\lfloor\frac{\ell}{2}\right\rfloor
\right),
\qquad 0\le\ell\le2m-1.
\]
Since $u\le j\le q+1$, the second coordinate minus the first is $2q+2-2j\ge0$ after an even number of steps and is one larger after an odd number. The first coordinate is at most $j+m-1\le q+m\le n$, and the second is at most $2q+m+2-j\le n$. Thus these continuations lie in $H_n$. At each level their first coordinates are distinct, so they are pairwise vertex-disjoint and end in
$D_{2q+k}$. Hence we obtain $q-u+2$ directed paths from $U_\lambda$ to
$D_{2q+k}$ which are pairwise vertex-disjoint except that the two paths
starting from $(c,2q+1-c)$ share this initial vertex. 
Since $(c,2q+1-c)\notin C$, the cut $C$ must meet these $q-u+2$ paths in
$q-u+2$ distinct vertices, contradicting $|C|=q-u+1$. 
Hence every minimum cut of size $q-u+1$ contains all vertices of $U_\lambda$, and therefore it is exactly $U_\lambda$. 
Then by \eqref{eq:ZG-all-equalities}, we have already proved that $N_H^+(S)\setminus(S\setminus D_\lambda)$ is a minimum cut of size $|U_\lambda|$. Hence 
\begin{equation}\label{eq:NHS=U}
N_H^+(S)\setminus(S\setminus D_\lambda)=U_\lambda.
\end{equation}

To prove $S=S\cap D_\lambda$, it is now sufficient to show $S\setminus D_\lambda=\varnothing$. 
Suppose that $S\setminus D_\lambda\ne\varnothing$, and choose $x\in S\setminus D_\lambda$. 
We have $i+j\le\lambda+k-1$ for $x=(i,j)$ and $\lambda+k\le2n$. Hence $x$ has a directed continuation in $H_n$ to $D_{\lambda+k}$. The first vertex on any such path which leaves $S\setminus D_\lambda$ is an upper cover of a point of $S$, and therefore lies in $N_H^+(S)\setminus(S\setminus D_\lambda)=U_\lambda$ by \eqref{eq:NHS=U}.
But $x\notin U_\lambda$, since $U_\lambda$ is disjoint from $S\setminus D_\lambda$, and $U_\lambda\subseteq D_{\lambda+1}$. Every step of the path increases $i+j$, so a path starting at $x\in D_{\lambda+1}\cup\cdots\cup D_{\lambda+k-1}$ cannot meet $U_\lambda$. This contradiction shows that $S\setminus D_\lambda=\varnothing$, or equivalently $S=S\cap D_\lambda$. By \eqref{eq:ZG-lower-form}, we obtain
\[
S=\{(i,2q-i):u\le i\le q\}.
\] 

Finally, a point of the lower boundary diagonal $D_\lambda=D_{2q}$ of $\mathcal B_\lambda$ corresponds to first coordinate $a=k$. Hence 
\begin{equation}\label{eq:ZG-X-form} 
X=\{(k,i,2q-i):u\le i\le q\}. 
\end{equation} 
Since $\lambda=t+3-k=2q$, we obtain $t=k+2q-3$. This proves necessity. 

For the sufficiency, suppose that \eqref{eq:ZG-X-form} holds. Since the first coordinate is already equal to $k$, there is no upper cover obtained by increasing that coordinate. The upper covers obtained by increasing the last two coordinates have union 
\[ 
\nabla X=\{(k,i,2q+1-i):u\le i\le q\}. 
\] 
Therefore $|\nabla X|=q-u+1=|X|$. The proof is complete. 
\end{proof}

\subsection{Proof of Proposition~\ref{prop:H-cross-level-count}}
Throughout this subsection, assume $s\ge1$ and $k=2n-1-4s\ge3$. Let $\pi_t:([k]\times H_n)_t\to H_n$ denote the projection $(a,i,j)\mapsto(i,j)$.

\begin{lem}\label{lem:H-cross-level-bijection} 
For $1\le u\le s$, put 
\[ 
X_u=\{(k,i,2s-i):u\le i\le s\}\subseteq ([k]\times H_n)_{r_0-2}, 
\quad 
Z_u=\nabla X_u\subseteq ([k]\times H_n)_{r_0-1}, 
\] 
and let 
\[ 
\overline Z_u:=\pi_{r_0-1}(Z_u) 
= 
\{(i,a-i):u\le i\le s\}\subseteq M_{n,s}, 
\quad a=2s+1. 
\] 
Then $\Theta_u$ defined by $\Theta_u(B):=\pi_{r_0-1}\bigl((B\setminus X_u)\cup Z_u\bigr)$
is a bijection between the $L$-defect-one antichains
$B\subseteq ([k]\times H_n)_{<r_0}$ satisfying $B\cap ([k]\times H_n)_{r_0-2}=X_u$ and the
$H$-tight subsets $S\subseteq M_{n,s}$ satisfying
$\overline Z_u\subseteq S$.
\end{lem} 

\begin{proof} 
Fix $u\in[1,s]$. We first show that $\Theta_u$ is well defined. Let $B\subseteq ([k]\times H_n)_{<r_0}$ be an $L$-defect-one antichain satisfying $B\cap ([k]\times H_n)_{r_0-2}=X_u$. By Corollary~\ref{cor:H-two-level-localization}, we obtain $B\subseteq ([k]\times H_n)_{r_0-2}\cup ([k]\times H_n)_{r_0-1}$. Hence $B\setminus X_u\subseteq ([k]\times H_n)_{r_0-1}$. By Proposition~\ref{prop:H-zero-growth}, applied with $q=s$, we have 
\[ 
Z_u=\nabla X_u=\{(k,i,2s+1-i):u\le i\le s\}\text{ and }|Z_u|=|X_u|. 
\] 
Therefore replacing $X_u$ by $Z_u$ is a zero-growth rank-pushing operation. Put $\widetilde S:= (B\setminus X_u)\cup Z_u$. 
By Lemma~\ref{lem:rank-pushing}, $\widetilde S$ is an antichain with $|\widetilde S|=|B|$ and
\[
|\Gamma_L(\widetilde S)|
=|\Gamma_L(B)|=|B|+1
=|\widetilde S|+1.
\]
Now $\widetilde S\subseteq ([k]\times H_n)_{r_0-1}$, so
$\Gamma_L(\widetilde S)\subseteq ([k]\times H_n)_{r_0}$ consists precisely of the upper
covers of $\widetilde S$. Under the projection $\pi_{r_0}$, we have $\pi_{r_0}(\Gamma_L(\widetilde S))=\nabla_H\Theta_u(B)$.
Therefore $|\nabla_H\Theta_u(B)|=|\Theta_u(B)|+1$, so
$\Theta_u(B)$ is $H$-tight. Moreover,
$Z_u\subseteq\widetilde S$ implies
$\overline Z_u=\pi_{r_0-1}(Z_u)\subseteq\Theta_u(B)$.
Hence $\Theta_u(B)$ belongs to the required family, and $\Theta_u$ is
well defined.

We next construct $\Theta_u^{-1}$ and prove it is also well defined. Let $S\subseteq M_{n,s}$ be
$H$-tight with $\overline Z_u\subseteq S$. Since
$Z_u\subseteq\pi_{r_0-1}^{-1}(S)$, define
\[
\Theta_u^{-1}(S)
:=
X_u\cup\left(\pi_{r_0-1}^{-1}(S)\setminus Z_u\right).
\]
We first show that $\Theta_u^{-1}(S)$ is an antichain. The two parts
lie in $([k]\times H_n)_{r_0-2}$ and $([k]\times H_n)_{r_0-1}$, respectively. If
$x\in X_u$ and
$y\in\pi_{r_0-1}^{-1}(S)\setminus Z_u$ were comparable, then
$x\lessdot y$, so $y\in\nabla X_u=Z_u$, a contradiction.
Thus $\Theta_u^{-1}(S)$ is an antichain.

Moreover,
$\Theta_u^{-1}(S)\cap ([k]\times H_n)_{r_0-2}=X_u$, and pushing $X_u$ upward gives
\[
\bigl(\Theta_u^{-1}(S)\setminus X_u\bigr)\cup\nabla X_u
=
\left(\pi_{r_0-1}^{-1}(S)\setminus Z_u\right)\cup Z_u
=
\pi_{r_0-1}^{-1}(S).
\]
Since $|\nabla X_u|=|X_u|$, Lemma~\ref{lem:rank-pushing} gives
$|\Theta_u^{-1}(S)|=|\pi_{r_0-1}^{-1}(S)|=|S|$ and
$\Gamma_L(\Theta_u^{-1}(S))=\Gamma_L(\pi_{r_0-1}^{-1}(S))$.
By the $H$-tightness of $S$,
\[
|\Gamma_L(\Theta_u^{-1}(S))|
=|\Gamma_L(\pi_{r_0-1}^{-1}(S))|
=|\nabla_HS|
=|S|+1
=|\Theta_u^{-1}(S)|+1.
\]
Hence $\Theta_u^{-1}(S)$ is an $L$-defect-one antichain, and
$\Theta_u^{-1}$ is well defined.

Finally, we check that $\Theta_u^{-1}\circ\Theta_u=\operatorname{id}$ and $\Theta_u\circ\Theta_u^{-1}=\operatorname{id}$. For any $L$-defect-one antichain $B\subseteq ([k]\times H_n)_{<r_0}$ satisfying $B\cap ([k]\times H_n)_{r_0-2}=X_u$, we have
\[
\begin{aligned}
\Theta_u^{-1}(\Theta_u(B))
&=
X_u\cup
\left(
\bigl((B\setminus X_u)\cup Z_u\bigr)\setminus Z_u
\right)\\
&=
X_u\cup(B\setminus X_u)
=
B,
\end{aligned}
\]
where $(B\setminus X_u)\cap Z_u=\varnothing$. 
Similarly, for every $H$-tight $S\subseteq M_{n,s}$ with $\overline Z_u\subseteq S$, we have
\[
\begin{aligned}
\Theta_u(\Theta_u^{-1}(S))
&=
\pi_{r_0-1}
\left(
\left(\pi_{r_0-1}^{-1}(S)\setminus Z_u\right)\cup Z_u
\right)\\
&=
\pi_{r_0-1}\bigl(\pi_{r_0-1}^{-1}(S)\bigr)
=
S.
\end{aligned}
\]
Thus $\Theta_u$ and $\Theta_u^{-1}$ are mutually inverse.
\end{proof}

We now count the antichains using this bijection.
\begin{proof}[Proof of Proposition~\ref{prop:H-cross-level-count}]
Put $m=n-2s-1$, $a=2s+1$, and $b=a+2m$.
By Corollary~\ref{cor:H-two-level-localization}, every cross-level $L$-defect-one antichain has a unique lower-rank part $X_u=\{(k,i,2s-i):u\le i\le s\}$ for some $1\le u\le s$. Hence Lemma~\ref{lem:H-cross-level-bijection} reduces the enumeration, for each fixed $u$, to counting the $H$-tight subsets $S\subseteq M_{n,s}$ satisfying $\overline Z_u = \{(i,a-i):u\le i\le s\} \subseteq S$. We count such sets using the five families in Lemma~\ref{lem:H-tight-classification}. For type~\textup{(I)}, we have
\[
S=\{(i,a-i):\alpha\le i\le v\},
\quad
1\le\alpha\le v\le s.
\]
The condition $\overline Z_u\subseteq S$ holds if and only if $v=s$ and $\alpha\le u$. Thus there are exactly $u$ sets of type~\textup{(I)}.

For type~\textup{(II)}, we have
\[
S=\{(i,a-i):\ell\le i\le s\}
\cup
\{(i,a+1-i):q\le i\le s+1\}.
\]
The parameters satisfy $1\le\ell\le q\le s+1$, 
and $\overline Z_u\subseteq S$ holds if and only if $\ell\le u$. 
For a fixed $\ell\le u$, the parameter $q$ can take the values $\ell,\ell+1,\ldots,s+1$, and hence there are $s+2-\ell$ choices. 
Therefore the number of type~\textup{(II)} sets is $\sum_{\ell=1}^{u}(s+2-\ell)$. 

No set of type~\textup{(III)} or type~\textup{(IV)} contains $\overline Z_u$.

For type~\textup{(V)}, every set contains the whole diagonal $D_a$, since $D_a\subseteq M_{n,s}\setminus D_b$. Hence every type~\textup{(V)} set contains $\overline Z_u$. There are therefore $s+2$ such sets.

When $m=1$, Lemma~\ref{lem:H-tight-classification} shows that the type~\textup{(V)} set with $\rho=0$ coincides with the type~\textup{(II)} set with $\ell=q=1$. This duplicated set contains $\overline Z_u$ for every $u$, so one must subtract one in this case.

Consequently, for fixed $u$, the number $N_u$ of $L$-defect-one antichains whose lower-rank part is $X_u$ is
\begin{align*}
N_u
&=u+\sum_{\ell=1}^{u}(s+2-\ell)+(s+2)-\mathbf 1_{\{m=1\}}\\
&=\frac{u(2s+5-u)}{2}+s+2-\mathbf 1_{\{m=1\}}.
\end{align*}
Since the lower-rank part $X_u$ is uniquely determined by the antichain, the classes corresponding to different values of $u$ are disjoint. Therefore $C_{n,s} = \sum_{u=1}^{s}N_u.$ Thus it can be computed as
\[
\begin{aligned}
C_{n,s}
&=
\sum_{u=1}^{s}
\left(
u+
\sum_{\ell=1}^{u}(s+2-\ell)
+s+2
-\mathbf 1_{\{m=1\}}
\right)\\
&=
\frac{s(s+1)}2
+
\sum_{\ell=1}^{s}(s-\ell+1)(s+2-\ell)
+
s(s+2)
-
s\mathbf 1_{\{m=1\}}.
\end{aligned}
\]
Putting $j=s-\ell+1$, we obtain
\[
\sum_{\ell=1}^{s}(s-\ell+1)(s+2-\ell)
=
\sum_{j=1}^{s}j(j+1)
=
\frac{s(s+1)(s+2)}3.
\]
Hence we obtain
\[
\begin{aligned}
C_{n,s}
&=\frac{s(s+1)}2
+\frac{s(s+1)(s+2)}3
+s(s+2)-s\mathbf 1_{\{m=1\}}\\
&=\frac{s(2s^2+15s+19)}6-s\mathbf 1_{\{m=1\}}\\
&=\frac{s(2s^2+15s+19)}6-s\mathbf 1_{\{n=2s+2\}}.
\end{aligned}
\]
This is the stated formula.
\end{proof}

\subsection{Proof of Proposition~\ref{prop:H-zero-boundary-count}}

\begin{proof}
We first reduce the count to antichains in the rank immediately below
the central rank. The central level is identified with $H_n$, and the
level immediately below it is identified with
$M_{n,0}=H_n\setminus\{(n,n)\}$. Suppose that a defect-one antichain
below the central level has lowest occupied rank $t\le r_0-2$.
Theorem~\ref{thm:weak-shadow-propagation} then forces its nonempty
lowest-rank part $X$ to satisfy $|\nabla X|=|X|$. This contradicts
Proposition~\ref{prop:H-zero-growth} with $s=0$. Hence every such
antichain is contained in $([k]\times H_n)_{r_0-1}$.

We next express the defect-one condition as a condition on the
boundary of a subset of $H_n$. Under the projection
$([k]\times H_n)_{r_0-1}\to M_{n,0}$, the antichains in question
correspond to the nonempty subsets $S\subseteq M_{n,0}$ satisfying
$|\nabla_HS|=|S|+1$. Since $D_1=\varnothing$, the same shadow
decomposition gives
$|\nabla_HS|-|S|=|N_H^+(S)\setminus S|$.
Thus, writing $\partial^+S:=N_H^+(S)\setminus S$, the defect-one
condition is equivalent to
\begin{equation}\label{eq:H-zero-boundary}
|\partial^+S|=1.
\end{equation}

We now determine where the unique boundary point can lie.
Suppose that $\partial^+S=\{e\}$. A first-exit argument shows that
every directed path from a point of $S$ to $(n,n)$ must pass through
$e$. If $x=(i,n)$ lies on the top boundary, there is a unique such
path, namely the top-boundary chain. If $x=(i,i)$, its only upper cover
is $(i,i+1)$. If $x=(i,j)$ with $i<j<n$, there are two paths from $x$
to $(n-1,n)$ meeting only at their endpoints. One goes via $(i,n)$
and then follows the top boundary; the other goes via $(j,j)$ and
then follows the zigzag path
\[
(j,j)\to(j,j+1)\to(j+1,j+1)\to\cdots\to(n-1,n).
\]
Consequently, apart from the initial point, the only possible
compulsory vertices are $(n,n)$, $(n-1,n)$, later vertices of the
top-boundary chain, and the first successor $(i,i+1)$ of a diagonal
point $(i,i)$.

We count the sets according to these possible positions of $e$.
If $e=(i,i+1)$ with $1\le i\le n-2$, then necessarily
$S=\{(i,i)\}$. This gives $n-2$ sets.

If $e=(a,n)$ with $2\le a\le n-2$, then $S$ is a nonempty final
interval in the chain
\[
(1,n)<(2,n)<\cdots<(a-1,n).
\]
There are $a-1$ such intervals, giving
$\sum_{a=2}^{n-2}(a-1)=\binom{n-2}{2}$ sets in total.

The two remaining positions of $e$ can be counted using upper order
ideals. If $e=(n-1,n)$, then \eqref{eq:H-zero-boundary} holds
precisely when $J=S\cup\{(n-1,n),(n,n)\}$ is an upper order ideal
of $H_n$ containing at least three elements. By
Proposition~\ref{prop:Hn-antichain-polynomial}, there are $2^n$ upper order
ideals. Excluding the empty ideal, $\{(n,n)\}$, and
$\{(n-1,n),(n,n)\}$ leaves $2^n-3$ possibilities.

Finally, if $e=(n,n)$, then \eqref{eq:H-zero-boundary} holds
precisely when $S\cup\{(n,n)\}$ is an upper order ideal containing
at least two elements. This gives $2^n-2$ possibilities.

These cases exhaust the possible boundary points. Therefore
\[
\begin{aligned}
D_1^-([k]\times H_n,L)
&=(n-2)+\binom{n-2}{2}+(2^n-3)+(2^n-2)\\
&=2^{n+1}+\binom{n-1}{2}-5,
\end{aligned}
\]
as required.
\end{proof}

\section{The elliptic obstruction}
\label{app:elliptic-obstruction}

\subsection{Proof of Lemma~\ref{lem:cubic-obstruction}}
\label{subsec:proof-cubic-obstruction}
We use the elliptic-logarithm method developed by Stroeker and Tzanakis~\cite{StroekerTzanakis1994}, in the general-cubic form of Stroeker and de Weger~\cite{StroekerDeWeger1999}, together with David's lower bound for linear forms in elliptic logarithms~\cite{David1995}.  The basic idea is to convert the original cubic Diophantine equation into an elliptic curve and express its rational points in terms of a Mordell--Weil basis. The restrictions on the original variables then produce a small linear form in elliptic logarithms. Comparing an explicit upper bound for this linear form with a lower bound from transcendence theory yields a finite bound for the Mordell--Weil coefficients; lattice reduction makes this bound effective enough for a final exact search.
\begin{proof}[Proof of Lemma~\ref{lem:cubic-obstruction}]
Suppose, for contradiction, that $F(\mu,s)=0$ has an integral solution with $\mu\ge2$ and $s\ge1$.

\medskip
\noindent\textbf{Step 1. Construct the elliptic model and the Mordell--Weil group.}

We first locate any possible integral solution and then pass to the associated elliptic curve.
For fixed $\mu\ge2$, differentiation with respect to $s$ gives
\[
\frac{\partial F}{\partial s}
=
-6s^2+(24\mu-30)s+12\mu^2+18\mu-31.
\]
The constant term is positive, whereas the leading coefficient is negative. Hence this quadratic has exactly one positive zero. Consequently, $s\mapsto F(\mu,s)$ first increases and then decreases on $[0,\infty)$. Moreover,
\[
F(\mu,0)=3\mu^3+6\mu^2-3\mu-12>0,\quad\lim_{s\to\infty}F(\mu,s)=-\infty.
\]
It follows that $F(\mu,\cdot)$ has exactly one positive real zero.

For $2\le \mu\le6$, this zero lies strictly between two consecutive integers, as follows from
\[
\begin{array}{c|cc}
\mu & \text{positive value} & \text{negative value}\\ \hline
2 & F(2,8)=6       & F(2,9)=-222\\
3 & F(3,15)=54     & F(3,16)=-606\\
4 & F(4,22)=66     & F(4,23)=-1254\\
5 & F(5,28)=1926   & F(5,29)=-24\\
6 & F(6,35)=2724   & F(6,36)=-282.
\end{array}
\]
Thus no positive integral zero exists for $2\le \mu\le6$.

Now suppose that $\mu\ge7$. Direct substitution gives
\begin{align*}
F(\mu,6\mu)&=75\mu^3-426\mu^2-189\mu-12\ge99\mu^2-189\mu-12>0,\\
F(\mu,7\mu)&=-11\mu^3-603\mu^2-220\mu-12<0.
\end{align*}
Since $F(\mu,\cdot)$ has exactly one positive real zero, that zero must lie in $(6\mu,7\mu)$. Thus every integral solution under consideration must satisfy $\mu\ge7$ and $6\mu<s<7\mu$.

We now pass from the cubic curve $F(\mu,s)=0$ to an elliptic curve by an explicit birational transformation. On the affine open set relevant to our argument, set
\begin{equation}\label{eq:forward-birational-map}
\begin{aligned}
	x&=\frac{5484\mu+8507s+4698}{51\mu+5s+15},\\
	y&=\frac{-2361\mu^2-34477\mu s+52353\mu+41646s^2+187407s+83292}
	{\mu(51\mu+5s+15)}.
\end{aligned}
\end{equation}
The inverse transformation is
\begin{equation}\label{eq:inverse-birational-map}
\begin{aligned}
	\mu&=\frac{(x-5866)(25x-889)}{D(x,y)},\\
	s&=-\frac{3(373x^2+5xy-62303x-1566y+2338958)}{D(x,y)},
\end{aligned}
\end{equation}
where
\[
D(x,y)=288x^2+5xy-95951x-8507y+6618599.
\]
Under this transformation, the cubic is birationally equivalent to the minimal Weierstrass model
\[
E:\qquad y^2+y=x^3-10713x+520204.
\]
All these rational-function identities are checked symbolically in the
certificate described in Appendix~\ref{app:elliptic-certificate}.

By the Mordell--Weil theorem, the group $E(\mathbb Q)$ is a finitely
generated abelian group.  In the present case, the supplementary
computations show that the torsion subgroup is trivial and that
$\operatorname{rank} E(\mathbb Q)=4$.  Moreover, the points
\[
\begin{aligned}
P_1&=(1561/100,-597809/1000),\\
P_2&=(889/25,-53713/125),\\
P_3&=(1072/9,26081/27),\\
P_4&=(7,667)
\end{aligned}
\]
form a $\mathbb Z$-basis of $E(\mathbb Q)$. Certificate~01A computes the full
Mordell--Weil group and proves that $E(\mathbb Q)\cong \mathbb Z^4$,
with trivial torsion and full saturation.  Certificate~01B verifies the exact
relations between the Magma generators and the four points above. 

Let $P\in E(\mathbb Q)$ be the point corresponding to $(\mu,s)$. Then we can write uniquely $P=n_1P_1+\cdots+n_4P_4$, and put $M=\max_{1\le i\le4}|n_i|$.

\medskip
\noindent\textbf{Step 2. The limiting point and the elliptic-logarithm linear form.}

We now identify the limiting point on the relevant real branch and define the linear form that measures the distance from $P$ to this point.
The homogeneous cubic part of $F(\mu,s)$ is
\[
3\mu^3+12\mu^2s+12\mu s^2-2s^3.
\]
Since the branch under consideration satisfies $6\mu<s<7\mu$, its
point at infinity has the form $[1:a:0]$ with $a\in(6,7)$.  Substituting
$s=a\mu$ into the homogeneous cubic part gives
\[
a^3-6a^2-6a-\frac32=0.
\]
It is not hard to compute that the relevant point at infinity of
the projective cubic is $[1:a:0]$.
Under the birational map \eqref{eq:forward-birational-map}, this point
is sent to
\begin{equation}\label{eq:def-Q0}
Q_0=\left(
\frac{5484+8507a}{51+5a},
\frac{-2361-34477a+41646a^2}{51+5a}
\right)\in E(\mathbb Q(a)).
\end{equation}
In particular, the first coordinate of $Q_0$ has minimal polynomial
\[
T^3-819T^2+52845T-1191877.
\]
This polynomial has no zero modulo $7$. By Gauss's lemma together
with the standard reduction-modulo-$p$ irreducibility criterion, the
original polynomial is therefore irreducible over $\mathbb Q$, i.e. $Q_0\notin E(\mathbb Q)$.

A direct calculation gives the discriminant of $E$ is $\Delta_E=-38215717087995<0$.
Hence the real locus $E(\mathbb R)$ is connected.

Let $\omega$ be the least positive real period of $E$, and choose normalized real elliptic logarithms
\[
u(P),\quad u(P_1),\ldots,u(P_4),\quad u(Q_0)\in[0,\omega).
\]
For brevity, write
\[
u_i=u(P_i)\quad(1\le i\le4),
\quad
u_0=u(Q_0).
\]
The numerical representatives of $\omega,u_1,u_2,u_3,u_4,u_0$ used in the two lattice reductions are enclosed by certified Arb balls in Certificate~07.  In particular, all nearest-integer roundings used below are determined by rigorous interval containments rather than by floating-point stability alone. 

Since $P=n_1P_1+\cdots+n_4P_4$, there is an integer $\mu_0$ such that
$u(P)=\mu_0\omega+\sum_{i=1}^4n_i u_i.$
The lift of $u_0-u(P)$ determined by the chosen real branch differs from $u_0-u(P)$ by an integral multiple $\ell\omega$ of the real period. Thus the corresponding elliptic integral is
\[
u_0-u(P)+\ell\omega
=
u_0-(\mu_0-\ell)\omega-\sum_{i=1}^4n_i u_i.
\]
Putting
$n_0=\mu_0-\ell,$
we obtain the elliptic-logarithm linear form
\begin{equation}\label{eq:def-linear-form}
L(P)=u_0-n_0\omega-\sum_{i=1}^4n_i u_i.
\end{equation}
By construction, $L(P)$ is the real elliptic integral from $P$ to $Q_0$ along the chosen branch. Notice that $L(P)\ne0$: otherwise, equality modulo the period lattice would give $P=Q_0$, contrary to $P\in E(\mathbb Q)$ and $Q_0\notin E(\mathbb Q)$.

\medskip
\noindent\textbf{Step 3. Upper bound for the elliptic-logarithm linear form.}

We derive an upper bound for $|L(P)|$ in terms of the Mordell--Weil coefficient size $M$.  Implicit differentiation on $F(\mu,s)=0$, followed by substitution in \eqref{eq:forward-birational-map}, gives the identity of invariant differentials
\begin{equation}\label{eq:differential-pullback}
\frac{dx}{2y+1}=-\frac{d\mu}{F_s(\mu,s)}.
\end{equation}
For each real $t\ge7$, the localization argument in Step~1 gives a unique root $s(t)\in(6t,7t)$. Since $6<s(t)/t<7$, the ratio $s(t)/t$ remains in a compact interval. Dividing
$F(t,s(t))=0$ by $t^3$ shows that every accumulation point $b$ of $s(t)/t$ satisfies
$b^3-6b^2-6b-\frac32=0,$
which has a unique root $a$ in $[6,7]$. Consequently, we get $\lim_{t\to\infty}s(t)/t=a$. Substitution in \eqref{eq:forward-birational-map} then gives $\lim_{t\to\infty}P(t)=Q_0$. Moreover,
\[
-F_s(t,s(t))=6s(t)^2-(24t-30)s(t)-12t^2-18t+31>60t^2,
\]
where the last inequality follows because the displayed quadratic is increasing for $s\ge6t$ and its value at $s=6t$ is $60t^2+162t+31$.  Equations \eqref{eq:def-Q0} and \eqref{eq:differential-pullback} therefore imply
\begin{equation}\label{eq:elliptic-integral-upper}
0<|L(P)|=\int_\mu^\infty\frac{dt}{|F_s(t,s(t))|}<\frac1{60\mu}.
\end{equation}

We next record the resulting bound for the period coefficient without relying
on a numerical period computation. Put
$f(x)=x^3-10713x+\frac{2080817}{4},$
and let $\xi$ be the unique real zero of $f$. Since $f(0)>0$ and $f(x)\to-\infty$ as $x\to-\infty$, we have $\xi<0$. Since the invariant
differential on the model $Y^2=f(x)$ is $dx/(2Y)$, the least positive
real period is
$\omega = \int_{\xi}^{\infty}\frac{dx}{\sqrt{f(x)}}.$
On $0\le x\le4$, the polynomial $f$ is positive and decreasing, and
$f(x)\le f(0)=2080817/4$. Therefore,
\[
\omega > \int_0^4\frac{dx}{\sqrt{f(x)}} \ge \frac{8}{\sqrt{2080817}} > \frac1{420}.
\]

Since $\mu\ge7$, equation \eqref{eq:elliptic-integral-upper} gives
$|L(P)|<\frac1{420}<\omega.$
On the other hand, $u_0-u(P)\in(-\omega,\omega)$. Since $L(P)=u_0-u(P)+\ell\omega$, the preceding two inequalities imply $|\ell|\le1$. Moreover,
\[
|\mu_0|=
\left|
\frac{u(P)}{\omega}
-\sum_{i=1}^4n_i\frac{u_i}{\omega}
\right|<
1+\sum_{i=1}^4|n_i|
\le4M+1.
\]
Since $\mu_0$ is an integer, it follows that $|\mu_0|\le4M$. Consequently,
\begin{equation}\label{eq:n0-bound}
|n_0|=|\mu_0-\ell|\le4M+1.
\end{equation}
It remains to relate $\mu$ to $M$. The $j$-invariant of $E$ is
$j_E=-\frac{559565271994368}{157266325465}.$
Since $b_2=0$, the explicit height-difference estimate of Silverman~\cite[Theorem~1.1]{Silverman1990}, formulated using the absolute logarithmic Weil height, gives, for every $R\in E(\overline{\mathbb Q})\setminus\{O\}$,
\begin{equation}\label{eq:height-difference}
\widehat h(R)-\frac12h(x(R))
<\frac{\log|\Delta_E|+\log\max\{1,|j_E|\}}{12}+1.07
<4.36.
\end{equation}
From $\mu\ge7$ and $6\mu<s<7\mu$, the numerator and denominator of the first formula in \eqref{eq:forward-birational-map}, before cancellation, both have absolute value less than $65705\mu$. Hence $h(x(P))<\log(65705\mu)$. For the lower height bound, put $\mathbf n=(n_1,n_2,n_3,n_4)^{\mathsf T}$, and let $\mathcal H=(\langle P_i,P_j\rangle_{\mathrm{NT}})_{1\le i,j\le4}$ be normalized so that
$\widehat h\left(\sum_{i=1}^4n_iP_i\right) = \mathbf n^{\mathsf T}\mathcal H\mathbf n.$
Thus
\[
\mathcal H_{ii}=\widehat h(P_i),
\quad
\mathcal H_{ij}
=
\frac{
\widehat h(P_i+P_j)-\widehat h(P_i)-\widehat h(P_j)
}{2}
\quad(i\ne j).
\]

Certificate~06 (labelled Certificate~06H in the accompanying files) gives a rigorous rational enclosure of this matrix.
More precisely, the two-sided height-difference estimate of
Silverman~\cite[Theorem~1.1]{Silverman1990} gives
$\left|\widehat h(R)-\frac12h(x(R))\right|<8$
for every affine point $R\in E(\mathbb Q)$ on the present curve.
For each
\[
R\in
\{P_1,P_2,P_3,P_4\}
\cup
\{P_i+P_j:1\le i<j\le4\},
\]
the certificate computes $x(16R)$ by exact rational arithmetic.
Since $\widehat h(16R)=256\widehat h(R)$, it follows that
\[
\frac{h(x(16R))}{512}-\frac1{32}
<
\widehat h(R)
<
\frac{h(x(16R))}{512}+\frac1{32}.
\]
All logarithms occurring here are enclosed by rational upper and lower bounds with explicit remainder estimates.

Let $C$ be the midpoint matrix of the resulting rational interval
matrix and let $R=(R_{ij})$ be its radius matrix. The exact computation
gives
\[
\rho:=\max_i\sum_{j=1}^4R_{ij}
=
\frac{85937500001}{500000000000}.
\]
For an admissible height-pairing matrix write
$\mathcal H=C+E_0$. By construction $|(E_0)_{ij}|\le R_{ij}$.
Since $E_0$ is symmetric, its operator norm is bounded by its maximal
absolute row sum, and hence
\[
\|E_0\|_2
\le \|E_0\|_{\infty}
\le \max_i\sum_{j=1}^4R_{ij}
=\rho.
\]
Thus $E_0\succeq-\rho I_4$, so every admissible height-pairing matrix
satisfies
\[
\mathbf v^{\mathsf T}\mathcal H\mathbf v
\ge
\mathbf v^{\mathsf T}A\mathbf v
\quad(\mathbf v\in\mathbb R^4),
\qquad
A:=C-\rho I_4.
\]
Certificate~06 verifies by exact rational arithmetic that all leading
principal minors of $A$ are positive. Hence $A$ is positive
definite. Moreover, it verifies
\[
\min_{1\le i\le4}
\frac{1}{(A^{-1})_{ii}}
=
\frac{
14651501224055023859638974109804626579833246124527
}{
35049496995026167487783130968596493901000000000000
}
>0.418.
\]
If $|n_i|=M=\max_j|n_j|$, the Cauchy--Schwarz inequality for the
inner product defined by $A$ gives
$\mathbf n^{\mathsf T}A\mathbf n \ge \frac{n_i^2}{(A^{-1})_{ii}} > 0.418M^2.$
Therefore,
$\widehat h(P) = \mathbf n^{\mathsf T}\mathcal H\mathbf n \ge 0.405M^2.$
Combining this with \eqref{eq:height-difference} and
$h(x(P))<\log(65705\mu)$, we obtain
\[
0.405M^2
\le
\widehat h(P)
<
\frac12\log(65705\mu)+4.36.
\]
Combining this inequality with \eqref{eq:elliptic-integral-upper} gives
\[
\log|L(P)|<\log\!\left(\frac{65705}{60}\right)+8.72-0.81M^2<16-0.81M^2.
\]
Thus we have proved, rather than assumed, the estimate
\begin{equation}\label{eq:LP1}
0<|L(P)|<\exp(16-0.81M^2).
\end{equation}

\medskip
\noindent\textbf{Step 4. Lower bound and the initial bound for the Mordell--Weil coefficients.}

We now combine an explicit lower bound for the nonzero linear form with the upper bound from Step~3.  If $M=0$ there is nothing to prove at this stage, so assume $M\ge1$. By \eqref{eq:n0-bound}, we get $|n_0|\le4M+1$. Therefore, on putting $N=\max\{|n_0|,|n_1|,\ldots,|n_4|\}$, we obtain
\begin{equation}\label{eq:N-versus-M}
M\le N\le4M+1\le5M.
\end{equation}

We apply Theorem~5 of Tzanakis~\cite{Tzanakis1996}, which is the explicit consequence of David's theorem~\cite[Theorem~2.1]{David1995} used in the general-cubic method of Stroeker and de Weger~\cite{StroekerDeWeger1999}. There are five algebraic points in the linear form, namely $Q_0,P_1,\ldots,P_4$, so Tzanakis's parameter is $k=5$; the common field degree is $D=[\mathbb Q(a):\mathbb Q]=3$. We take $\mathcal E=e$ and
\[
\mathcal A_0=51,\quad\mathcal A_1=\cdots=\mathcal A_5=34,
\]
where $\mathcal A_0$ is the period parameter and $\mathcal A_1,\ldots,\mathcal A_5$ correspond to the five points. To match Tzanakis's short-Weierstrass normalization, put $Y=y+1/2$; then
\[
Y^2=x^3+Ax+B,\quad A=-10713,\quad B=\frac{2080817}{4}.
\]
With the logarithmic projective height used in \cite[pp.~188--189]{Tzanakis1996}, the curve-height parameter is exactly
\[
\begin{aligned}
h_E
&=\max\left\{1,
h\!\left(-\frac{10713}{4},\frac{2080817}{64}\right),
h(j_E)\right\}\\
&=\log(559565271994368)
=33.958181298\ldots<34.
\end{aligned}
\]
For completeness, we make the period normalization in the admissibility conditions explicit. Let $\omega_1,\omega_2$ be a reduced period basis for the differential $dx/(2Y)$, put $\tau=\omega_2/\omega_1$, and retain $\omega$ for its least positive real period. Tzanakis uses the differential $dx/Y$, whose corresponding periods and elliptic logarithms are $\omega^{\mathrm T}=2\omega$, $\omega_j^{\mathrm T}=2\omega_j$ $(j=1,2)$, and $u^{\mathrm T}(Q)=2u(Q)$. Thus $\tau$ and the normalized elliptic logarithm
$u(Q)/\omega=u^{\mathrm T}(Q)/\omega^{\mathrm T}$
are unchanged, and
\[
\frac{3\pi(\omega^{\mathrm T})^2}{D|\omega_1^{\mathrm T}|^2\Im\tau}
=
\frac{3\pi\omega^2}{D|\omega_1|^2\Im\tau}.
\]
Consequently, the period and point admissibility parameters are identical in the two normalizations, provided the real period and the complex period basis are scaled together. Certificate~08 proves by Arb ball arithmetic that
\[
16.712<
\frac{3\pi\omega^2}{D|\omega_1|^2\Im\tau}
<16.713<51,
\]
For the five points $Q\in\{P_1,\ldots,P_4,Q_0\}$,
Certificate~08 certifies, in the same order, the conservative upper bounds $2.001$, $1.649$, $11.817$, $6.910$, $14.785$,
for the quantities obtained by multiplying this number by $(u(Q)/\omega)^2$; in particular, all five are less than $34$. Equation \eqref{eq:height-difference} gives $\widehat h(P_i)<8.04$ for $1\le i\le4$.

It remains to estimate the height of $Q_0$. Recall that the first coordinate $x(Q_0)$ has minimal polynomial
\[
g(T)=T^3-819T^2+52845T-1191877.
\]
On the unit circle $|T|=1$, we have
\[
|T^3-819T^2+52845T|\le1+819+52845<1191877.
\]
Set $f(T)=-1191877$. Then, on the unit circle $|T|=1$,
\[
|g(T)-f(T)|
=
|T^3-819T^2+52845T|
\le 53665
<|f(T)|.
\]
Rouch\'e's theorem therefore shows that $g$ has no zero in the closed unit disk. Hence all three conjugates of $x(Q_0)$ have absolute value greater than $1$, and consequently
$h(x(Q_0)) = \frac13\log1191877.$
Applying \eqref{eq:height-difference} to $Q_0\in E(\overline{\mathbb Q})$, we obtain
\[
\widehat h(Q_0)
<
\frac16\log1191877+4.36
<6.70.
\]
Thus all the conditions on the $\mathcal A_i$ hold. Certificate~08 also evaluates the six factors in the upper admissibility condition for $\mathcal E$ and proves that each is greater than $1$. Hence $\mathcal E=e$ is admissible. With exactly this normalization, Tzanakis's formula gives
\[
\begin{aligned}
c_4
&=
2.9\cdot10^{42}
3^{14}4^{72}7^{138.3}\cdot51\cdot34^5\\
&=
5.4005580033\ldots\cdot10^{218}
<5.41\cdot10^{218},\\
c_5&=1+\log3<2.099,\\
c_6&=h_E+1+\log3<36.099.
\end{aligned}
\]
The small-$N$ alternative in \cite[Theorem~5]{Tzanakis1996} is
\[
N<
\max\left\{
\exp(eh_E),
|q_i|,
\exp(\mathcal A_i/D):
0\le i\le5
\right\},
\]
where the $q_i$ are the denominators of the coefficients in the linear form. In the present situation all these denominators are equal to $1$, and
\[
\exp(eh_E)<10^{41},
\quad
\exp(\mathcal A_i/D)<10^8
\quad(0\le i\le5).
\]
If the small-$N$ alternative holds, then $M\le N<10^{41}<2\cdot10^{116}$,
so the desired initial bound follows immediately in this case. It therefore remains to
consider the second alternative of \cite[Theorem~5]{Tzanakis1996}.

With the normalization specified above, the linear form to which \cite[Theorem~5]{Tzanakis1996} applies is
\[
\Lambda(P)=2L(P)
=
2u_0-n_0(2\omega)-\sum_{i=1}^4 n_i(2u_i).
\]
The integer coefficients are unchanged by this rescaling, so the same bound $N\le5M$ applies.
In the remaining case, Theorem~5 of Tzanakis~\cite{Tzanakis1996}, together with \eqref{eq:N-versus-M}, gives
\begin{equation}\label{eq:LP2}
2|L(P)|=|\Lambda(P)|
>
\exp\!\left(
-5.41\cdot10^{218}
(\log(5M)+2.099)
(\log\log(5M)+36.099)^7
\right).
\end{equation} 
Combining \eqref{eq:LP1} and \eqref{eq:LP2}, we obtain
\begin{equation}\label{eq:M3}
0.81M^2-16-\log 2
<
5.41\cdot10^{218}
(\log(5M)+2.099)
(\log\log(5M)+36.099)^7.
\end{equation}
Set the initial bound $M_0=2\cdot10^{116}$. Certificate~08 evaluates this ratio using rigorous Arb ball arithmetic
and proves that its exact value lies in the following interval:
\begin{equation}\label{eq:M4}
0.993<\frac{
	5.41\cdot10^{218}
	(\log(5M_0)+2.099)
	(\log\log(5M_0)+36.099)^7
}{
	0.81M_0^2-16-\log 2
}
<0.994.
\end{equation}
Furthermore, for $M\ge M_0$, the logarithmic derivative of the numerator in \eqref{eq:M4}, with respect to $\log M$, is
\[
\frac{1}{\log(5M)+2.099}
+
\frac{7}{
\log(5M)
\bigl(\log\log(5M)+36.099\bigr)
}
<0.005.
\]
Denote the ratio in \eqref{eq:M4} by $R(M_0)$, then \eqref{eq:M4} gives $R(M_0)<0.994<1$. For $M\ge M_0$, the denominator is positive, and \eqref{eq:M3} is precisely $R(M)>1$. Here $R(M)$ denotes the ratio in \eqref{eq:M4} with $M_0$
replaced by $M$. Moreover, for $M\ge M_0$,
\[
\frac{d}{d\log M}\log(0.81M^2-16-\log 2)
=\frac{1.62M^2}{0.81M^2-16-\log 2}>2.
\]
However, for $M\ge M_0$
\[
\frac{d\log R(M)}{d\log M}
=
\frac{d\log(\operatorname{num}R)}{d\log M}
-
\frac{d\log(\operatorname{den}R)}{d\log M}
<0.005-2<0
\]
implies $R(M)$ is strictly decreasing on
$[M_0,\infty)$, and this contradicts $R(M)>1>R(M_0)$.
Consequently, $M<2\cdot10^{116}$.

\medskip
\noindent\textbf{Step 5. Two lattice reductions.}

We now apply LLL lattice reduction~\cite{LenstraLenstraLovasz1982} to sharpen the initial bound on $M=\max_{1\le i\le 4}|n_i|$,
with the aim of reducing it to a range that can be checked exhaustively by exact computation.	
We first record the reduction estimate used in both rounds.  

It is known that $M\le M_0$, and since the real elliptic logarithm is additive modulo the real period
$\omega$, the relation $P=\sum_{i=1}^{4}n_iP_i$ implies that there exists an integer $\mu_0$ such that $u(P)=\mu_0\omega+\sum_{i=1}^4 n_i u_i$.
Moreover, the lift of $u_0-u(P)$ corresponding to the chosen real
branch differs from it by an integral multiple of $\omega$. 
Thus, for some $\ell\in\mathbb Z$, $L(P)=u_0-u(P)+\ell\omega$.  
Putting $n_0=\mu_0-\ell$, we retain the normalization
\[
|n_0|\le4M+1,
\quad
L(P)=u_0-n_0\omega-\sum_{i=1}^4n_i u_i,
\]
where $0<|L(P)|<\exp(16-0.81M^2)$. Let $C>1$, and put
\[
r_i=[Cu_i]\quad(1\le i\le4),
\quad
r_\omega=[C\omega],
\quad
r_0=[Cu_0],
\]
where $[\,\cdot\,]$ denotes nearest-integer rounding.
Let $e_1,\ldots,e_4$ be the standard basis vectors of $\mathbb Z^4$.
Define
\[
\mathcal L_C
=
\left\langle
(e_1,r_1),\ldots,(e_4,r_4),
(0,0,0,0,r_\omega)
\right\rangle_{\mathbb Z}
\subseteq\mathbb Z^5
\]
and let $\mathbf y=(0,0,0,0,r_0)$. Write the Euclidean distance from $\mathbf y$ to the lattice $\mathcal L_C$ as $d_C=\operatorname{dist}(\mathbf y,\mathcal L_C)$. 
Suppose $d_C^2>4M_0^2+(4M_0+1)^2$, and define
\[
\delta_C
:=
\sqrt{d_C^2-4M_0^2}-(4M_0+1)>0.
\]
We consider the lattice vector
\[
\mathbf v
=
\sum_{i=1}^4n_i(e_i,r_i)
+n_0(0,0,0,0,r_\omega)
\in\mathcal L_C.
\]
The first four coordinates of $\mathbf y-\mathbf v$ are $-n_1,-n_2,-n_3,-n_4$. Its final coordinate is $r_0-n_0r_\omega-\sum_{i=1}^4n_i r_i$. Using the definition of $L(P)$, we may write this coordinate as
\[
r_0-n_0r_\omega-\sum_{i=1}^4n_i r_i=CL(P)+(r_0-Cu_0)-n_0(r_\omega-C\omega)-\sum_{i=1}^4n_i(r_i-Cu_i).
\]
Every rounding error on the right-hand side has absolute value at most $1/2$. Moreover, we have $|n_0|\le4M+1$ and $\sum_{i=1}^4|n_i|\le4M$.
It follows that
\[
\begin{aligned}
\left|
r_0-n_0r_\omega-\sum_{i=1}^4n_i r_i
\right|
&\le
C|L(P)|
+\frac12
+\frac{|n_0|}{2}
+\frac12\sum_{i=1}^4|n_i|\\
&<
C\exp(16-0.81M^2)+4M+1.
\end{aligned}
\]
Consequently, we get
\[
\begin{aligned}
d_C^2
&\le\|\mathbf y-\mathbf v\|^2\\
&<
\sum_{i=1}^4n_i^2
+
\left(
C\exp(16-0.81M^2)+4M+1
\right)^2\\
&\le
4M_0^2
+
\left(
C\exp(16-0.81M^2)+4M_0+1
\right)^2.
\end{aligned}
\]
The assumed inequality $d_C^2>4M_0^2+(4M_0+1)^2$ ensures that $d_C^2-4M_0^2>0$. Therefore,
\[
\sqrt{d_C^2-4M_0^2}
<
C\exp(16-0.81M^2)+4M_0+1.
\]
Subtracting $4M_0+1$, taking logarithms and rearranging yields
\[
M^2< \frac{\log C+16-\log\delta_C}{0.81}.
\]

For each $z\in\{u_1,u_2,u_3,u_4,\omega,u_0\}$ and each $C\in\{10^{14},10^{590}\}$, Certificate~07 constructs a certified interval $I_z$ containing $z$ and verifies that
\[
CI_z\subset
\left(r_z-\frac12,r_z+\frac12\right),
\quad r_z=[Cz].
\]
Thus every rounded integer used to construct the two lattices is uniquely determined. Once these integers have been fixed, the lattice bases and target vectors are exact integer data; Certificates~03 and~04 perform the subsequent closest-vector computations exactly. 

We now apply this reduction estimate twice. For the first reduction, take $M_0=2\cdot10^{116}$ and $C=10^{590}$. The exact closest-vector computation recorded in Appendix~\ref{app:elliptic-certificate} gives
$d_1^2>20M_0^2+8M_0+1.$
Certificate~03 verifies by exact integer arithmetic that
\[
d_1^2-4M_0^2>
\left(7\cdot10^{117}+4M_0+1\right)^2.
\]
Hence $\delta_1>7\cdot10^{117}$. By the reduction estimate above,
\[
M^2
<\frac{\log(10^{590})+16-\log\delta_1}{0.81}
<\frac{473\log10+16-\log7}{0.81}
<37^2.
\]
Therefore $M<37$, and hence $M\le36$.

For the second reduction, take $M_0=36$ and $C=10^{14}$. Certificate~07 proves that the six integers used to construct this lattice are the unique nearest-integer roundings at this scaling. With these exact integer data, Certificate~04 applies Magma's exact \texttt{ClosestVectors} routine and gives $d_2^2=116890$, which satisfies the hypothesis of the reduction estimate. 
Moreover, it yields the rigorous lower bound $\delta_2>189$.
By the reduction estimate above and $\log10<2.303$,
\[
M^2
<\frac{14\log10+16-\log\delta_2}{0.81}
<\frac{14\cdot2.303+16}{0.81}
<64.
\]
Thus $M<8$, and hence $M\le7$. It is therefore sufficient to examine only $(n_1,n_2,n_3,n_4)\in[-7,7]^4$.

\medskip
\noindent\textbf{Step 6. Exact finite enumeration.}

The exact enumeration consists of $15^4=50625$ coefficient vectors. For each vector, the program computes $P=n_1P_1+\cdots+n_4P_4$ using exact rational arithmetic, applies the inverse birational map, and checks the integrality and range conditions on $\mu$ and $s$.  The zero vector gives the point at infinity $O$ and is handled separately, since the affine inverse formulas do not apply to $O$. The points with $D(x,y)=0$ are excluded from the relevant image in the final enumeration paragraph of Appendix~\ref{subsec:elliptic-computational-certificate}.  No vector produces an integral pair satisfying $F(\mu,s)=0$, for $\mu\ge7$ and $6\mu<s<7\mu$. The lattice data, closest-vector calculations, inverse birational map, and complete enumeration certificate are given in Appendix~\ref{app:elliptic-certificate}.

This contradicts the existence of the assumed integral solution. Hence $F(\mu,s)=0$ has no integral solution with $\mu\ge2$ and $s\ge1$.

\end{proof}

\subsection{Computational certificate}
\label{app:elliptic-certificate}

The supplementary archive contains six self-contained Magma scripts~\cite{BosmaCannonPlayoust1997}, three Python certificates, and the corresponding output logs described in the accompanying README. Certificates~02 and~03 were rerun from the revised Magma sources, and their complete current logs are included at the root of the archive. An independent Python/Arb revision verifier is also included as a supplementary cross-check. The certificate scripts are:
\begin{enumerate}
\item \path{01A_mordell_weil_group_only.m}, which verifies the triviality of the rational torsion subgroup and computes the full Mordell--Weil group;

\item \path{01B_ultra_short_exact_check.m}, which verifies by exact rational elliptic-curve arithmetic that the four points stated in the proof form the same Mordell--Weil basis as the generators returned in the accompanying Certificate~01A run;

\item \path{02_maps_height_david.m}, which checks the birational maps and the differential identity exactly, and independently cross-checks the numerical parameters entering the David bound;

\item \path{03_first_lll.m}, which verifies the first elliptic-logarithm lattice reduction and proves $M\leq36$;

\item \path{04_second_lll.m}, which verifies the second elliptic-logarithm lattice reduction and proves $M\leq7$;

\item \path{05_exact_enumeration.m}, which enumerates all $15^4=50625$ remaining coefficient vectors using exact rational arithmetic.

\item \path{06_height_matrix_certificate.py} (Certificate~06), which uses exact integer and rational arithmetic, together with rigorous rational logarithm
bounds, to prove
\[
\widehat h(n_1P_1+\cdots+n_4P_4)
\ge
0.405\max_i|n_i|^2.
\]
\item \path{07_real_period_elliptic_logs_certificate.py}, which uses Arb ball arithmetic to enclose the real period and the five real elliptic logarithms, and rigorously certifies every rounding used in Certificates~03 and~04;

\item \path{08_david_parameters_certificate.py}, which uses exact algebraic input and Arb ball arithmetic to certify all admissibility inequalities and numerical constants entering the David--Tzanakis initial bound.
\end{enumerate}

Certificates~01A and~01B together verify the Mordell--Weil assertion used in the proof.  Certificate~01A proves that $E(\mathbb Q)\cong\mathbb Z^4$ and that the rational torsion subgroup is trivial; the rank and saturation flags returned by Magma are both true.  If $G_1,\ldots,G_4$ denote the generators printed in the corresponding accompanying Certificate~01A log, then Certificate~01B verifies the exact relations
\[
\begin{pmatrix}
G_1\\G_2\\G_3\\G_4
\end{pmatrix}
=
\begin{pmatrix}
0&-1&1&1\\
0&0&0&-1\\
1&0&-1&0\\
0&1&0&0
\end{pmatrix}
\begin{pmatrix}
P_1\\P_2\\P_3\\P_4
\end{pmatrix}.
\]
The change-of-basis matrix has determinant $-1$.  Hence
$P_1,\ldots,P_4$ form a Mordell--Weil basis of $E(\mathbb Q)$.

The lattice reductions use two bounds established earlier in the proof of
Lemma~\ref{lem:cubic-obstruction}: the estimate
$0<|L(P)|<\exp(16-0.81M^2)$ in~\eqref{eq:LP1} and the initial bound
$M<2\cdot10^{116}$ established in Step~4 of the proof.  These are intermediate estimates in
the proof, and the reductions do not assume the conclusion of the lemma.

\subsection{Verification data and lattice reductions}
\label{subsec:elliptic-computational-certificate}

We record the exact and certified computations used in the proof of
Lemma~\ref{lem:cubic-obstruction}. Certificate~02 verifies the rational-function
identities in \eqref{eq:forward-birational-map}--\eqref{eq:inverse-birational-map},
the coordinates of $Q_0$ in \eqref{eq:def-Q0}, and the differential identity
\eqref{eq:differential-pullback}. Certificates~01A and~01B verify the
Mordell--Weil basis. Certificates~06--08 supply the rigorous height, period,
elliptic-logarithm, and David--Tzanakis estimates, while Certificates~03--05
perform exact integer or rational computations.

All rational-function and elliptic-curve identities used in this appendix are verified exactly. The rigorous lower bound for the N\'eron--Tate height pairing is supplied separately by Certificate~06.

The real period and the five real elliptic logarithms used in the two lattice reductions are enclosed by Certificate~07 using Arb midpoint-radius ball arithmetic. The input polynomials and rational points are exact, and all polynomial-root isolation, Carlson-integral evaluation, and subsequent arithmetic operations are performed with directed error bounds. Certificate~07 then verifies explicit interval containments fixing all nearest-integer roundings uniquely.

Once these rounded integers have been fixed, the lattice bases, target vectors, closest-vector calculations, and squared-distance calculations in Certificates~03 and~04 are exact over $\mathbb Z$. The parameters entering the David--Tzanakis estimate are rigorously enclosed by Certificate~08. The corresponding finite-precision calculations in Certificate~02 are retained only as an independent numerical cross-check.

\paragraph{Certified period and elliptic logarithms.}
Put
\[
Y=y+\frac12,
\quad
f(x)=x^3-10713x+\frac{2080817}{4}.
\]
Let $\xi$ be the unique real zero of $f$, and let
$\zeta,\overline\zeta$ be its two nonreal zeros. Define $I(x) = R_F\bigl(x-\xi,x-\zeta,x-\overline\zeta\bigr),$ where $R_F$ denotes Carlson's symmetric elliptic
integral~\cite{Carlson1995}. Equivalently,
$I(x)=\int_x^\infty\frac{dt}{2\sqrt{f(t)}}$ for $x>\xi$.
The real period is
\[
\omega=2R_F\bigl(0,\xi-\zeta,\xi-\overline\zeta\bigr)
=\int_\xi^\infty\frac{dt}{\sqrt{f(t)}}.
\]

For a real affine point $R=(x(R),y(R))$ with $y(R)+\tfrac12\ne0$,
the real elliptic logarithm, normalized to $u(R)\in[0,\omega)$, is given by
\[
u(R)
=
\begin{cases}
I(x(R)),
& y(R)+\dfrac12<0,\\[6pt]
\omega-I(x(R)),
& y(R)+\dfrac12>0.
\end{cases}
\]
In particular, this formula gives $u_1,\ldots,u_4$.  The point $Q_0$
lies on the upper branch, and hence $u_0=\omega-I(x(Q_0)).$

Certificate~07 starts with the exact rational polynomial $f$ and the
exact minimal polynomial $a^3-6a^2-6a-\frac32$ of the number $a$ occurring in $Q_0$. It isolates all their roots in
certified complex balls. Before evaluating $u_0$, it also verifies the
curve equation for $Q_0$ exactly by clearing denominators and reducing
modulo the minimal polynomial of $a$.

The Carlson integrals and all subsequent operations are evaluated at $750$ decimal digits using Arb ball arithmetic~\cite{Johansson2017}. The resulting certified balls have centers beginning
\[
\begin{aligned}
\omega&=0.615134685591167604099518514350\ldots,\\
u_1&=0.212802433486132940022858161049\ldots,\\
u_2&=0.193174049321678622005198210391\ldots,\\
u_3&=0.517246759484073082839728624775\ldots,\\
u_4&=0.395521315488017005327971293202\ldots,\\
u_0&=0.578570757851769889192964515668\ldots.
\end{aligned}
\]
\paragraph{Certified rounding containments.}
For $\mathcal Z=\{u_1,u_2,u_3,u_4,\omega,u_0\},$ Certificate~07 constructs, for every $z\in\mathcal Z$, an Arb ball $I_z$ containing $z$. If $r_z=[Cz]$ denotes nearest-integer rounding, then for each of the two scales $C=10^{14}$ and $C=10^{590}$ the script rigorously verifies that the distance of $CI_z$ from either half-integer boundary is greater than $10^{-2}$. Equivalently,
\begin{equation}\label{eq:certified-rounding-containment}
CI_z\subset
\left(
r_z-\frac12+\frac1{100},
r_z+\frac12-\frac1{100}
\right)
\subset
\left(r_z-\frac12,r_z+\frac12\right).
\end{equation}
For the second reduction, $C=10^{14}$, the certified rounded integers are
\[
\begin{aligned}
(r_{u_1},r_{u_2},r_{u_3},r_{u_4},r_{\omega},r_{u_0})
={}&(21280243348613,19317404932168,51724675948407,\\
&\quad 39552131548802,61513468559117,57857075785177).
\end{aligned}
\]
For the first reduction, $C=10^{590}$, the six exact $590$-digit rounded integers are printed in full in the Certificate~07 log and are asserted there to agree entry by entry with the integers used in Certificate~03. Thus all twelve nearest-integer roundings used in the two lattice constructions are uniquely determined by certified interval arithmetic.

\paragraph{Rigorous David--Tzanakis parameters.}
Certificate~08 gives a self-contained verification of all transcendental
inequalities entering the David--Tzanakis initial bound. It uses Arb ball
arithmetic with directed rounding, isolates the roots of the exact curve polynomial,
evaluates the AGM period formulae and the five Carlson elliptic integrals,
and proves the period and point admissibility inequalities for
$\mathcal A_0=51$, $\mathcal A_i=34$, and $\mathcal E=e$. Exact projective
height data together with the explicit one-sided Silverman estimate verify
the required canonical-height upper bounds. Finally, direct Arb evaluations
prove the stated bounds for $c_4,c_5,c_6$, all branches of the small-$N$
alternative, and the endpoint and derivative inequalities at
$M_0=2\cdot10^{116}$. Certificate~02 remains the exact Magma verification of
the birational maps and differential identity; its floating-point David
output is used only as an independent numerical cross-check.

\paragraph{First reduction.}
Take $M_0=2\cdot10^{116}$ and $C=10^{590}$. For $r_i=[Cu_i]$, $r_\omega=[C\omega]$, and $r_0=[Cu_0]$, let $\mathcal L$
be the lattice generated by the rows of
\[
B=
\begin{pmatrix}
1&0&0&0&r_1\\
0&1&0&0&r_2\\
0&0&1&0&r_3\\
0&0&0&1&r_4\\
0&0&0&0&r_\omega
\end{pmatrix}.
\]
Set $\mathbf y=(0,0,0,0,r_0)$.
Exact closest-vector computation gives a unique closest vector and squared
distance
\[
\begin{split}
d_1^2={}&
68458651583052304198049256658341101449164082617855023404034\\
&83146299452391527572067966648905463911991889932165245135738\\
&64731017937911732116526580654560264031320139117691433735972\\
&19754663454588548340762869144210215948097731569879560663200.
\end{split}
\]
In particular, Certificate~03 verifies the stronger exact inequality
\[
d_1^2-4M_0^2>\left(7\cdot10^{117}+4M_0+1\right)^2.
\]
Putting $\delta_1 = \sqrt{d_1^2-4M_0^2}-(4M_0+1),$ the displayed exact inequality gives $\delta_1>7\cdot10^{117}$.
Moreover, Certificate~03 verifies the exact comparison used in the manuscript. Since $C=10^{590}$, $\delta_1>7\cdot10^{117}$, $\log 10<2.303$, and $\log 7>1.94$, we obtain
\[
\log C+16-\log\delta_1
<473\cdot2.303+16-1.94
=1103.379
<0.81\cdot37^2.
\]
Hence $M<37$, and therefore $M\le36$.

Certificate~07 proves that the six integers defining the displayed lattice and target are the unique nearest-integer roundings of the corresponding scaled elliptic logarithms. Certificate~03 then applies Magma's exact \texttt{ClosestVectors} routine to these integral data. The script asserts that the returned list contains exactly one closest vector and that its squared distance is the stated integer $d_1^2$.

\paragraph{Second reduction.}
Now take $M_0=36$ and $C=10^{14}$. The rounded integers are
\[
\begin{aligned}
(r_1,r_2,r_3,r_4,r_\omega,r_0)
={}&(21280243348613,19317404932168,51724675948407,\\
&\quad 39552131548802,61513468559117,57857075785177).
\end{aligned}
\]
Relative to the original row basis, the unique closest lattice vector has coefficient vector $(151,114,-212,190,-31)$. Its displacement from the target is $(151,114,-212,190,7)$, and therefore
\[
d_2^2
=151^2+114^2+(-212)^2+190^2+7^2
=116890.
\]
Consequently,
\[
\delta_2
=
\sqrt{116890-4\cdot36^2}-(4\cdot36+1)
=
\sqrt{111706}-145.
\]
Since $111706>334^2$, we have $\delta_2>189>1$. Therefore, using only $\log 10<2.303$,
\[
M^2
<\frac{\log(10^{14})+16-\log\delta_2}{0.81}
<\frac{14\cdot2.303+16}{0.81}
<64.
\]
Hence $M<8$, and therefore $M\le7$.

Certificate~07 proves that the six displayed integers are the unique nearest-integer roundings for $C=10^{14}$. In Certificate~04, the resulting lattice basis and target vector are therefore exact integral data. Magma's exact \texttt{ClosestVectors} routine returns a list containing exactly one closest vector, and the script asserts $\#\mathrm{CV}=1, \quad d_2^2=116890.$ The script also recomputes, using exact integer arithmetic, the displacement $(151,114,-212,190,7)$ and verifies that its squared norm is \(116890\). Thus the value and uniqueness of the closest vector are explicitly checked by Certificate~04.

\paragraph{Final exact enumeration.}
Every remaining point has the form
\[
P=n_1P_1+n_2P_2+n_3P_3+n_4P_4,
\quad -7\leq n_i\leq7.
\]
The inverse rational functions on $E$ are
\[
D(x,y)=288x^2+5xy-95951x-8507y+6618599,
\] $\mu=\frac{(x-5866)(25x-889)}{D(x,y)},$ and
\[
s=-\frac{3(373x^2+5xy-62303x-1566y+2338958)}{D(x,y)}.
\]
All $15^4=50625$ coefficient vectors were enumerated using exact arithmetic in $\mathbb Q$.  The zero vector gives $O\in E(\mathbb Q)$ and is handled separately.  Among the remaining affine points there were two ordinary integral inverse images,	$(\mu,s)=(1,-2)$, $(\mu,s)=(0,-4)$, and two points with $D(x,y)=0$. None of the ordinary inverse images
satisfies $\mu\ge7$ and $6\mu<s<7\mu$. The points with $D(x,y)=0$ cannot arise from
a solution in the required range either. Indeed, the denominator
$51\mu+5s+15$ in the forward formula for $x$ is positive there, and
\[
\begin{aligned}
&5484\mu+8507s+4698-600(51\mu+5s+15)\\
&\qquad=-25116\mu+5507s-4302
>7926\mu-4302>0,
\end{aligned}
\]
where we used $s>6\mu$ and $\mu\ge7$. Hence $x>600$. Similarly,
\[
\begin{aligned}
&1000(51\mu+5s+15)-(5484\mu+8507s+4698)\\
&\qquad=45516\mu-3507s+10302
>20967\mu+10302>0,
\end{aligned}
\]
using $s<7\mu$, and therefore $x<1000$. Thus $600<x<1000$ throughout
the relevant image. On the other hand, the inverse identity implies
\[
\mu D(x,y)=(x-5866)(25x-889).
\]
Since neither $5866$ nor $889/25$ lies in $(600,1000)$, one necessarily has
$D(x,y)\ne0$ on the relevant image.

It follows that the original cubic equation has no integral solution in the required range.

\begin{rem}
Each script terminates with an explicit verification message. The complete
scripts, raw output logs, software versions, SHA-256 checksums, and reproduction
instructions are included in the supplementary archive. The generators returned
by \texttt{MordellWeilGroup} may depend on the Magma random seed, so
Certificate~01B is tied to the generators printed in the accompanying
Certificate~01A log. Certificate~05 is an exhaustive exact enumeration and uses
no random or probabilistic search. No call to
\texttt{SetClassGroupBounds("GRH")} is made; hence no GRH-dependent class-group
bounds are used.
\end{rem}
\section*{Declaration of generative AI and AI-assisted technologies in the manuscript preparation process}
During the preparation of this work, the authors used ChatGPT
(OpenAI) to obtain feedback on the exposition and to assist with
exploratory computations and language revision. No AI-generated
output was treated as mathematical evidence. The authors
independently verified all mathematical arguments and computations,
reviewed and revised all AI-assisted content, and take full
responsibility for the manuscript.
%%%%%%%%%%%%%%%%%%%%%%%%%%%%%%%%%%%%%%%%%%%%%%%%%%%%%%%%%%%%%% references

\end{document}